\documentclass[11pt,reqno]{amsart}
\usepackage{hyperref} 
\usepackage{amssymb,epsfig,graphics}
\usepackage{amsfonts}
\usepackage{amsmath}
\usepackage{graphicx}
\usepackage{amsfonts} \usepackage{amsmath, amsthm}
\usepackage{epsfig} \usepackage{verbatim}
\usepackage{amssymb, amsthm, amsmath, bm, tikz, enumerate, color}
\usepackage{ bm, tikz, enumerate}
\usepackage{amsfonts}
\usepackage{amsmath}
\usepackage{graphicx}
\usepackage{amsfonts} \usepackage{amsmath, amsthm}
\usepackage{epsfig} \usepackage{verbatim}
\usepackage{cancel}
\usepackage{cite, comment}
\usepackage{xcolor}
\usepackage{tikz-cd} 
\usetikzlibrary{spy}
\usepackage{comment,xcolor}

\newtheorem{theorem}{Theorem}[section]

\newtheorem{lemma}[theorem]{Lemma}
\newtheorem{sub-lemma}[theorem]{Sub-Lemma}

\let\eps=\varepsilon

\def\fp{{\mathfrak{p}}}

\def\1{{{\mathit 1} \!\!\>\!\! I} }
\def\DS{\displaystyle}

\numberwithin{equation}{section}

\begin{document}
\title[Correlation functions for $\mathbb{Z}^d$ covers of hyperbolic flows]
{Asymptotic expansion of correlation functions for $\mathbb{Z}^d$ covers of hyperbolic flows.}

\author{Dmitry Dolgopyat}
\address{Department of Mathematics, University of Maryland, College Park, MD 20741, USA}
\email{dmitry@math.umd.edu}
\author{P\'eter N\'andori}
\address{Department of Mathematical Sciences, Yeshiva University, New York, NY, 10016, USA}
\email{peter.nandori@yu.edu}
\author{Fran\c{c}oise P\`ene}
\address{Univ Brest, Universit\'e de Brest, Institut Universitaire de France, IUF, UMR CNRS 6205,  Laboratoire de
Math\'ematiques de Bretagne Atlantique, LMBA, France}
\email{francoise.pene@univ-brest.fr}

\begin{abstract}
We establish expansion of every order for the correlation function of sufficiently regular
observables of $\mathbb Z^d$ extensions of some hyperbolic flows. Our examples include 
the $\mathbb Z^2$
periodic Lorentz gas
and geodesic flows on abelian covers of compact manifolds with negative curvature.
\end{abstract}

\date{\today}
\maketitle

\section{Introduction}

\subsection{Setup.}
\label{sec:abs}
Let $(M,\nu,T)$ be a probability preserving dynamical system.
Consider $(\tilde M,\tilde\nu,\tilde T)$--the $\mathbb Z^d$-extension of 
$(M,\nu,T)$ by $\kappa:M\rightarrow \mathbb Z^d$ for a positive integer $d$.
Let $(\Phi_t)_{t\ge 0}$ be
the suspension semiflow  over $( M, \nu, T)$
with roof function 
$\tau:M\rightarrow (0,+\infty)$
and let 
$(\widetilde\Phi_t)_{t\ge 0}$ be the corresponding $\mathbb{Z}^d$ cover.
That is, $(\widetilde \Phi_t)_{t\geq 0}$ is
the semi-flow
defined on
$$\widetilde\Omega:=\{(x,\ell,s)\in M\times\mathbb Z^d\times[0,+\infty)\, :\, s\in[0,\tau(x))\}$$
such that $\widetilde\Phi_t(x,\ell,s)$ corresponds to $(x,\ell,s+t)$ by identifying
$(x,\ell,s)$ with $(Tx,\ell+\kappa(x),s-\tau(x)).$ This semi-flow  
preserves the restriction $\tilde\mu$ on $\widetilde\Omega$ of the product measure $\nu\otimes \mathfrak m\otimes \mathfrak l$, where
$\mathfrak m$ is the counting measure on $\mathbb Z^d$
and $\mathfrak l$ is the Lebesgue measure on $[0,+\infty)$.

In the present paper we study
the following correlation functions
\begin{equation*}
\label{defC}
C_t(f,g):=\int_{\widetilde\Omega}f.g\circ \widetilde\Phi_t \, d\tilde\mu\, ,
\end{equation*}
as $t$ goes to infinity, for suitable observables $f,g$.
Our goal is 
to establish expansions of the form
\begin{equation}
\label{KTermExp}
 C_t(f,g)=\sum_{k=0}^{K}C_k(f,g)\, t^{-\frac d2-k}+o(t^{-\frac d2-K})\, .
\end{equation}
More precisely we assume that $\Phi_t$ is
$C^\infty$ away from singularities, which is a finite (possibly empty) union of positive 
codimension submanifolds. We say that $\tilde\Phi_t$ 
{\em admits a complete asymptotic expansion in inverse powers of $t$} 
if for $f$ and $g$
which are $C^\infty$ and have compact support which is disjoint from the singularities of $\tilde\Phi$,
the correlation function $C_t(f,g)$ admits the expansion \eqref{KTermExp} for each $K\in \mathbb{N}.$
In this paper we establish a complete asymptotic expansion in inverse powers of $t$ for two classical
examples of hyperbolic systems: Lorentz gas and geodesic flows on abelian covers of negatively
curved manifolds. In fact, our results are more general. Namely, 
\begin{itemize}
\item we consider an abstract setup potentially applicable to other hyperbolic flows; 
\item we allow the support of $\mathfrak f$ and $\mathfrak g$ to be unbounded (provided they decay sufficiently fast);
\item 
we allow $\mathfrak f$ and $ \mathfrak g$ to take non-zero values on 
the singularities of the flow. In addition, we
allow them to be only H\"older continuous (note that continuity is required in the flow direction as well)
with one of them being $C^{\infty}$ in the flow direction.
\end{itemize}


\subsection{Related results} 
The correlation function \eqref{defC} has been studied by several authors.
The leading term ($K=0$) for hyperbolic maps 
(for functions of non-zero integral) is sometimes called mixing, Krickeberg mixing
or local mixing. In case of $\mathbb Z^d$ extensions as above, 
it is a consequence of some versions of the local limit theorem. See related results
in e.g. \cite{AD01, G11, G89, GH88, SV04}. 
Less is known about higher order expansions for maps, but see the recent results in \cite{Soazmixing}.
For flows, the leading term has been studied in e.g. \cite{AN17, DN1, I08, T19}.
We also mention that there are other quantities besides the correlation functions whose
 asymptotic expansions are of interest. In particular, the asymptotic 
 expansions have been obtained
(using techniques similar to ones employed in the present paper)
for the rate of convergence in the central limit theorem 
\cite{FL18} and for the number of periodic orbits in a given homology class \cite{Mar, PS}.

There are several other results for some hyperbolic systems preserving an infinite measure 
which may not be a $\mathbb Z^d$ cover and so the powers may be different from 
$-\frac d2-k$.
See the leading term in e.g. \cite{DN2, MT18, OP17} 
and expansions in e.g. \cite{LT16, MT13, MT17}. 
We note that the expansions in the above papers are of the form
$\DS \phi(t) \tilde\mu(f) \tilde\mu(g)$ where $\phi(t)$ admits an expansion of the form
$\DS \phi(t)=\sum_{k=1}^K a_k t^{-\beta_k} +o\left(t^{-\beta_K}\right).$
Thus these expansions do not give the leading term in the case where 
$\tilde\mu(f)\tilde\mu(g)=0$ and they are not suitable for studying the limiting behavior of 
ergodic sums of zero mean functions. In contrast, our expansion 
provides the leading term for many observables of zero mean.

\subsection{Layout of the paper.}
The rest of the paper is organized as follows. In Section 
\ref{sec:res}, we present some abstract results on expansion
of correlation functions for general suspension semiflows and flows. 
Theorems \ref{THMGENE} and \ref{THMGENE2} guarantee that under a list of technical assumptions,
expansions of the kind \eqref{KTermExp} hold. The results are proved by a careful study of the twisted transfer operator. 
One major difference from the case of maps
(cf. \cite{Soazmixing}) is the extra assumption \eqref{DMbound1} (along the lines of \cite{Dima98}).
In Section \ref{sec:mixbilliard} we study billiards and
 verify the abstract assumptions of Theorem \ref{THMGENE2}
for the Lorentz gas obtaining a complete asymptotic expansions in inverse powers of $t$ for that system.
In Section \ref{sec:geo}, we verify the abstract assumptions for geodesic flows
on $\mathbb Z^d$ covers of compact negatively curved Riemannian manifolds.
Some technical computations are presented in the Appendix.

\section{Abstract results.}
\label{sec:res}


\label{higherorder}
\subsection{Notations}
\label{sec:notations}
We will work with symmetric multilinear forms.
Let $\mathfrak{S}_{m}$ be the set of permutations of
$\{1,...,m\}$.
We identify the set of symmetric $m$-linear forms on
$\mathbb C^{d+1}$ with 
$$\mathcal S_m:=\left\{A=(A_{i_1,...,i_m})_{(i_1,...,i_m)
}\in\mathbb C^{\{1,...,d+1\}^{m}}\ :\ \forall i_1,...,i_m,\ \ \forall \mathfrak s\in\mathfrak S_m,\ A_{i_{\mathfrak s (1)},...,i_{\mathfrak s( m)}}=A_{i_1,...,i_m}
     \right\}\, .$$
For any
$A\in \mathcal S_m$ and
$B\in\mathcal S_k$, we define $A\otimes B$ as the element
$C$ of $\mathcal S_{m + k}$
such that
$$\forall i_1,...,i_{m+k}\in\{1,...,d+1\},\quad C_{i_1,...,i_{m+k}}=\frac 1{(m+k)!}\sum_{\mathfrak{s}\in\mathfrak{S}_{m+k}}A_{i_{\mathfrak{s(1)}},...,i_{\mathfrak{s}(m)}}
B_{i_{\mathfrak{s}(m+1)}...,i_{\mathfrak{s}(m+k)}} \, .$$
Note that $\otimes$ is associative and commutative.
For any
$A\in\mathcal S_m$ and
$B\in\mathcal S_k$ with $k\le m$, we define $A* B$ as the element
$C\in\mathcal S_{m-k}$ such that
$$\forall i_1,...,i_{m-k}\in\{1,...,d+1\},\quad C_{i_1,,...,i_{m-k}}=\sum_{i_{m-k+1},...,i_m\in\{1,...,d+1\}}A_{i_1,...,i_{m}}
B_{i_{m-k+1},...,i_{m}} .$$
Note that when $k=m=1$, $A*B$ is simply the scalar product $A.B$.
For any $C^m$-smooth function $F:\mathbb C^{d+1}\rightarrow\mathbb C$, we write $F^{(m)}$ for its
differential of order $m$, which is identified with a $m$-linear form on $\mathbb C^{d+1}$. 
We write $A^{\otimes k}$ for the product $A\otimes...\otimes A$.
With these notations, Taylor expansions of $F$ at $0$ are simply written
$$\sum_{k=0}^m \frac 1{k!}F^{(k)}(0)*x^{\otimes k}\,  .$$
It is also worth noting that $A*(B\otimes C)=(A*B)*C$, for every $A\in\mathcal S_m$, $B\in\mathcal S_k$ and $C\in\mathcal S_\ell$ with $m\ge k+\ell$.

For any $\nu\otimes\mathfrak l$-integrable function $h_0:M\times\mathbb R\rightarrow \mathbb C$, we set
$$\hat h_0(x,\xi):=\int_{\mathbb R}e^{i\xi s}h_0(x,s)\, ds \, ,$$
(this quantity is well defined for $\nu$-a.e. $x$).

Notations $\lambda_0^{(k)}$, $a_0^{(k)}$,
$\Pi_0^{(k)}$ stand for the $k$-th derivatives of $\lambda$, $a$ and $\Pi$ at 0.\\
We write $P$ for the Perron-Frobenius operator of $T$ with respect to $\nu$, 
which is defined by:
\begin{equation}
\label{PerronFr}
\forall f,g\in L^2(\nu),\quad \int_M Pf.g\, d\nu=\int_M f. g\circ T\, d\nu.
\end{equation}
We also consider the family $(P_{\theta,\xi})_{\theta\in[-\pi,\pi]^d,\xi\in\mathbb R}$ of operators given by
\begin{equation}
\label{PerronFrTw}
P_{\theta,\xi}(f):=P\left( e^{i\, \theta\cdot\kappa}e^{i\, \xi\tau}f\right) \, .  
\end{equation}
To simplify notations, we write $\nu(h):=\int_Mh\, d\nu$.

Let $\Sigma$ be a $(d+1)$-dimensional positive symmetric matrix.
We will denote by $\Psi  = \Psi_{\Sigma}$ the $(d+1)$-dimensional centered Gaussian density with covariance matrix $\Sigma$:
\begin{equation}
\label{defpsi}
\Psi(s) = \Psi_{\Sigma}(s):= \frac {e^{-\frac 12\Sigma^{-1}*s^{\otimes 2}}}
             {(2\pi)^{
\frac {d+1}2
}\sqrt{\det \Sigma}}\, .
\end{equation}
In particular, $\Psi^{(k)}$ is the differential of $\Psi$ of order $k$. 
Let 
\begin{equation}
\label{Psi-FT}
a_s:=e^{-\frac 12\Sigma*s^{\otimes 2}}
\end{equation}
 be the Fourier transform of $\Psi.$
Given a non-negative
integer $\alpha$ and a real number $\gamma$, we define
\begin{equation}
\label{defh}
h_{\alpha, \gamma}: \mathbb R^2 \rightarrow \mathcal{S}_{m}, \quad 
h_{\alpha, \gamma}(s,z) = z^{\gamma} \Psi^{(\alpha)} \left(\boldsymbol{0}, s/ \sqrt{z/\nu (\tau)} \right)
\end{equation}
where $\boldsymbol{0}$ denotes the origin in $\mathbb{R}^d.$

We will use the notations
$$\kappa_n:=\sum_{k=0}^{n-1}\kappa\circ T^k\quad\mbox{ and }\quad
\tau_n:=\sum_{k=0}^{n-1}\tau\circ T^k\, .$$
Note that with this notation, we have
$$\widetilde\Phi_t(x,\ell,s)=\left(T^n x, \ell+\kappa_n(x),s+t-\tau_n(x)\right) ,\quad\mbox{ with }n\ \mbox{ s.t. }\ 
\ \tau_n(x)\le s+t<\tau_{n+1}(x)\, .$$

It will be also useful to consider the suspension flow 
$(\Phi_t)_{t\ge 0}$ over $( M, \nu, T)$
with roof function $ \tau$ which is defined on 
$\Omega:=\{(x,s)\in M\times[0,+\infty)\, :\, s\in[0,\tau(x))\} $
and preserves the measure $\mu$ which is the restriction of the product measure $\nu
\otimes
\mathfrak l$ to $\Omega$.
Note that $\mu$ is a finite measure but not necessarily a probability measure. 

%
%
%
%
\subsection{A general result under spectral assumptions}

\begin{theorem}\label{THMGENE}
Assume $\tau$ uniformly bounded from above and below.
Let $\Sigma$ be a $(d+1)$-dimensional positive symmetric matrix.
Let $K$ and $J$ be two positive integers such that $ 3\le J\le K+3$.
Let $\mathcal B$ be a Banach space of complex valued functions
$f:M\rightarrow \mathbb C$ such that $\mathcal B\hookrightarrow  L^1 ( M,  \nu)$ and 
$\mathbf{1}_{ M} \in \mathcal B$.
Assume that 
$(P_{\theta,\xi})_{\theta\in[-\pi,\pi]^d,\xi\in \mathbb R}$ is a family of linear continuous operators on $\mathcal B$
such that
there exist constants $b\in(0,\pi]$, $C>0$, $\vartheta \in(0,1)$, $\beta>0$ and three functions $\lambda_\cdot:[-b,b]^{d+1} \to \mathbb C$ (assumed to be $C^{K+3}$-smooth) and
$\Pi_\cdot,R_\cdot:[-b,b]^{d+1} \to \mathcal L(\mathcal B,\mathcal B)$ (assumed to be $C^{
K+1
}$-smooth) such that $\Pi_0=\mathbb E_{\nu} [\cdot]\mathbf 1_{
 M}\, ,$ and
$\tilde\lambda_{\theta,\xi}:= \lambda_{\theta,\xi} e^{-i\xi\nu(\tau)}$ satisfies
\begin{equation}\label{LambdaDer} \forall k<J,\quad \tilde\lambda^{(k)}_0=a^{(k)}_0,
\end{equation}
 and, in $\mathcal L(\mathcal B,\mathcal B)$,
\begin{equation}\label{decomp2}
\forall s\in[-b,b]^{d+1},\quad P_s =\lambda_s\Pi_s+R_s,\quad 
\Pi_s R_s      = R_s \Pi_s = 0,\quad
\Pi_s^{2}  = \Pi_{s}\, ,
\end{equation}
\begin{equation}\label{majoexpo}
\sup_{s\in [-b,b]^{d+1}} \Vert{R_s^k}\Vert_{\mathcal L(\mathcal B,\mathcal B)}+\sup_{\theta\in [-\pi,\pi]^d\setminus[-b,b]^d,\ |\xi|\le b} \Vert{P_{\theta,\xi}^k}\Vert_{\mathcal L(\mathcal B,\mathcal B)} \leq C \vartheta^k\, .
\end{equation}
Let $f,g:  \widetilde\Omega\rightarrow\mathbb C$ 
be two 
functions.
We assume that there exist two families $(f_\ell)_{\ell\in\mathbb Z^d}$ and $(g_\ell)_{\ell\in\mathbb Z^d}$ of functions defined on
$M\times\mathbb R\rightarrow\mathbb C$ and vanishing outside $\widetilde\Omega_0:=\widetilde\Omega\cup\left(M\times
\left[-\frac{\inf\tau}{10},0\right]\right)$ such that
$$\forall h \in \{f,g\}\ \ \ \forall (x,\ell,s)\in\widetilde\Omega,\quad h(x,\ell,s)=
h_\ell(x,s)+h_{\ell +\kappa(x)}(Tx,s-\tau(x))\, .$$
We assume moreover that one of these families is made of functions continuous in the last variable and that
\footnote{The notation $\Vert G\Vert_{\mathcal B'}$ means here $\Vert G\Vert_{\mathcal B'}:=\sup_{F\in\mathcal B,\ \Vert F\Vert_{\mathcal B}=1}\left|\mathbb E_{\nu}\left[G. F\right]\right|$.}
\begin{equation}\label{CCC00}
\int_{\mathbb R}\sum_{\ell\in\mathbb Z^d} (1+|\ell|^{K})  (\Vert
f_\ell(\cdot,u)
\Vert_{\mathcal B}+\Vert
g_\ell(\cdot,u)
\Vert_{\mathcal B'})\, du<\infty\, ,
\end{equation}
\begin{equation}\label{CCC00bis}
\exists p_0,q_0\in[1,+\infty]\ s.t.\ \frac 1{p_0}+\frac 1{q_0}=1\quad\mbox{and}\quad \sum_{\ell,\ell'\in\mathbb Z^d}\Vert f_\ell\Vert_{L^{p_0}(\nu\otimes\mathfrak l)}\, \Vert g_{\ell'}\Vert_{L^{q_0}(\nu\otimes\mathfrak l)}<\infty\, ,
\end{equation}
\begin{equation}\label{CCC0}
\sup_{
\xi\in\mathbb R
}\sum_{\ell,\ell'\in\mathbb Z^d}\Vert \hat f_\ell(\cdot,-\xi)\Vert_{\mathcal B}\Vert \hat g_{\ell'}(\cdot,\xi)\Vert_{\mathcal B'}<\infty
\end{equation}

Assume furthermore 
that $\hat f_\ell(\cdot,\xi) \in \mathbb B$ for every
$\ell\in\mathbb Z^d$ and $\xi\in\mathbb R$, where $\mathbb B$ is a Banach space such that 
\begin{equation}
\label{DMbound}
\sup_{\theta\in{[-\pi,\pi]^d}}\Vert P^n_{\theta,\xi}\Vert_{\mathcal L(\mathbb B,L^1)} \leq C |\xi|^{\alpha} e^{-n \delta |\xi|^{-\alpha}}
\end{equation}
for some suitable
positive
$C, \delta, \alpha$
and
\begin{equation}
\label{cond:fgdecay}
\forall\gamma>0,\quad\sum_{\ell,\ell'\in\mathbb Z^d}\left(
\Vert \hat f_\ell(\cdot,-\xi)\Vert_{\mathbb B} \, \Vert \hat g_{ \ell'}(\cdot,\xi)\Vert_{\infty}
\right)=O(|\xi|^{-\gamma})\, .
\end{equation}
Then
\begin{equation}
C_t(f,g)=\sum_{p=0}^{\lfloor \frac K2\rfloor}
   \tilde C_{p}(f,g)\left(\frac{t}{\nu(\tau)}\right)^{-\frac {d}2
-p
}+o\left(t^{-\frac {K+d}2}\right)\, ,
\end{equation}
as $t\rightarrow +\infty$
where
\begin{eqnarray}
\tilde C_{p}(f,g)&:=&
\sum
\frac 1{q!}\int_{\mathbb R}   \partial_2^qh_{m+j+r,k-\frac{m+j+d+r+1}2}
(s\sqrt{\nu(\tau)},1)
(-s)^q\, ds\label{Formula1}\\
&\ &*\frac {i^{m+j}}{r! m!}\left(\sum_{\ell,\ell'}\int_{\mathbb R^2}\nu\left(g_{\ell'}(\cdot,v)\left(\Pi^{(m)}_0( f_{\ell}(\cdot,u))\right)\right)\, \otimes(\ell'-\ell,u-v)^{\otimes r} dudv\otimes A_{j,k}\right)\nonumber
\end{eqnarray}
where the first sum is taken over the nonnegative integers $m,j,r,q,k$
satisfying 
$$m+j+r+q-2k=2p\quad\text{and }j\ge kJ$$ 
and
$\partial^q_2 h_{\alpha,\gamma}$ denotes the derivative of order $q$ 
with respect to the second variable of $h_{\alpha,\gamma}$ (defined by \eqref{defh})
and 
$A_{j,k}\in \mathcal S_j$  is given by 
\eqref{poly} of Appendix \ref{AppTaylor} for $k>0,$
$A_{0,0}=1$ and $A_{j,0}=0$ for $j>0.$
\end{theorem}

\begin{proof}[Proof of Theorem \ref{THMGENE}]

{\bf Step 1: Fourier transform.}

Notice that
\begin{equation}\label{keyformula}
C_t(f,g)=\sum_{\ell,\ell'}\sum_{n\ge 0}\int_{
M\times\mathbb R
}
f_\ell(x,s)\, g_{\ell'}\left(T^nx,s+t-\tau_n(x)\right)\mathbf 1_{\{\kappa_n(x)=\ell'-\ell\}}\ \, d
(\nu\otimes\mathfrak l)
(x,s) \, ,
\end{equation}
due to the dominated convergence theorem, \eqref{CCC00bis}
and the fact that the sum over $n$ is compactly supported, as explained below.
Indeed $g_{\ell'}(T^nx,s+t-\tau_n(x))\ne 0$ implies that 

%

$$-\frac{\inf \tau}{10}\le s+t-\tau_n(x)<\tau(T^nx), \text{ i.e. }
 \tau_n(x)-\frac{\inf\tau}{10}-s\le t<\tau_{n+1}(x)-s
\text{ with }-\frac{\inf\tau}{10}\le s<\tau(x)$$
and so
the sum over $n$ 
in \eqref{keyformula} is in fact is supported in $\{t_-,t_- + 1, ..., t_+ \}$, where
$$
t_- = \lceil t/\sup\tau\rceil
-2
, \quad t_+ = \lfloor t/ \inf\tau\rfloor + 2
\, .$$

Note that
\begin{equation}\label{Fourier1}
\mathbf 1_{\{\kappa_n(x)=\ell'-\ell\}}=\frac 1{(2\pi)^d}\int_{[-\pi,\pi]^d}e^{-i\, \theta\cdot(\ell'-\ell)}e^{i\, \theta\cdot\kappa_n}\, {d\theta}\, .
\end{equation}

Moreover, for every $x\in M$ and every positive integer $n$,
$$ h_{\ell,\ell',x,n}(\cdot):=\int_{
\mathbb R}
f_\ell(x,s)\, g_{\ell'}\left(T^nx,s+\cdot\right)\, ds $$
is the convolution of $f_\ell(x,-\cdot)$ with $g_{\ell'}(T^nx,\cdot)$. 
Due to \eqref{CCC00bis}, for $\nu$-a.e. $x$ and any choice of $\ell,\ell', n$, this $h_{\ell,\ell',x,n}(\cdot)$ well defined. Furthermore, it is
continuous (since $f_\ell(x,\cdot)$ or $g_{\ell'}(T^n x,\cdot)$ is continuous) with compact support and its Fourier transform is 
$$\hat f_\ell(x,-\cdot)\hat g_{\ell'}(T^nx,\cdot)\in L^\infty(\mathbb R)\cap L^1(\mathbb R).$$ Consequently,
$ h_{\ell,\ell',x,n}$ is equal to its inverse Fourier transform, that is
$$h_{\ell,\ell',x,n}(t-\tau_n(x))=\frac 1{2\pi}\int_{\mathbb R}e^{-i\xi(t-\tau_n(x))}
  \hat f_\ell(x,-\xi)\hat g_{\ell'}(T^nx,\xi)  \, d\xi \, .$$
Combining this with  \eqref{keyformula} and with \eqref{Fourier1},
we obtain

\begin{align}
&\ C_t(f,g)\nonumber\\
&=   
\frac 1{(2\pi)^{d+1}}\sum_{\ell,\ell'}\sum_{n\ge 0}\int_{M}\left(\int_{[-\pi,\pi]^d\times \mathbb R}\!\!\!\!\!\!\!\!\!\!\!\!\!\!\!\! e^{-i\xi t}
\hat f_\ell(x,-\xi)\, e^{-i\theta\cdot(\ell'-\ell)} 
e^{i\theta\cdot\kappa_n(x)}
e^{i\xi\tau_n(x)}\hat g_{\ell'}\left(T^nx,\xi\right)\, d\theta d\xi\right) d\nu(x) \label{EE0}\\
&=   \frac 1{(2\pi)^{d+1}}\sum_{\ell,\ell'}
\sum_{n=  {t_-}}^{t_+ }
\int_{M}\left(\int_{[-\pi,\pi]^d\times \mathbb R}e^{-i\xi t} e^{-i\theta\cdot(\ell'-\ell)}
P_{\theta,\xi}^n\left(\hat f_\ell(\cdot,-\xi)\right)\hat g_{\ell'}\left(\cdot,\xi\right)\, d\theta d\xi\right) d\nu \, ,\label{EE1}
\end{align}
where we used the fact that $P^n(e^{i\theta\cdot\kappa_n+i\xi\tau_n}F)=P_{\theta,\xi}^nF$.
We split 
$(2\pi)^{d+1} C_t(f,g)=I_1 + I_2$ where
$I_1$ stands  the contribution of $\xi\in [-b,b]$ and 
$I_2$ stands  the contribution of $|\xi|>b.$ 

{\bf Step 2: Reduction to the integration over a compact domain.}

Here we prove that $\DS |I_2| = o\left(t^{-\frac{K+d}2}\right)$. Observe that
\begin{gather*}
|I_2| \leq \sum_{\ell,\ell'}\sum_{n=  t_-}^{t_+} \int_{[-\pi,\pi]^d\times ([-\infty,-b] \cup [b,\infty])} \int_{M}
| P_{\theta,\xi}^n\left(\hat f_\ell(\cdot,-\xi)\right)\hat g_{\ell'}\left(\cdot,\xi\right)| \,d\nu  d\theta d\xi \\
\leq C't
\int_{[-\pi,\pi]^d\times ([-\infty,-b] \cup [b,\infty])}
\left(\sup_{n \in [t_-, t_+]} 
\sum_{\ell,\ell'} \left\| P_{\theta,\xi}^n\left(\hat f_\ell(\cdot,-\xi)\right) \right\|_{1} \|  
{\hat g}_{\ell'}\left(\cdot,\xi\right) \|_{\infty} \right)d\theta d\xi\, .
\end{gather*}
Now due to \eqref{DMbound}, we have  
\begin{gather*}
|I_2| \leq C''t \int_{[-\pi,\pi]^d}
 \int_{b<|\xi|} |\xi|^{\alpha} e^{-\delta t|\xi|^{-\alpha}}  \sum_{\ell,\ell'}
\left\Vert \hat f_\ell(\cdot,-\xi)\right\Vert_{\mathbb B}\, \left\Vert \hat g_{\ell'}(\cdot,\xi)\right\Vert_{\infty} d\xi d\theta\, .
\end{gather*}
We apply \eqref{cond:fgdecay} to see that for any $\gamma>0$ there is  $C''_\gamma>0$
such that 
$$
|I_2|\leq C''_{\gamma}t \int_{b<|\xi|} e^{-\delta t|\xi|^{-\alpha}} |\xi|^{\alpha - \gamma} d \xi 
\le C''_{\gamma} t^{2+\frac {1-\gamma}\alpha}\int_{\mathbb R}
    e^{-\delta |u|^{-\alpha}} |u|^{\alpha-\gamma}\, du.
$$
Choosing $\gamma$ large, we get $\DS |I_2| =  o\left(t^{-\frac{K+d}2}\right)$.
In the remaining part of the proof, we compute $I_1$.

{\bf Step 3: Expansion of the leading eigenvalue and eigenprojector.}

First, we use \eqref{decomp2}, \eqref{majoexpo}  and \eqref{CCC0} to write

$$C_t(f,g)
 \simeq  \frac 1{(2\pi)^{d+1}}\sum_{\ell,\ell'}\sum_n\int_{[-b,b]^{d+1}}e^{-i\xi t} e^{-i\theta\cdot(\ell'-\ell)}
\lambda_{\theta,\xi}^{n}\nu\left(\Pi_{\theta,\xi}\left(\hat f_\ell(\cdot,-\xi)\right)\hat g_{\ell'}\left(\cdot,\xi\right)\right)\, d(\theta,\xi)\, , $$
where $\simeq$ means that the difference between the LHS and the RHS is $o\left(t^{-\frac{K+d}2}\right)$.

Now the change of variables $(\theta, \xi)\mapsto (\theta, \xi)/\sqrt n$ gives
$$ 
C_t(f,g) \simeq
\frac{1}{(2\pi)^{d+1}}\sum_{\ell,\ell'}\sum_n n^{-\frac{d+1}2} \mathcal{I}
 (\ell, \ell', n) $$
where
$$ \mathcal I(\ell, \ell', n) = 
\int_{[-b\sqrt n,b\sqrt{n}]^{d+1}}\!\!\!\!\!\!\!\!\!\!\!\!\!\!\!\!\!\! e^{-i\frac{\xi}{\sqrt{n}}t } e^{-i\theta\cdot\frac{\ell'-\ell}{\sqrt{n}}}
\lambda_{(\theta,\xi)/\sqrt{n}}^{n}\nu\left(\Pi_{(\theta,\xi)/\sqrt{n}}\left(\hat f_\ell(\cdot,-\frac{\xi}{\sqrt{n}})\right)\right.
 \left.\hat g_{\ell'}\left(\cdot,\frac{\xi}{\sqrt{n}}\right)\right)\, d\theta\, d\xi\, . $$
Next 
with an error $o\left(t^{-\frac{K+d}2}\right)$, we can replace $\mathcal{I}(\ell, \ell',n)$ in the last sum by
\begin{equation}\label{AAA}
\int_{[-b\sqrt{n},b\sqrt{n}]^{d+1}}\!\!\!\!\!\!\!\!\!\!\!\!\!\!\!\!\!\!\! e^{-i\frac{\xi}{\sqrt{n}} t} e^{-i\theta\cdot\frac{\ell'-\ell}{\sqrt{n}}}
\lambda_{(\theta,\xi)/\sqrt{n}}^{n} \sum_{m=0}^{K+1}\frac 1{m!}\nu\left(\Pi^{(m)}_0
\left(\hat f_\ell\left(\cdot,-\frac{\xi}{\sqrt{n}}\right)\right)
\hat g_{\ell'}\left(\cdot,\frac{\xi}{\sqrt{n}}\right)\right)
*\frac{(\theta,\xi)^{\otimes m}}{n^{\frac m2}}\, d\theta\, d\xi \, .
\end{equation}
Indeed, for every $u\in\mathbb R^{d+1}$, there exist $\omega \in [0,1]$ and $x_{u} = \omega u$ such that
$$\Pi_{u}(\cdot)= \sum_{m=0}^{K}\frac 1{m!}\Pi^{(m)}_0(\cdot)*u^{\otimes m}+\frac 1{(K+1)!}\Pi^{(K+1)}_{x_{u}}(\cdot)*{u^{\otimes (K+1)}}.$$   
Denote
$$E_n:=\int_{[-b\sqrt{n},b\sqrt{n}]^{d+1}} \left|\lambda_{s/\sqrt{n}}^n\right|\,\left\Vert\Pi^{(K+1)}_{x_{s/\sqrt{n}}}-\Pi^{(K+1)}_{0}\right\Vert\, |s|^{K+1}\, ds\, .$$
Then $\DS \lim_{n\rightarrow +\infty}E_n=0$
by the Lebesgue dominated convergence theorem.
Therefore
$$\lim_{t\rightarrow +\infty} t^{\frac{K+d}2}  \sum_{n=  {t_-}}^{t_+ } n^{-\frac{d+1}2} \frac{E_n}{n^{\frac {K+1}2}}=0\, ,$$
justifying the replacement of $\Pi$ by its jet.

Recalling elementary identities
$a_{s/\sqrt n}^n =a_s$ and $a_s/a_{s/\sqrt{2}}=a_{s/\sqrt{2}}$, Lemma \ref{POLY} gives
\[
\left\vert \tilde\lambda_{s/\sqrt{n}}^{n}
   - a_{s}\sum_{k=0}^{\lfloor (K+1)/(J-2)\rfloor}\sum_{j=kJ}^{K+1+2k}
n^kA_{j,k}*(s/\sqrt{n})^{\otimes j} \right\vert\\
\le a_{s/\sqrt{2}}\, n^{-\frac{K+1}2}(1+|s|^{K_0})\eta(s/\sqrt{n})\, ,
\]
with $\DS \lim_{t\rightarrow 0}\eta(t)=0$ and $\DS \sup_{[-b,b]^d}|\eta|<\infty$.
Let $$E'_{n}:=
\int_{[-b\sqrt{n},b\sqrt{n}]^{d+1}}a_{s/\sqrt{2}}(1+|s|^{K_0})\eta(s/\sqrt{n})\, ds\, .
$$
Since the Lebesgue dominated convergence theorem gives 
$\DS \lim_{n\to\infty} E'_{n}=0$,
the same argument as above shows that
the error term arising from replacing
in \eqref{AAA}
$\tilde\lambda_{s/\sqrt{n}}^{n}$ by the above sum is negligible.
Since $\tilde\lambda_{\theta,\xi}=\lambda_{\theta,\xi} e^{-i\xi\nu(\tau)}$, 
we conclude
$$
C_t(f,g) \simeq\frac{1}{(2\pi)^{d+1}}\sum_{\ell,\ell'}\sum_n n^{-\frac {d+1}2}\int_{[-b\sqrt{n},b\sqrt{n}]^{d+1}}e^{-i\xi\frac{t-n\nu(\tau)}{\sqrt{n}}}e^{-i\theta.\frac{\ell'-\ell}{\sqrt{n}}}
a_{(\theta,\xi)}$$
$$\sum_{m=0}^{K+1}\frac 1{m!}\nu \left(\hat g_{\ell'}\left(\cdot,\frac{\xi}{\sqrt{n}}\right)\Pi^{(m)}_0\left(\hat f_\ell\left(\cdot,-\frac{\xi}{\sqrt{n}}\right)\right)\right)*\frac{(\theta,\xi)^{\otimes m}}{n^{\frac m2}} 
\left(
\sum_{k=0}^{\lfloor (K+1)/(J-2)\rfloor}\sum_{j=kJ}^{(K+1)+2k}
n^kA_{j,k}*\frac{(\theta,\xi)^{\otimes j}}{n^{\frac j2}}\right)\, d\theta\, d\xi.
$$

{\bf Step 4. Integrating by parts.}

Note that $\forall A\in\mathcal S_j,\forall B\in\mathcal S_m$ and $s\in \mathbb C^{d+1}$, $(B*s^{\otimes m})(A*s^{\otimes j})=(A\otimes B)*s^{\otimes (m+j)}$.
We claim that 
\begin{gather}
\frac{1}{(2 \pi)^{d+1}}
\int\displaylimits_{[-b\sqrt{n},b\sqrt{n}]^{d+1}}e^{-i\xi\frac{t-n\nu(\tau)}{\sqrt{n}}-i\theta.\frac{\ell'-\ell}{\sqrt{n}}}
a_{(\theta,\xi)}
\nu
\left(
\hat g_{\ell'}\left(\cdot,\frac{\xi}{\sqrt{n}}\right)\left(\Pi^{(m)}_0\left(\hat f_\ell\left(\cdot,-\frac{\xi}{\sqrt{n}}\right)\right) \otimes A_{j,k}\right)
\right)
\nonumber\\
*\left(\theta, \xi \right)^{\otimes(m+j)}
d\theta d\xi \nonumber \\
= i^{m+j}
\int_{\mathbb R^2}\Psi^{(m+j)}\left(\frac {\ell'-\ell} {\sqrt{n}},\frac {t-n\, \nu(\tau)+u-v}{\sqrt{n}}\right)*
\nu
\left(\Pi^{(m)}_0(f_\ell(\cdot,u))g_{\ell'}(\cdot,v)\otimes A_{j,k}\right)\, dudv\nonumber\\
 + o\left(\rho ^n 
\sup_{\xi\in\mathbb R}
\left\Vert \hat f_\ell(\cdot,\xi)\right\Vert_{\mathcal B}\, 
\left\Vert \hat g_{\ell'}(\cdot,\xi)\right\Vert_{\mathcal B'}
\right)
 \label{eq:conv}
\end{gather}
where $\Psi$ 
is defined by \eqref{defpsi} and
$\rho <1.$ Note that the integration in the second line of \eqref{eq:conv} is over a compact set
since
$f_\ell$ and $g_{\ell'}$
vanish outside of
a compact set.

To prove \eqref{eq:conv}, we first note that,
due to \eqref{CCC0}
by making an exponentially small error we can replace the integration in the first line to 
$\mathbb{R}^{d+1}.$ Second, we observe that 
$\Pi_0^{(m)} \hat{f}_{\ell}=\widehat{f_{m,\ell}}$ where
$f_{m,l}=\Pi_0^{(m)} f_{\ell}$
and that $\hat h(\xi/\sqrt{n})=\widehat{(\sqrt{n}h(\sqrt{n}\cdot))}(\xi)$.
Third, since $a$ 
is the Fourier transform of $\Psi,$
it follows that 
$$
(\theta,\xi)\mapsto
(-i)^{\sum_{j=1}^{d+1} k_j} \;\theta_1^{k_1}\dots \theta_d^{k_d} \xi^{
k_{d+1}
} a_{(\theta, \xi)}
\text{ is the Fourier transform of }
s\mapsto
\frac{\partial^{\sum_{j=1}^{d+1} k_j}}{(\partial s_1)^{k_1}\cdots
(\partial s_{d+1})^{k_{d+1}}}\Psi
.$$
Fourth, we use the inversion formula for the Fourier transform. To take the inverse Fourier 
transform with respect to $\xi$ we note that we have a triple product, which is a Fourier transform of the
triple convolution of the form
$$ 
i^{m+j}
\int_{\mathbb{R}^2}
\Psi^{(m+j)}\left(\frac {\ell'-\ell} {\sqrt{n}},\frac {t-n\, \nu(\tau)}{\sqrt{n}}
-t_1-t_2
\right)*
  n
 f_{m,\ell}
(\cdot, -\sqrt{n} t_1) g_{\ell'} (\cdot, \sqrt{n} t_2) dt_1 dt_2. $$
Making the change of variables $u=-\sqrt{n} t_1,$ $v=\sqrt{n} t_2$ we obtain \eqref{eq:conv}.

Formula
\eqref{eq:conv} implies that
\begin{equation}
\label{INT22}
 C_t(f,g)
\simeq \sum_{m=0}^{K+1}
\sum_{k=0}^{\lfloor (K+1)/(J-2)\rfloor}\sum_{j=kJ}^{K+1+2k}
\frac{i^{m+j}}{m!}\sum_{\ell,\ell'}\sum_n
n^{-\frac {m+j+d+1-2k}2} 
\end{equation}
\begin{equation*}
  \int_{[
-\frac{\inf\tau}{10}
, \sup \tau)^2
}\Psi^{(m+j)}\left(\frac {\ell'-\ell} {\sqrt{n}},\frac {t-n\, \nu(\tau)+u-v}{\sqrt{n}}\right)*
    \nu \left(\Pi^{(m)}_0(f_\ell(\cdot,u))g_{\ell'}(\cdot,v)\otimes A_{j,k}\right)\, dudv 
\end{equation*}

{\bf Step 5: Simplifying the argument of $\Psi$.}


Note that there exist  $a_0,a'_0, c_{m+j},c'_{m+j}>0$ such that, for every $\ell',\ell\in\mathbb Z^2$ and every $u,v\in
(-\frac{\inf\tau}{10},\sup \tau)
$,
\begin{equation}\label{AAAA}
\Psi^{(m+j)}\left(\frac {\ell'-\ell} {\sqrt{n}},\frac {t-n\, \nu(\tau)+u-v}{\sqrt{n}}\right)\le c_{m+j}e^{-\frac{a_0}n\left((\ell'-\ell)^2+(t-n\nu(\tau)+u-v)^2\right)}\le c'_{m+j}e^{-\frac{a'_0}n(t-n\nu(\tau))^2}\, .
\end{equation}
Combining this estimate with Lemma \ref{sumint} (with $\alpha=0$), we obtain that
$$\sup_{u,v\in(-\frac{\inf\tau}{10},\sup \tau)}\sum_{n=  t_-}^{t_+}
n^{-\frac {m+j+d+1-2k}2}
\left|\Psi^{(m+j)}\left(\frac {\ell'-\ell} {\sqrt{n}},\frac {t-n\, \nu(\tau)+u-v}{\sqrt{n}}\right)\right| = 
O\left( t^{-\frac {m+j+d-2k}2} \right) \, .$$
Therefore, the terms of \eqref{INT22} corresponding to $(m,k,j)$
with $m+j-2k>K$ are in $o\left(t^{-\frac{K+d}2}\right)$ and so
the third summation in \eqref{INT22} can be replaced by $\sum_{j=kJ}^{K-m+2k}$.
The constraint $K-m+2k\ge kJ$ implies that we can replace the second
summation in \eqref{INT22} by $\sum_{k=0}^{\lfloor K/(J-2)\rfloor}$.

Next let $p=K-m-j+2k.$
 We claim that we can replace
$\Psi^{(m+j)}\left(\frac {\ell'-\ell} {\sqrt{n}},\frac {t-n\, \nu(\tau)-u-v}{\sqrt{n}}\right)$ in \eqref{INT22} by
$$\sum_{r=0}^{p}\frac 1{r!\, n^{\frac r 2}}\Psi^{(m+j+r)}\left(0,\frac {t-n\, \nu(\tau)}{\sqrt{n}}\right)*(\ell'-\ell,u-v)^{\otimes r}\, .$$
Indeed by Taylor's theorem, we just need to verify that for 
\begin{eqnarray} \label{PsiRemainder}
&\ &\lim_{t\rightarrow +\infty}t^{\frac {K+d}2}\sum_{\ell,\ell'}\int_{\mathbb R^2}\Vert f_\ell(\cdot,u)\Vert_{\mathcal B}\Vert g_{\ell'}(\cdot,v)\Vert_{\mathcal B'}|(\ell'-\ell,u-v)|^p 
\sum_{n= t_-}^{t_+}  n^{-\frac {m+j+d+1-2k+p}2}\\
&\ &\ \ \sup_{x\in(0,1)}\left|\Psi^{(m+j+p)}\left(x\frac {\ell'-\ell} {\sqrt{n}},\frac {t-n\, \nu(\tau)+x(u-v)}{\sqrt{n}}\right)-\Psi^{(m+j+p)}\left(0,\frac {t-n\, \nu(\tau)}{\sqrt{n}}\right)\right|\, dudv
\nonumber\\
&=&0  \nonumber \, .
\end{eqnarray}
By \eqref{AAAA} and
Lemma \ref{sumint}
\begin{gather*}
\sum_{n= t_-}^{t_+}
n^{-\frac {m+j+d+1-2k-p}2}
\sup_{x\in(0,1)}\left|\Psi^{(m+j+p)}\left(x\frac {\ell'-\ell} {\sqrt{n}},\frac {t-n\, \nu(\tau)+x(u-v)}{\sqrt{n}}\right)\right|\\
\le c'_{m+j+p}\, \sum_{n= t_-}^{t_+}
n^{-\frac {m+j+d+1-2k+p}2}
 e^{-\frac{a'_0}{n}(t-n\, \nu(\tau))^2}=O(t^{-\frac {m+j+d-2k+p}2})
 \end{gather*}
uniformly in 
$\ell,\ell'\in\mathbb Z^d$ and 
$u,v\in(-\frac{\inf\tau}{10},\sup\tau)$.
This combined with
\eqref{CCC00}
shows that the
LHS of \eqref{PsiRemainder} is dominated by an integrable function, so \eqref{PsiRemainder}
follows by the dominated convergence theorem.

Therefore
\begin{eqnarray}
C_t(f,g)&\simeq &\sum_{\ell,\ell'}
\sum_{m=0}^{
K+1
}\sum_{k=0}^{\lfloor K/(J-2)\rfloor}\sum_{j=kJ}^{K{-m+2k}}\sum_{r=0}^{K-m-j+{2k}}\frac{i^{m+j}}{\,r! m!}
\sum_{n=t_-}^{t_+}
n^{-\frac {m+j+d+r+1-2k}2}
\Psi^{(m+j+r)}\left(0,\frac {t-n\nu(\tau)}{\sqrt{n}}\right)*\nonumber\\
&\ &\int_{\mathbb R^2}\left(\nu\left(g_{\ell'}(\cdot,v)\left(\Pi^{(m)}_0( f_{\ell}(\cdot,u))\right)\, \otimes(\ell'-\ell,+u-v)^{\otimes r} dudv\otimes A_{j,k}\right)\right)\, .
\label{EE3}
\end{eqnarray}

{\bf Step 6: Summing over $n$.}

Performing the summation over $n$ and using Lemma \ref{sumint} 
we obtain
\begin{equation}
C_t(f,g)\simeq\sum_{\ell,\ell'}
\sum_{m=0}^{K+1}\sum_{k=0}^{\lfloor K/(J-2)\rfloor}\sum_{j=kJ}^{
K-m+2k
}\sum_{r=0}^{K-m-j+2k}\sum_{q=0}^{K+2k-m-j-r}
\frac{i^{m+j}(t/\nu(\tau))^{-\frac {m+j+d+r+q-2k}2}}{\,r! m!\, q!\, 
  {(\nu(\tau))^{\frac{q+1}{2}}}}\label{Formula1.5}
\end{equation}
$$
\int_{\mathbb R}   \partial_2^qh_{m+j+r,k-\frac{m+j+d+r+1}2}(s,1)(-s)^q\, ds 
$$
$$*\left(\int_{\mathbb R^2}\nu\left(g_{\ell'}(\cdot,v)\left(\Pi^{(m)}_0( f_{\ell}(\cdot,u))\right)\right)\, \otimes(\ell'-\ell,u-v)^{\otimes r} dudv\otimes A_{j,k}\right)
\, .$$
Therefore
$
C_t(f,g)\simeq \sum_{p=0}^{K}\tilde C_{p/2}(f,g)\left(\frac{t}{\nu(\tau)}\right)^{-\frac {d+p}2}$
where
\begin{eqnarray}
\tilde C_{p/2}(f,g)&:=&
\sum
\frac 1{q!}\int_{\mathbb R}   \partial_2^qh_{m+j+r,k-\frac{m+j+d+r+1}2}(s\sqrt{\nu(\tau)},1)
(-s)^q\, 
ds\\
&\ &*\frac {i^{m+j}}{r! m!}\left(\sum_{\ell,\ell'}\int_{\mathbb R^2}\nu\left(g_{\ell'}(\cdot,v)\left(\Pi^{(m)}_0( f_{\ell}(\cdot,u))\right)\right)\, \otimes(\ell'-\ell,u-v)^{\otimes r} dudv\otimes A_{j,k}\right),\nonumber
\end{eqnarray}
and the first sum is taken over the nonnegative integers $m,j,r,q,k$
satisfying $m+j+r+q-2k=p$.
Applying 
Lemma \ref{vanishint} with $b= m+j+r$, we see that $\tilde C_{p/2} = 0$ if $p$ is an odd integer. This
concludes the proof of Theorem \ref{THMGENE}.
\end{proof}

%
%
%
%
%
%


\subsection{A general result for hyperbolic systems}
 Here we
consider extensions of systems 
with good spectral properties.
\begin{theorem}\label{THMGENE2}
Assume $\tau$ and $\kappa$ uniformly bounded, 
and that $\inf\tau>0$.
Let $\Sigma$ be a 
$(d+1)$-dimensional positive symmetric matrix.
Let $K,J$ be two integers such that, $ 3\le J\le L=K+3$. 
Let $(\mathcal V,\Vert\cdot\Vert_{\mathcal V})$ be a complex Banach space of functions $f: M\rightarrow\mathbb C$ such that $\mathcal V\hookrightarrow L^{\infty}(\nu)$.
Assume that $( M,\nu, T)$ is an extension, by $\mathfrak p: M\rightarrow\bar \Delta$, of a dynamical system $(\bar \Delta,\bar\nu,\bar T)$ with Perron-Frobenius operator $\bar P$ and
that
there exists a Banach space $\mathcal B$ of complex functions $f:\bar\Delta\rightarrow \mathbb C$ such that $\mathcal B\hookrightarrow  L^1 (\bar \Delta, \bar \nu)$ and 
$\mathbf{1}_{\bar \Delta} \in \mathcal B$. 
Assume moreover that the following conditions hold true:
\begin{itemize}
\item there exist a positive integer $m_0$ and a $\bar\nu$-centered bounded
function $\bar\kappa:\bar \Delta\rightarrow\mathbb Z^d$ such that 
$\bar\kappa\circ\mathfrak p=\kappa\circ  T^{m_0}$,
\item there exist
$\beta_0\ge 0,$ 
 a function $\bar\tau:\bar \Delta\rightarrow\mathbb R$
and a function $\chi:M\rightarrow \mathbb R$ s.t. 
$\tau=\bar\tau\circ\mathfrak p+\chi-\chi\circ T$
and
for every $\xi\in\mathbb R$, we have
$e^{i\xi\, \chi}\in\mathcal V$ with $\left\Vert e^{i\xi\,\chi}\right\Vert_{\mathcal V}=
O\left(
|\xi|^{\beta_0}
\right)$
and $(\bar\tau_{m_0})^qe^{-i\xi\bar \tau_{m_0}}\in\mathcal B$ for every 
$q\le L$.
\item $(\bar P_{\theta,\xi}:\bar f\mapsto \bar P(e^{i\theta\cdot\bar\kappa}e^{i\xi\,\tau}\bar f))_{(\theta,\xi)\in[-\pi,\pi]^d\times\mathbb R}$  is a family of linear continous operators on $\mathcal B$
such that
\begin{equation}\label{boundedOp}
\sup_{\theta,\xi,n}\Vert \bar P_{\theta,\xi}^n\Vert<\infty\, ,
\end{equation}
and there exist constants $b\in(0,\pi]$, $C>0$, $\vartheta \in(0,1)$, $\beta>0$ and three functions $\lambda_\cdot:[-b,b]^{d+1} \to \mathbb C$ and
$\Pi_\cdot,R_\cdot:[-b,b]^{d+1} \to \mathcal L(\mathcal B,\mathcal B)$ (assumed to be $C^{L}$-smooth) such that
\begin{equation}\label{DLlambdab}
\tilde\lambda_{\theta,\xi}:= \lambda_{\theta,\xi} e^{-i\xi\nu(\tau)}=1-\frac 12\Sigma*(\theta,\xi)^{\otimes 2}+o(|(\theta,\xi)|^2),\quad \mbox{as}\ 
(\theta,\xi)
\rightarrow 0\, ,
\end{equation}
 $\DS \lambda_0=1$ and $\DS  \Pi_0 = \mathbb E_{\bar\nu} [\cdot]\mathbf 1_{
\bar \Delta}$ and such that, in $\mathcal L(\mathcal B,\mathcal B)$,
\begin{equation}\label{decomp2b}
\forall s\in[-b,b]^{d+1},\quad \bar P_s =\lambda_s\Pi_s+R_s,\quad 
\Pi_s R_s      = R_s \Pi_s = 0,\quad
\Pi_s^{2}  = \Pi_{s}\, ,
\end{equation}
\begin{equation}\label{majoexpob}
 \forall k \in \mathbb N
\sup_{m=0,...,L}\sup_{s\in [-b,b]^{d+1}} \Vert{(R_s^k)^{(m)}}\Vert_{\mathcal L(\mathcal B,\mathcal B)}+\sup_{\theta\in [-\pi,\pi]^d\setminus[-b,b]^d,\ |\xi|\le 
b
} \Vert{\bar P_{\theta,\xi}^k}\Vert_{\mathcal L(\mathcal B,\mathcal B)} \leq C \vartheta^k.
\end{equation}
Furthermore, there is a Banach space $\mathbb B$ such that
\begin{equation}
\label{DMbound1}
\exists C,\delta,\alpha>0,\quad \sup_{\theta\in{[-\pi,\pi]^d}}\Vert \bar P^n_{\theta,\xi}\Vert_{\mathcal L(\mathbb B,L^1)} \leq C |\xi|^{\alpha} e^{-n \delta |\xi|^{-\alpha}}\, ,
\end{equation}
and $\forall k<J$, $\tilde\lambda^{(k)}_0=a^{(k)}_0$ where $a_s$ is given by \eqref{Psi-FT}. 
\item 
there exist $C_0>0$ and $\vartheta \in(0,1)$ and 
continuous
linear maps $\boldsymbol{\Pi}_n:\mathcal V\rightarrow \mathcal B\cap\mathbb B$,
such that, for every $f\in\mathcal V $ and every integer $n\ge m_0$ 
and for any $\theta \in [-\pi, \pi]^d, \xi \in \mathbb R$ and for any non-negative integer 
$j
=0,...,L
$,
\begin{equation}\label{approxiA}
 \Vert  f\circ  T^{n} - \boldsymbol{\Pi}_n(f) \circ \mathfrak p 
\Vert_{\infty} \le  C_0 \Vert f\Vert_{\mathcal V}\,  \vartheta^n\, ,
\end{equation}
\begin{equation}\label{approxiB}
\left\Vert \bar  P_{\theta,\xi}^{2n}(e^{-i\theta.\bar \kappa_{n-m_0}-i\xi. \bar \tau_n}\boldsymbol{\Pi}_n f)\right\Vert_{\mathbb B}\le C_0 (1 + |\xi|) \Vert f\Vert_{\mathcal V}\, ,
\end{equation}
\begin{equation}\label{approxiC}
\left\Vert \frac{\partial^j}{\partial (\theta,\xi)^j}(\bar  P_{\theta,\xi}^{2n}(e^{-i\theta.\bar \kappa_{n-m_0}-i\xi. \bar \tau_n}\boldsymbol{\Pi}_n f))\right\Vert_{\mathcal B}\le C_0 n^j 
 (1 + |\xi|) \Vert f\Vert_{\mathcal V}\, ,
\end{equation} 
\begin{equation}\label{approxiD}
 \left\Vert   \frac{\partial^j}{\partial (\theta,\xi)^j} (\boldsymbol{\Pi}_n(f)e^{i\theta\cdot\bar\kappa_{n-m_0}+i\xi.\bar\tau_n})\right\Vert_{\mathcal B'} \le C_0n^j
 \Vert f\Vert_{\mathcal V}\, ,
\end{equation}

with $\bar \kappa_n:=\sum_{k=0}^{n-1}\bar\kappa\circ\bar T^k$
and $\bar\tau_n:=\sum_{k=0}^{n-1}\bar\tau\circ\bar T^k$.
\end{itemize}
Let $f,g:\widetilde\Omega\rightarrow \mathbb C$
such that
\begin{equation}\label{decomp_h}
\forall h \in\{
f,g\}\ \ \ \forall (x,\ell,s)\in\widetilde\Omega,\quad h(x,\ell,s)=h_\ell(x,s)+
h_{\ell+\kappa(x)}(Tx,s-\tau(x))\, ,
\end{equation}
where $(f_\ell)_{\ell\in\mathbb Z^d}$ and $(g_\ell)_{\ell\in\mathbb Z^d}$ are two families of functions defined on
$M\times\mathbb R\rightarrow\mathbb C$ and vanishing outside $\widetilde\Omega_0:=\widetilde\Omega\cup\left(M\times
 \left[-\frac{\inf\tau}{10},0\right]\right)$.
We assume moreover that one of these families is made of functions continuous in the last variable and that
there exists $\beta_0$ such that
$\xi\mapsto e^{i\xi.\chi}\hat f_\ell(\cdot,\xi) $ 
and $\xi\mapsto e^{i\xi.\chi}\hat g_\ell(\cdot,\xi) $ are $C^{L}$
from
$\mathbb R$ to $\mathcal V$ and for every 
$k=0,...,{L}$, 
\begin{equation}\label{CCC0a}
\sup_{|\xi|\le b}\sum_{\ell\in\mathbb Z^d}\left(\left\Vert  \frac{\partial^k}{\partial \xi^k}
    \left(e^{-i\xi.\chi}\hat f_\ell(\cdot,\xi)\right)\right\Vert_\mathcal V 
    +\left\Vert  \frac{\partial^k}{\partial \xi^k}
    \left(e^{-i\xi.\chi}\hat g_\ell(\cdot,\xi)\right)\right\Vert_\mathcal V \right)<\infty\, ,
\end{equation}
\begin{equation}\label{CCC00a}
\sum_{\ell}\int_{\mathbb R} {(1+|\ell|)^{K}}  \left(\Vert
 f_\ell(\cdot,u)
\Vert_{\mathcal V}+\Vert  
g_\ell(\cdot,u)
\Vert_{\mathcal V}\right)\, du<\infty\, ,
\end{equation}
\begin{equation}\label{CCCa}
\forall\gamma>0,\quad \sum_{\ell,\ell'}\left(\Vert e^{i\xi.\chi}\hat f_\ell(\cdot,-\xi)\Vert_{\mathcal V} \, \Vert e^{
-i\xi.\chi
} \hat g_{
\ell'
}(\cdot,\xi)\Vert_{\mathcal V}\right)=O(|\xi|^{-\gamma})\, .
\end{equation}

\begin{equation}\label{CCC00bisa}
\sum_{\ell\in\mathbb Z^d}\Vert f_\ell\Vert_\infty<\infty\quad\mbox{or}\quad \sum_{\ell\in\mathbb Z^d}\Vert g_\ell\Vert_\infty<\infty\, ,
\end{equation}

Then
\begin{equation*}
C_t(f,g)=\sum_{p=0}^{\lfloor \frac K2\rfloor}
   \tilde C_{p}(f,g)\left(\frac{t}{\nu(\tau)}\right)^{-\frac {d}2+p}+o\left(t^{-\frac {K+d}2}\right)\, ,
\end{equation*}
as $t\rightarrow +\infty$,
where
\begin{eqnarray}
\tilde C_{p}(f,g)&=&
\sum
\frac 1{q!}
\frac{1}{(\nu(\tau))^{\frac{q+1}{2}}}
\int_{\mathbb R}   \partial_2^qh_{m+j+{ r},k-\frac{m+j+d+{r}+1}2}(s,1)(-s)^q\, ds \label{Formula2}
\\
&&*
\frac {i^{m+j}}{r! m!}
\left( \sum_{\ell, \ell'}\int_{\mathbb R^2}\mathcal B_m\left(f_{\ell}(\cdot,u),g_{\ell'}(\cdot,v)\right)\, \otimes(\ell'-\ell,u-v)^{\otimes { r}}\, dudv\otimes A_{j,k}\right)\, ,\nonumber
\end{eqnarray}
where the first sum is taken over the nonnegative integers $m,j,r,q,k$
satisfying 
$$m+j+r+q-2k=2p\quad\text{and }j\ge kJ,$$ 
$h$ is defined in \eqref{defh}, $A_{j,k}$ for $k>0$
are the multilinear forms  given by equation \eqref{poly} from Appendix \ref{AppTaylor},
$A_{0,0}=1$ and 
$A_{j,0}=0$ for $j>0$
and 
$\mathcal{B}_m: \mathcal{V} \times \mathcal{V}\to \mathcal{S}_m$ 
are bilinear forms defined in \eqref{defBn} below.
\end{theorem}

To define $\mathcal B_m$ we need
the following preliminary lemma, the proof of which is given at the end of this section, after the proof of Theorem \ref{THMGENE2}.

\begin{lemma}\label{technical}
Under the assumptions of Theorem \ref{THMGENE2},
let $u,v:M\times([-\pi,\pi]^d]\times\mathbb R)\rightarrow\mathbb C$ be two functions such that $(\theta,\xi)\mapsto e^{-i\xi\chi}u(\cdot,\theta,\xi)$ and
$(\theta,\xi)\mapsto e^{-i\xi\chi} v(\cdot,\theta,\xi)$ are $L$ times differentiable at 0 as functions from
$[-\pi,\pi]^d\times\mathbb R$ to $\mathcal V$.\\
Then, for every integer $N=0,...,L$, the quantity
$$\mathcal A_{N}(u,v):=\lim_{n\rightarrow +\infty}\left(\mathbb E_{\nu}\left[u(\cdot,-\theta,-\xi)e^{i\theta \cdot\kappa_n+i\xi\tau_n}v( T^n(\cdot),\theta,\xi)\right]\lambda_{\theta,\xi}^{-n}\right)_{|(\theta,\xi)=0}^{(N)}$$
is well defined and satisfies 
$$\left|\mathcal A_N(u,v)\right|=O\left(
\Vert u\Vert_{\mathcal W, +}\, \Vert v\Vert_{\mathcal W, -} 
\right) .$$
Moreover for each $\bar L\in \mathbb{N}$ we have 
$$\left|\mathcal A_N(u,v)-\left(\mathbb E_{\nu}\left[u(\cdot,-\theta,-\xi)e^{e^{i\theta \cdot\kappa_n+i\xi\tau_n}}v(\bar T^n(\cdot),\theta,\xi)\right]\lambda_{\theta,\xi}^{-n}\right)_{|(\theta,\xi)=0}^{(N)}\right|=O\left(\Vert u\Vert_{\mathcal W, +}\, \Vert v\Vert_{\mathcal W, -}
n^{- \bar L}\right)\, $$
with
$$\Vert u\Vert_{\mathcal W, { \pm}}:=\sum_{m=0}^{L}\left\Vert   
\left(e^{
- i\xi \chi} u(\cdot,\theta,\xi)\right)^{(m)}_{|(\theta,\xi)=0} 
\right\Vert_{\mathcal V}<\infty\, .$$
\end{lemma}
We let $\mathcal{B}_m$ to be the restriction of $\mathcal A_m$ on the space of functions depending on
neither $\theta$ nor $\xi.$ Thus
\begin{equation}
\label{defBn}\mathcal B_m(F,G):=
\lim_{n\rightarrow +\infty}\left(\mathbb E_{\nu}\left[F(\cdot)e^{i\theta \cdot\kappa_n(.)
+i\xi(\tau_n(.)-n\nu(\tau))}G(T^n(\cdot))\right]\tilde\lambda_{\theta,\xi}^{-n}\right)_{|(\theta,\xi)=0}^{(m)} .
\end{equation}

Observe that \eqref{Formula2} has the same form as \eqref{Formula1}
with $\nu\left(G\Pi_0^{(m)}(F)\right)$ replaced by 
$ \mathcal B_m(F,G)$. In fact these two quantities coincide under the assumptions of 
Theorem \ref{THMGENE}.
 More precisely, suppose that 
$(M,\nu, T)=(\bar\Delta,\bar\nu,\bar T)$.
Then,
for $(\theta,\xi)\in[-b,b]^{d+1}$,
$$ \lim_{n\rightarrow +\infty}\left(\mathbb E_{\nu}\left[F(\cdot)e^{i\theta \cdot\kappa_n(.)
    +i\xi(\tau_n(.)-n\nu(\tau))}G(T^n(\cdot))\right]\tilde\lambda_{\theta,\xi}^{-n}\right)=
\lim_{n\rightarrow +\infty}\left(\mathbb E_{\nu}\left[\left(P_{\theta, \xi}^n F\right) G\right]\lambda_{\theta,\xi}^{-n}\right)
$$
$$= \lim_{n\rightarrow +\infty}\nu\left(G\left[\Pi_{\theta, \xi} F+\lambda_{\theta, \xi}^{-n} R_{\theta, \xi}^n F\right]\right)
=
\nu(G\Pi_{\theta, \xi} (F)).
$$
In particular, in this case $\mathcal{B}_0(F, G)=\nu(G \Pi_0 (F)).$
A similar argument shows that $$\mathcal{B}_m(F, G)=\nu(G \Pi_0^{(m)} (F)), $$
see the proof of Lemma \ref{technical} for details.

We also note that due to mixing of $T$ we have
\begin{equation}
\label{B0-Mixing}
\mathcal{B}_0 (F, G)=\nu(F) \nu(G). 
\end{equation}

Let us mention that $ \mathcal B_m(F,G)$ for $m\le 3$ as well as $\lambda_0^{(k)}$ for $k\le 4$ have been computed
in \cite{Soazmixing} in the case of the Sinai billiard with finite horizon
with $\kappa_n$ instead of $(\kappa_n,\tau_n-n\nu(\tau))$ (see Lemma 4.3 and Propositions A.3 and A.4 therein) but the formulas can be extended to  the present context since $(\kappa,\tau)$ 
is dynamically Lipschitz and since the reversibility property stated in
\cite[Lemma 4.3]{Soazmixing} also holds for $(\kappa,\tau)$.

\begin{proof}[Proof of Theorem \ref{THMGENE2}.]
We note that the proof of Theorem \ref{THMGENE2}
    is in many places similar to the proof of Theorem \ref{THMGENE} so below we mostly concentrate on the
    places requiring significant modifications. We note that  we could have presented Theorem \ref{THMGENE2} without
    discussing Theorem \ref{THMGENE} first, however, since the formulas are quite cumbersome in the present setting we prefer to
    discuss the argument in the simpler setup of Theorem \ref{THMGENE} first.

Decreasing the value of $b$ if necessary, we can assume 
that
\begin{equation}
\label{condb2}
\forall s\in[-b,b]^{d+1},\;
\vartheta^{\frac{1}{10L(d+1)}}\le |\lambda_s|\le a_{s/\sqrt{2}}\, ,
\end{equation}
where $\vartheta$ is given by \eqref{majoexpob}.
Let
$k_t:=\lceil 
(L+\frac{L+1+d}2)\log t/|
\log
\vartheta
|\rceil$.


We consider $F_t,  G_t:\bar\Delta\times \mathbb Z^d\times \mathbb R\rightarrow\mathbb C$ given by
$$\forall \ell\in\mathbb Z^d,\forall \xi\in\mathbb R,\quad F_t(\cdot,\ell,\xi):=\boldsymbol{\Pi}_{k_t}(e^{-i\xi\chi(\cdot)}\hat f_\ell(\cdot,\xi))\ \ \  \mbox{and}\ \ \ G_t(\cdot,\ell,\xi):=\boldsymbol{\Pi}_{k_t}(e^{-i\xi\chi(\cdot)}\hat g_\ell(\cdot, \xi))\, .$$
As in \eqref{EE0}, 
using \eqref{CCC00a} and \eqref{CCC00bisa},
$C_t(f,g)$ is equal to 
\begin{equation}
\frac 1{(2\pi)^{d+1}}\sum_{\ell,\ell'}
\sum_{n= t_-}^{ t_+}
\int_{M}\left(\int_{[-\pi,\pi]^d\times \mathbb R}e^{-i\xi t}
\hat f_\ell(x,-\xi)\, e^{-i\theta\cdot(\ell'-\ell)}e^{i\theta\cdot\kappa_n(x)}e^{i\xi\tau_n(x)}\hat g_{\ell'}\left(T^nx,\xi\right)\, d\theta d\xi\right) d\nu(x)\, . \label{hypeq1}
\end{equation}
In order to apply the spectral method, as in the proof of Theorem \ref{THMGENE},
we want to reduce the integration over $M$ in \eqref{hypeq1} to integration over 
${\bar \Delta}$.
Namely

\begin{eqnarray}
&\ &\mathbb E_{\nu}\left[
\hat f_\ell(\cdot,-\xi)\, e^{i\theta\cdot\kappa_n}e^{i\xi\tau_n}\hat g_{\ell'}\left(T^n\cdot,\xi\right)\right]\nonumber\\
&=&\mathbb E_{\nu}\left[
e^{i\xi\chi\circ T^{k_t}}\hat f_\ell(T^{k_t}(\cdot),-\xi)\, e^{i\theta\cdot \bar\kappa_{n}\circ \bar T^{k_t-m_0}\circ \mathfrak p}e^{i\xi\bar\tau_n\circ\bar T^{k_t}\circ\mathfrak p}e^{-i\xi\chi\circ T^{k_t+n}}\hat g_{\ell'}\left(T^{k_t+n}\cdot,\xi\right)\right]\nonumber\\
&=&\mathbb E_{\nu}\left[
e^{i\xi\chi\circ T^{k_t}}\hat f_\ell(T^{k_t}(\cdot),-\xi)e^{-i\theta.\bar\kappa_{k_t-m_0}\circ\mathfrak p-i\xi.\bar\tau_{k_t}\circ\mathfrak p}\, e^{i\theta\cdot \bar\kappa_{n}\circ \mathfrak p}e^{i\xi\bar\tau_n\circ\mathfrak p} \right.\nonumber\\
&\ &\ \ \ \ \left.e^{i\theta.\bar\kappa_{k_t-m_0}\circ\bar T^n\circ\mathfrak p+i\xi.\bar\tau_{k_t}\circ\bar T^n\circ\mathfrak p}e^{-i\xi\chi\circ T^{k_t+n}}\hat g_{\ell'}\left(T^{k_t+n}\cdot,\xi\right)\right]\label{FFFa}\\
&=&\mathbb E_{\bar\nu}\left[
F_t(\cdot,\ell,-\xi )e^{-i\theta.\bar\kappa_{k_t-m_0}-i\xi.\bar\tau_{k_t}}\, e^{i\theta\cdot \bar\kappa_{n}}e^{i\xi\bar\tau_n} \right.\nonumber\\
&\ &\ \ \ \ \left.e^{i\theta.\bar\kappa_{k_t-m_0}\circ\bar T^n+i\xi.\bar\tau_{k_t}\circ\bar T^n}G_t\left(\bar T^{n}(\cdot),\ell',\xi\right)\right]+O\left(\vartheta^{k_t}
d_{\ell,\ell'}(\xi)
\right), \, \nonumber
\end{eqnarray}
with 
$d_{\ell,\ell'}(\xi)
:=\left(\Vert e^{i\xi.\chi}\hat f_\ell(\cdot,-\xi)\Vert_{\mathcal V} \, \Vert e^{
-i\xi.\chi} \hat g_{\ell'}(\cdot,\xi)\Vert_{\mathcal V}\right)$
where we used 
\begin{itemize}
\item the $T$-invariance of $\nu$ and the definitions of $\bar \kappa$ and $\bar \tau$ in the first equation,
\item the identities $\bar\kappa_n\circ \bar T^{k_t-m_0}=\bar\kappa_n-\bar\kappa_{k_t-m_0}+
\bar\kappa_{k_t-m_0}\circ\bar T^{n}$
and
$\bar \tau_n\circ \bar T^{k_t}=\bar\tau_n-\bar\tau_{k_t}+\bar\tau_{k_t}\circ\bar T^n$ in the second one,
\item \eqref{approxiA}
and $\mathcal V\hookrightarrow L^\infty(\nu)$
in the last one.
\end{itemize}

Now using the properties of Perron-Frobenius operator given by \eqref{PerronFr} and 
\eqref{PerronFrTw} we obtain

$$ \mathbb E_{\nu}\left[
\hat f_\ell(\cdot,-\xi)\, e^{i\theta\cdot\kappa_n}e^{i\xi\tau_n}\hat g_{\ell'}\left(T^n\cdot,\xi\right)\right]
$$
\begin{equation}\label{key2}
=\mathbb E_{\bar\nu}\left[\bar P_{\theta,\xi}^{n}(\bar F_{t,-\theta}(\cdot,\ell,-\xi))\bar G_{t,\theta}(\cdot,\ell',\xi)\right]+O\left(\vartheta^{k_t}
d_{\ell,\ell'}(\xi)
\right)\, ,
\end{equation}
where
\begin{gather*}\bar F_{t,-\theta}(x,\ell,-\xi):=F_t(x,\ell,-\xi) e^{-i\theta\bar \kappa_{k_t-m_0}(x)}e^{-i\xi\bar \tau_{k_t}(x)} \\
\bar G_{t,\theta}(x,\ell',\xi):= G_{t}(x,\ell',\xi)e^{i\theta\bar \kappa_{k_t-m_0}(x)}e^{i\xi\bar \tau_{k_t}(x)}.
\end{gather*}

Due to \eqref{CCC0a} and \eqref{CCCa},
substituting \eqref{key2} into
 \eqref{hypeq1}  yields
\begin{eqnarray}
C_t(f,g)
&=&   \frac 1{(2\pi)^{d+1}}\sum_{\ell,\ell'}
\sum_{n=t_-}
^{ t_+} \int_{[-\pi,\pi]^d\times \mathbb R}
\left(e^{-i\xi t}
\, e^{-i\theta\cdot(\ell'-\ell)}\right.
\nonumber\\
&\ &\ \ \left.\mathbb E_{\bar\nu}\left[\bar P_{\theta,\xi}^{n-2k_t}\left(\bar P_{\theta,\xi}^{2k_t} \bar F_{t,-\theta}(\cdot,\ell,-\xi)\right)\bar G_{t,\theta}(\cdot,\ell',\xi)\right]\right)\, d\theta d\xi+
O(\vartheta^{k_t}). \label{FFF1}
\end{eqnarray}
Note that \eqref{FFF1} is the analogue of \eqref{EE1} (with $(M,\nu)$, $P_{\theta,\xi}^n$, 
 $\hat f_\ell(\cdot,-\xi)$ and $\hat g_{\ell'}(\cdot,\xi)$ being replaced by $(\bar\Delta,\bar\nu)$,
 $\bar P_{\theta,\xi}^{n-2k_t}$, $\bar P_{\theta,\xi}^{2k_t} \bar F_{t,-\theta}(\cdot,\ell,-\xi)$ 
 and $\bar G_{t,\theta}(\cdot,\ell',\xi)$, respectively).

 Due to \eqref{approxiB} and \eqref{approxiC}
$$\Vert \bar P_{\theta,\xi}^{2k_t}\bar F_{t,-\theta}(\cdot,\ell,-\xi)\Vert_{\mathcal B}+\Vert \bar P_{\theta,\xi}^{2k_t}\bar F_{t,-\theta}(\cdot,\ell,-\xi)\Vert_{\mathbb B}\le 2C_0
(1 + |\xi|)
\Vert e^{i\xi\chi(\cdot)}\hat f_\ell(\cdot,-\xi))\Vert_{\mathcal V}. $$
Next, we estimate
\begin{eqnarray*}\Vert \bar G_{t,\theta}(\cdot,\ell,\xi)\Vert_{\mathcal B'}
&\le&
\Vert \bar G_{t,\theta}(\cdot,\ell',\xi)\Vert_{\infty} \\
&\le& 
\Vert e^{-i\xi\chi(\cdot)}\hat g_{\ell'}(\cdot,\xi)\Vert_{\infty} +  
\Vert e^{-i\xi\chi \circ T^n}\hat g_{\ell'}(T^n(\cdot),\xi) - \boldsymbol{\Pi}_{k_t}(e^{-i\xi\chi(\cdot)}\hat g_{
\ell'}(\cdot, \xi))\circ \mathfrak p\Vert_{\infty}\\
&\le&
(1+C_0)\Vert e^{-i\xi\chi(\cdot)}\hat g_{\ell'}(\cdot,\xi)\Vert_{\mathcal V},
\end{eqnarray*}
where we used the fact that $L^{\infty}$ is continuously embedded into $\mathcal B'$ in the first line,  
the definition of $G_t$ and the triangle inequality in the second one and \eqref{approxiA} 
and $\mathcal V\hookrightarrow L^\infty(\nu)$
 in the third one.
Therefore, due to \eqref{CCCa},
$$\forall\gamma>0,\quad \sum_{\ell,\ell'\in\mathbb Z^d}\Vert \bar P_{\theta,\xi}^{2k_t} \bar F_{t,-\theta}(\cdot,\ell,-\xi)\Vert_{\mathbb B} \Vert \bar G_{t,\theta}(\cdot,\ell',\xi)\Vert_{\infty}=O(|\xi|^{-\gamma})\, .$$

Hence, proceeding 
as in Step 2
of the proof of Theorem \ref{THMGENE} we  obtain that
\begin{eqnarray}
C_t(f,g)&\simeq &   
\frac 1{(2\pi)^{d+1}}\sum_{\ell,\ell'}
\sum_{n= t_-}
^{t_+}\int_{[-b,b]^{d+1}}e^{-i\xi t}
\, e^{-i\theta\cdot(\ell'-\ell)}\nonumber\\
&\ &\ \ \ \mathbb E_{\bar\nu}\left[\bar P_{\theta,\xi}^{n-2k_t}\left(\bar P_{\theta,\xi}^{2k_t} \bar F_{t,-\theta}(\cdot,\ell,-\xi)\right)\bar G_{t,\theta}(\cdot,\ell',\xi)\right]\, d\theta d\xi\, .\label{FFF2}
\end{eqnarray}
Using \eqref{key2} again we obtain
\begin{eqnarray}
C_t(f,g)&\simeq&   
\frac 1{(2\pi)^{d+1}}\sum_{\ell,\ell'}
\sum_{n = t_-}
^{t_+}\int_{[-b,b]^{d+1}}e^{-i\xi t}
\, e^{-i\theta\cdot(\ell'-\ell)}\nonumber\\&\ &\ \ \ 
\mathbb E_{\nu}\left[
\hat f_\ell(\cdot,-\xi)\, e^{i\theta\cdot\kappa_n}e^{i\xi\tau_n}\hat g_{\ell'}\left(T^n\cdot,\xi\right)\right]\, d\theta d\xi
\, .\label{FFF2b}
\end{eqnarray}

Moreover, for every $(\theta,\xi)\in[-b,b]^{d+1}$ and every integer $n$ satisfying 
$t_- \le n\le t_+$, using Taylor expansion,
the following holds true
\begin{eqnarray}
&\ &\mathbb E_{\nu}\left[
\hat f_\ell(\cdot,-\xi)\, e^{i\theta\cdot\kappa_n}e^{i\xi\tau_n}\hat g_{\ell'}\left(T^n\cdot,\xi\right)\right]
\lambda_{\theta,\xi}^{-n}\nonumber\\
&=&\sum_{N=0}^{
L-1
}\frac 1{N!}  \left(\mathbb E_{\nu}\left[
\hat f_\ell(\cdot,-\xi)\, e^{i\theta\cdot\kappa_n}e^{i\xi\tau_n}\hat g_{\ell'}\left(T^n\cdot,\xi\right)\right]\lambda_{\theta,\xi}^{-n}\right)^{(N)}_{|(\theta,\xi)=0}*{(\theta,\xi)}^{\otimes N}\nonumber\\
&+&\hskip-4mmO\left(\sup_{u \in [0,1], (\theta',\xi') = (u\theta,u\xi)}\left(
\frac{\mathbb E_{\nu}\left[
\hat f_\ell(\cdot,-\xi')\, e^{i\theta'\cdot\kappa_n}e^{i\xi'\tau_n}\hat g_{\ell'}\left(T^n\cdot,\xi'\right)\right]}{\lambda_{\theta,\xi}^n}\right)^{
(L)
}_{|(\theta',\xi')}
|(\theta,\xi)|^{
L}\right)\hskip-1mm.\label{DDD1}
\end{eqnarray}

Let us study the derivatives involved in this formula. First,
since $\boldsymbol{\Pi}_{k_t}$ is linear and continuous, for every $m=0,...,L$, we have
\begin{equation}\label{PiKJets0}
\left(\boldsymbol{\Pi}_{k_t}\left(e^{-i\xi\chi}\hat h_\ell(\cdot,\theta,\xi)\right)\right)^{(m)}_{|(\theta,\xi)}
=\boldsymbol{\Pi}_{k_t}\left(\left(e^{-i\xi\chi}\hat h_\ell(\cdot,\theta,\xi)\right)^{(m)}_{|(\theta,\xi)}\right)\, .
\end{equation}
Using \eqref{PiKJets0} and \eqref{FFFa} we obtain the following analogue of \eqref{key2},
$$
\left|\left(\mathbb E_{\nu}\left[
\hat f_\ell(\cdot,-\xi)\, e^{i\theta\cdot\kappa_n}e^{i\xi\tau_n}\hat g_{\ell'}\left(T^n\cdot,\xi\right)\right]\lambda_{\theta,\xi}^{-n}\right)_{(\theta,\xi)}^{(L)}\right|= $$
\begin{equation}
\left(\mathbb E_{\bar\nu}\left[\bar P_{\theta,\xi}^{n-2k_t}\left(\bar P_{\theta,\xi}^{2k_t}\left(\bar F_{t,-\theta}(\cdot,\ell,-\xi)\right)\right)\bar G_{t,\theta}(\cdot,\ell',\xi)\right]\lambda_{\theta,\xi}^{-n}\right)^{(L)}_{(\theta,\xi)}
\label{DDD2} 
 +O\left(\vartheta^{k_t} n^{L}
\widetilde d_{\ell,\ell'}(\xi)
\left\vert\lambda_{\theta,\xi}^{-n}\right\vert\right)
\end{equation}
with 
$\widetilde d_{\ell,\ell'}(\xi)
:=\sup_{m,m'=0,...,L}\left(\left\Vert\frac{\partial^m}{\partial \xi^m}\left( e^{i\xi.\chi}\hat f_\ell(\cdot,-\xi)\right)\right\Vert_{\mathcal V} \, \left\Vert \frac{\partial^{m'}}{\partial \xi^{m'}}\left( e^{
-i\xi.\chi} \hat g_{\ell'}(\cdot,\xi)\right)\right\Vert_{\mathcal V}\right)$.

Using \eqref{decomp2b}, \eqref{approxiC}, \eqref{approxiD}, we find that the first term of \eqref{DDD2} is bounded from above by
$$C_0^2(1+|\xi|)\sup_{m=0,...,L} k_t^{m} \widetilde d_{\ell,\ell'}(\xi)
\left\Vert \left((R_{\theta,\xi}^{n-2k_t}/\lambda^n_{\theta,\xi})+\lambda^{-2k_t}_{\theta,\xi})\Pi_{\theta,\xi}\right)^{(L-m)}_{(\theta,\xi)} \right\Vert_{\mathcal L(\mathcal B,\mathcal B)}\, ,
$$
which is in $O\left(k_t^{L} \widetilde d_{\ell,\ell'}(\xi)\left(\frac{\vartheta^{n-2k}}{\vartheta^{\frac n{10}}}+ \vartheta^{-\frac{k_t}{5L(d+1)}}\right)\right)$.
This observation, combined with \eqref{DDD1}, \eqref{DDD2} and
our choice of $k_t$
yields
\begin{eqnarray}
&\ &\mathbb E_{\nu}\left[
\hat f_\ell(\cdot,-\xi)\, e^{i\theta\cdot\kappa_n}e^{i\xi\tau_n}\hat g_{\ell'}\left(T^n\cdot,\xi\right)\right]
\lambda_{\theta,\xi}^{-n}\label{eq:thm2prod} \\
&=&
\sum_{N=0}^{L-1}\frac 1{N!}  \left(\mathbb E_{\nu}\left[
\hat f_\ell(\cdot,-\xi)\, e^{i\theta\cdot\kappa_n}e^{i\xi\tau_n}\hat g_{\ell'}\left(T^n\cdot,\xi\right)\right]\lambda_{\theta,\xi}^{-n}\right)^{(N)}_{|(\theta,\xi)=0}*{(\theta,\xi)}^{\otimes N}\nonumber 
+ O\left( n^{\frac {2}5} \widetilde d_{\ell,\ell'}(\xi) |(\theta,\xi)|^{L}\right)\\
&\ &
+O\left(n^{-\frac{L+1+d}2}
\widetilde d_{\ell,\ell'}(\xi)
\left\vert\lambda_{\theta,\xi}^{-n}\right\vert\right)
  \, ,
\end{eqnarray}
for $(\theta,\xi)\in[-b,b]^{d+1}$.

Now we apply Lemma \ref{technical} to conclude that \eqref{eq:thm2prod} is equal to 
\begin{equation}
\label{lemma1.6app1}
\sum_{N=0}^{L-1}\frac 1{N!} \mathcal A_N\left(\hat f_\ell, \hat g_{\ell'}\right)*(\xi,\theta)^{\otimes N}
+ O\left(
\widetilde d_{\ell,\ell'}(\xi)
\left(n^{-\frac{K+d+1}2}+
n^{\frac{2}5}
|(\theta,\xi)|^{L}
+n^{-\frac{L+1+d}2}
\left\vert\lambda_{\theta,\xi}^{-n}\right\vert
\right)\right)\, .
\end{equation}
Recalling the notation $a_s:=e^{-\frac 1 2\Sigma*s^{\otimes 2}}$ and Lemma \ref{POLY}, we have
\begin{eqnarray}
\lambda_{s}^n
& = & \label{lemma1.6app2}
e^{in\xi\nu(\tau)}a_{s\sqrt{n}}\sum_{k=0}^{\lfloor (K+1)/(J-2)\rfloor}\sum_{j=kJ}^{K+1+2k}
n^kA_{j,k}*s^{\otimes j}\\
&+& O\left(
a_{s\sqrt{n}/\sqrt{2}}\, n^{-\frac{K+1}2}(1+
|s\sqrt{n}|^{K_0}
)\eta(s)\, \right)\, , \nonumber
\end{eqnarray}
where $\DS \lim_{s\rightarrow 0}\eta(s)=0$.
Note that 
the
modulus of
{the dominating term of 
 \eqref{lemma1.6app1} is bounded by 
 $ O\left(
\widetilde d_{\ell,\ell'}(\xi)
\right)$
uniformly in $(\theta,\xi)\in [-b,b]^{d+1}$
and that
the modulus of
$\lambda_s^n$ in \eqref{lemma1.6app2} is bounded by $ O(a_{s\sqrt{n}/\sqrt{2}})$ 
(the first one follows from Lemma
\ref{technical},
the second one follows from \eqref{condb2}).
Thus multiplying
\eqref{lemma1.6app1} and \eqref{lemma1.6app2}
we conclude
\begin{eqnarray}
&\ &\mathbb E_{\nu}\left[
\hat f_\ell(\cdot,-\xi)\, e^{i\theta\cdot\kappa_n}e^{i\xi\tau_n}\hat g_{\ell'}\left(T^n\cdot,\xi\right)\right]
\label{DDD0}
  \\
&=&\sum_{N=0}^{L-1}\sum_{k=0}^{\lfloor (K+1)/(J-2)\rfloor}\sum_{j=kJ}^{K+1+2k} \frac {e^{in\xi\nu(\tau)}a_{s\sqrt{n}} n^k}{N!}\left( \mathcal A_N\left(\hat f_\ell, \hat g_{\ell'}\right)\otimes A_{j,k}\right)*s^{\otimes (N+j)}\nonumber\\
&+& O\left(|\lambda_s^n|
\widetilde d_{\ell,\ell'}(\xi)
\left(n^{-\frac{K+d+1}2}+
n^{\frac{2}5}
|s|^{L}
+n^{-\frac{L+1+d}2}
\left\vert\lambda_{s}^{-n}\right\vert
\right)\right)\nonumber\\
&+&O\left(\sum_{N=0}^{L-1}\frac 1{N!} \mathcal A_N\left(\hat f_\ell, \hat g_{\ell'}\right)*s^{\otimes N}
a_{s\sqrt{n}/\sqrt{2}}\, n^{-\frac{K+1}2}(1+
|s\sqrt{n}|^{K_0}
)\eta(s)\right)\nonumber
\end{eqnarray}
where $s=(\theta,\xi)$.
This leads to the following error term
\begin{eqnarray}
&\ &O\left(
\widetilde d_{\ell,\ell'}(\xi)
\left(a_{s\sqrt{n}/\sqrt{2}}\left(n^{-\frac{K+d+1}2}+
n^{\frac{2}5}
|s|^{L}
\right)+n^{-\frac{L+1+d}2}
\right)\right) \nonumber \\
&+&O\left(\widetilde d_{\ell,\ell'}(\xi)
a_{s\sqrt{n}/\sqrt{2}}\, n^{-\frac{K+1}2}(1+
|s\sqrt{n}|^{K_0}
)\eta(s)\right)\;\; \label{DDD0bis}
\\
&=& O\left(  
\widetilde d_{\ell,\ell'}(\xi)\, 
\left(
n^{-\frac{L+1+d}2}+
a_{s\sqrt{n}/\sqrt{2}}\left(n^{-\frac{K+d+1}2}+
n^{\frac{2}5}
|s|^{L}+
n^{-\frac{K+1}2}\left(1+
|s\sqrt{n}|^{K_0}
\right
)\eta(s)\right)\right) 
\right)\, , \nonumber
\end{eqnarray}
Observe that
\begin{eqnarray*}
&\ &\int_{\mathbb R^{d+1}} a_{s\sqrt{n}/\sqrt 2}\left(n^{-\frac {K+d+1}2}+n^{\frac{2}5}
|s|^{L}+
n^{-\frac{K+1}2}(1+|s\sqrt{n}|^{K_0})\eta(s)\right)\, ds\\
&=& n^{-\frac{d+1}2}\int_{\mathbb R^{d+1}} a_{s/\sqrt 2}\left(n^{-\frac {K+d+1}2}+n^{
\frac{2}5
-\frac{L}2}|s|^{L}+n^{-\frac{K+1}2}
(1+|s|^{K_0})\eta(s/\sqrt n)\right)\, ds\\
&=&o\left(n^{-\frac{K+2+d}2}  \right)\, .
\end{eqnarray*}
Therefore \eqref{CCC0a}, \eqref{FFF2b} and \eqref{DDD0}, 
\eqref{DDD0bis}
}
imply
\begin{equation}
\label{defI}
C_t(f,g)
\simeq 
\frac 1{(2\pi)^{d+1}}\sum_{N=0}^{
{
L-1
}
}\frac 1{N!}\sum_{k=0}^{\lfloor (K+1)/(J-2)\rfloor}\sum_{j=kJ}^{K+1+2k} 
\sum_{\ell,\ell'}
\sum_{n= t_-}
^{t_+} \mathcal I^{N, k, j}_{ \ell, \ell', n},
\end{equation}
where 
$$
\mathcal I^{N, k, j}_{ \ell, \ell', n} = n^k \int_{[-b,b]^{d+1}}e^{-i\xi( t-n\nu(\tau))}
\, e^{-i\theta\cdot(\ell'-\ell)}
\left(\mathcal A_N\left(\hat f_\ell, \hat g_{\ell'}\right)\otimes A_{j,k}\right)*(\theta,\xi)^{\otimes (N+j)}
a_{\sqrt{n}(\theta,\xi)} \, d\theta d\xi\, .
$$
By changing variables, we see that
$$
\mathcal I^{N, k, j}_{ \ell, \ell', n}
= n^{-\frac{d+1+N+j-2k}2}
\int_{[-b\sqrt{n},b\sqrt{n}]^{d+1}}
\left( \mathcal A_N\left(\hat f_\ell, \hat g_{\ell'}\right)\otimes  A_{j,k}\right)
*e^{-i\frac{\xi (t-n\nu(\tau))}{\sqrt{n}}}
\, e^{-i \frac{\theta\cdot(\ell'-\ell)}{\sqrt n}}
(\theta,\xi)^{\otimes (N+j)}
a_{\theta,\xi}\, d\theta d\xi. 
$$
At first sight, this expression looks simpler than \eqref{eq:conv}
since $\mathcal A_N\left(\hat f_\ell, \hat g_{\ell'}\right)$ 
does not depend on $\xi$ and so no convolution is involved when taking 
the inverse Fourier transform.
Namely we obtain
\begin{gather}
\label{eq:mathcalI}
\mathcal I^{N, k, j}_{ \ell, \ell', n} \approx 
(2\pi)^{d+1}
n^{-\frac{d+1+N+j-2k}2} i^{N+j} 
\Psi^{(N+j)}\left(\frac{\ell'-\ell}{\sqrt{n}},\frac{t-n\nu(\tau)}{\sqrt{n}}\right)
*\left(
\mathcal A_N\left(\hat f_\ell, \hat g_{\ell'}\right)\otimes  A_{j,k}
 \right),
\end{gather}
where $\mathcal I \approx \mathcal I'$ means that \eqref{defI} holds for $\mathcal I$ and $\mathcal I'$ at
the same time
(i.e. the difference obtained when substituting $\mathcal I$ and $\mathcal I'$ to \eqref{defI} is in
$o\left(t^{-\frac{K+d}2}\right)$).
Now recall the definition $\mathcal B_N$ from \eqref{defBn}. 
Note that the difference between $\mathcal A_N$ and $\mathcal B_N$ is that the latter 
one is defined for function that do not depend on $\xi$. Thus
\begin{equation}
\label{eq:An1}
 \mathcal A_N\left(\hat f_\ell, \hat g_{\ell'}\right)= 
\sum_{m_1+m_2+m_3=N}\frac{N!}{m_1!m_2!m_3!}
(-1)^{m_1}\mathcal B_{m_2}\left((\hat f(.,\ell,\xi))^{(m_1)}_{|\xi=0},(\hat g(.,\ell,\xi))^{(m_3)}_{|\xi=0}\right).
\end{equation}
Note that
$$
(\hat f(x,\ell,\xi))^{(m_1)}_{|\xi=0} (\hat g(y,\ell,\xi))^{(m_3)}_{|\xi=0}
=  \int_{\mathbb R^2}(iu)^{m_1}(iv)^{m_3} f(x,\ell,u) g(y,\ell,v) dudv.
$$
Thus \eqref{eq:An1} is equal to
$$
\sum_{m_1+m_2+m_3=N}\frac{N!}{m_1!m_2!m_3!}
\int_{\mathbb R^2}(0,-iu)^{\otimes m_1}\otimes(0,iv)^{\otimes m_3}\otimes \mathcal B_{m_2}\left(f(\cdot,\ell,u), g(\cdot,\ell,v)\right)\, dudv . 
$$
Now using the binomial theorem, we find that \eqref{eq:An1} is equal to
$$
 \sum_{m=0}^N \frac{N!}{m! (N-m)!}
\int_{\mathbb R^2}(0,i(v-u))^{\otimes N-m}\otimes \mathcal B_{m}\left(f(\cdot,\ell,u), g(\cdot,\ell,v)\right)dudv.
$$
Substituting this into \eqref{eq:mathcalI} and using \eqref{defI} and the identity 
$(-1)^{N-m}i^{N+N-m} = 
 i^{m}$, we find
\begin{gather*}
C_t(f,g)
\simeq  
\sum_{N=0}^
{L-1}
\sum_{k=0}^{\lfloor (K+1)/(J-2)\rfloor}\sum_{j=kJ}^{K+1+2k} 
\sum_{\ell,\ell'}
 \sum_{m=0}^N 
 \sum_{n= t_-}^{t_+}
 \frac{1}{m!(N-m)!} i^{m+j} n^{-\frac{d+1+N+j-2k}2}  \\
\Psi^{(N+j)}\left(\frac{\ell'-\ell}{\sqrt{n}},\frac{t-n\nu(\tau)}{\sqrt{n}}\right)
*\left(
\int_{\mathbb R^2}(0,u-v)^{\otimes N-m}\otimes \mathcal B_{m}\left(f(\cdot,\ell,u), g(\cdot,\ell,v)\right)dudv\otimes  A_{j,k}
 \right).
\end{gather*}
Now proceeding as in Step 5 of the proof of Theorem \ref{THMGENE}
we find
\begin{gather*}
C_t(f,g)
\simeq 
\sum_{N=0}^{
L-1}\sum_{k=0}^{\lfloor (K+1)/(J-2)\rfloor}\sum_{j=kJ}^{K+1+2k} 
\sum_{\ell,\ell'} \sum_{m=0}^N \sum_{r=0}^{K-N-j+2k}
\sum_{n= t_-}^{t_+}
\frac{i^{m+j}}{m!(N-m)!r! n^{\frac{d+1+N+j +r -2k}2} } 
\\
\Psi^{(N+j+r)}\left(0,\frac{t-n\nu(\tau)}{\sqrt{n}}\right)
*(\ell' - \ell)^{\otimes r} \left(
\int_{\mathbb R^2}(0,u-v)^{\otimes N-m}\otimes \mathcal B_{m}\left(f(\cdot,\ell,u), g(\cdot,\ell,v)\right)dudv\otimes  A_{j,k}
\right)\, .
\end{gather*}

Performing summation over $n$ as in Step 6 of the proof of Theorem \ref{THMGENE}
(using again Lemma \ref{sumint}),
we derive
\begin{gather*}
C_t(f,g)
\simeq 
\sum_{N=0}^{
K}\sum_{k=0}^{\lfloor K/(J-2)\rfloor}\sum_{j=kJ}^{K+1+2k} 
\sum_{\ell,\ell'} \sum_{m=0}^N \sum_{r=0}^{K-N-j+2k} \sum_{q=0}^{K+2k-N-j-r}
\frac{1}{m!(N-m)!r!q!} i^{m+j} \\
\frac{(t/\nu(\tau))^{-\frac{d+N+j +r +q -2k}2} }{(\nu (\tau))^{\frac{q+1}{2}}}
\int_{\mathbb R}   \partial_2^qh_{N+j+r,k-\frac{N+j+d+r+1}2}(s,1)(-s)^q\, ds
\\
*(\ell' - \ell)^{\otimes r} \left(
\int_{\mathbb R^2}(0,u-v)^{\otimes N-m}\otimes \mathcal B_{m}\left(f(\cdot,\ell,u), g(\cdot,\ell,v)\right)dudv\otimes  A_{j,k}
 \right).
\end{gather*}
We will set $R = N-m+r$. The binomial theorem tells us that,
$m,j,k$ being fixed, for every $R=0,...,K-m-j+2k$, the following identity holds true
$$
\sum_{
(r,N)\, :\, N-m+r=R
} \frac{R!}{(N-m)! r!}  (\ell' - \ell)^{\otimes r} \otimes (0,u-v)^{\otimes N-m} = (\ell'-\ell,u-v)^{\otimes { R}}\, .
$$
We conclude that
\begin{equation*}
C_t(f,g)\simeq
\sum_{\ell,\ell'} 
\sum_{m=0}^{K}\sum_{k=0}^{\lfloor K/(J-2)\rfloor}\sum_{j=kJ}^{K-m+2k}\sum_{R=0}^{K-m-j+
2k}
\sum_{q=0}^{K+2k-m-j- R} \frac{i^{m+j}(t/\nu(\tau))^{-\frac {m+j+d+ R+q-2k}2}}{R! m!\, q!\, (\nu(\tau))^{\frac {q+1}2}}
\end{equation*}
$$
\int_{\mathbb R}   \partial_2^qh_{m+j+ R,k-\frac{m+j+d+R+1}2}(s,1)(-s)^q\, ds $$
$$*\left(\int_{\mathbb R^2}\mathcal B_m\left(f_{\ell}(\cdot,u),g_{\ell'}(\cdot,v)\right)\, \otimes(\ell'-\ell,u-v)^{\otimes  R}\, dudv\otimes A_{j,k}\right)\, .
$$
This implies the theorem.
\end{proof}


\begin{proof}[Proof of Lemma \ref{technical}]
Let $N\in\{0,...,L\}$ be fixed.
Let us prove that, for every $N$, 
$$\left(\mathcal A_{N,n}(u,v):=\left(\mathbb E_{\nu}\left[u(\cdot,-\theta,-\xi) e^{i\theta \cdot\kappa_n+i\xi\tau_n}v( T^n(\cdot),\theta,\xi)\right]\lambda_{(\theta,\xi)}^{-n}\right)_{|(\theta,\xi)=0}^{(N)}\right)_n$$
is a Cauchy sequence.
Observe that \eqref{FFFa} is valid  
with $k_t$ replaced by any
integer $k$ such that $m_0\le k\le n$. That is, for such $k$ we have
\begin{eqnarray*}
\mathcal A_{N,n}(u,v)&=&\left(\mathbb E_{\nu}\left[\left(e^{i\xi\chi\circ T^k}u(T^k(\cdot),-\theta,-\xi) e^{-i\theta\bar \kappa_{k-m_0}\circ\mathfrak p-i\xi\bar \tau_{k}\circ\mathfrak p}\right) e^{i\theta \cdot\bar\kappa_{n}\circ \mathfrak p+i\xi\bar\tau_n\circ\mathfrak p}\right.\right.\\
&\ &\ \ \ \ \left.\left. e^{i\theta\bar \kappa_{k-m_0}\circ \bar T^n\circ\mathfrak p+i\xi\bar \tau_{k}\circ \bar T^n\circ \mathfrak p}e^{-i\xi\chi\circ T^{n+k}}v( T^{n+k}(\cdot),\theta,\xi)\right]\lambda_{(\theta,\xi)}^{-n}\right)_{|(\theta,\xi)=0}^{(N)}.
\end{eqnarray*}
Thus,
we obtain
\begin{equation}\label{ANuv}
{\mathcal A}_{N,n}(u,v)=\widetilde {\mathcal A}_{N,n}\left(\tilde U_k,\tilde V_k\right),
\end{equation}
where
$$ \widetilde {\mathcal A}_{N,n}(U,V)=\left(\mathbb E_\nu\left[U(\cdot,-\theta,-\xi) e^{i\theta \cdot\bar\kappa_{n}\circ \mathfrak p+i\xi\bar\tau_n\circ\mathfrak p}V(T^n(\cdot),\theta,\xi)\right]\lambda_{(\theta,\xi)}^{-n}\right)_{|(\theta,\xi)=0}^{(N)},$$
$$
\tilde U_k(\cdot,\theta,\xi):=(e^{-{i}\xi\chi}u(\cdot,\theta,\xi))\circ T^k .
 e^{i(\theta\cdot\bar \kappa_{k-m_0}+\xi\bar \tau_{k})\circ\mathfrak p}\, ,
$$
and
$$\tilde V_{k}(\cdot,\theta,\xi):=(e^{{-i}\xi.\chi}v(\cdot,\theta,\xi))\circ T^k .e^{i(\theta\cdot\bar \kappa_{k-m_0}+\xi\cdot\bar \tau_{k})\circ\mathfrak p}\, .$$
Recall \eqref{approxiA} and denote
$$
U_{k}(\cdot,\theta,\xi):=\boldsymbol{\Pi}_{k}(e^{-{ i}\xi\chi}u(\cdot,\theta,\xi)). 
e^{{i}(\theta\cdot\bar \kappa_{k-m_0}+\xi\bar \tau_{k})}
\quad\mbox{and}\quad
V_{k}(\cdot,\theta,\xi):=\boldsymbol{\Pi}_{k}(e^{-i\xi\chi}v(\cdot,\theta,\xi)) .e^{i(\theta\cdot\bar \kappa_{k-m_0}
+\xi\bar \tau_{k})}\, .$$
Since $\boldsymbol{\Pi}_{k}$ is linear
and continuous and since
$(\theta,\xi)\mapsto e^{-i\xi.\chi}u(\cdot,\theta,\xi)$
is $L$ times differentiable at 0 as a $\mathcal V$-valued function, for every $m=0,...,L$, we have
\begin{equation}
  \label{PiKJets}
\left(\boldsymbol{\Pi}_{k}\left(e^{-i\xi\chi}u(\cdot,\theta,\xi)\right)\right)^{(m)}_{|(\theta,\xi)=0}
=\boldsymbol{\Pi}_{k}\left(\left(e^{-i\xi\chi}u(\cdot,\theta,\xi)\right)^{(m)}_{|(\theta,\xi)=0}\right)\, .
\end{equation}

Thus
$$
\left\Vert 
\left(e^{-i\xi\chi\circ T^{k}}u(T^{k}(\cdot),\theta,\xi)\right)^{(m)}_{|(\theta,\xi)=0}-  
\left(\boldsymbol{\Pi}_{k}(e^{-i\xi\chi}u(\cdot,\theta,\xi))\right)^{(m)}_{|(\theta,\xi)=0} 
\circ \mathfrak p
\right\Vert_\infty $$
\begin{equation} 
 \leq 
 C_0{\vartheta}^{k}\left\Vert \left(e^{-i\xi\chi}
u(\cdot,\theta,\xi)\right)^{(m)}_{|(\theta,\xi)=0}\right\Vert_{\mathcal V}
\leq  C_0 \vartheta^k \|u\|_{\mathcal{W},+}\, , \label{PlusNorm}
\end{equation}
and idem by replacing $u$ by $v$ (and $i$ by $-i$).
Next, observe that
\begin{equation}
\label{ExpDer}
\left\Vert  \bar\tau_n^m+|\bar\kappa_n|^m \right\Vert_\infty+\left|\left(\lambda^{-n}\right)^{(m)}_{|(\theta,\xi)=0}\right|= O(n^m).
\end{equation}
Combining \eqref{PiKJets}, \eqref{PlusNorm}, and \eqref{ExpDer} we obtain

\begin{eqnarray}
  &\ &\mathcal A_{N,n}(u,v)-\widetilde{\mathcal A}_{N,n}({U_k\circ\mathfrak p, V_k\circ\mathfrak p})=
  \widetilde{\mathcal A}_{N,n}(\widetilde{U}_k ,\widetilde{V}_k)-\widetilde{\mathcal A}_{N,n}({U_k\circ\mathfrak p, V_k\circ\mathfrak p})
  \label{AOneSide}
  \\
&=&\left(\mathbb E_{\nu}\left[
 e^{i\theta\cdot\kappa_n}e^{i\xi\tau_n}
\left(\tilde U_k\left(\cdot,-\theta,-\xi\right)\tilde V_k\left(T^n(\cdot),\theta,\xi\right)
-  U_k\left(\mathfrak p(\cdot),-\theta,-\xi\right) V_k\left(\mathfrak p(T^n(\cdot)),\theta,\xi\right)\right)\right]\lambda_{\theta,\xi}^{-n}\right)_{|(\theta,\xi)=0}^{(N)}
  \nonumber \\
&=& O\left(n^N \vartheta^k
\| u\|_{\mathcal W, +} \| v\|_{\mathcal W, -}
\right) \, . \nonumber
\end{eqnarray}

Let $k_n:=\lceil {\log ^2 n}\rceil$.
Take $n'\in [n, 2n].$
Using \eqref{AOneSide} we obtain
\begin{align*}
 \left\vert {\mathcal A}_{N,n}(u,v)\right.&\left.-{\mathcal A}_{N,n'}(u,v)\right\vert
\\
&\le\left\vert \widetilde {\mathcal A}_{N,n}
( U_{k_n}\circ \mathfrak p,V_{k_n}\circ\mathfrak p)
-\widetilde {\mathcal A}_{N,n'}( U_{k_n}\circ \mathfrak p,V_{k_n}\circ\mathfrak p)\right\vert+
O\left({n^{N}\Vert u\Vert_{\mathcal W, +}\, \Vert v\Vert_{\mathcal W, -}
\vartheta^{k_n}
}
\right).
\end{align*}
The main term on the RHS equals to 
\begin{equation}\label{TwoProj}
  \mathbb E_{\bar \nu}\left[\left((\lambda_t^{-n}\bar P_t^{n-2k_n}-\lambda_t^{-n'}\bar P_t^{n'-2k_n})\left(\bar P_t ^{2k_n}\left(U_{k_n}(\cdot,
-t
)\right)\right)V_{k_n}(\cdot,t)\right)_{|t=0}^{(N)}\right].
\end{equation}  
Since $\lambda_t^{- \tilde n}\bar P_t^{\tilde n-2k_n}=\lambda_t^{-2k_n}\Pi_t+\lambda_t^{-\tilde n}R_t^{
\tilde n
-2k_n}$
we can use the definition of $\mathcal B'$ to bound {\eqref{TwoProj}} by 
\begin{align*}
&\left \| \left((\lambda_t^{-n} R_t^{n-2k_n}-\lambda_t^{-n'} R_t^{n'-2k_n})\left(\bar P_t ^{2k_n}\left(U_{k_n}(\cdot,
-t
)\right)\right)V_{k_n}(\cdot,t)\right)_{|t=0}^{(N)} \right \|_{L^1(\bar \nu)} \leq \\
\leq & C_N 
{
\max_{n'\in[n,2n],1 \leq m_1 \leq N} (\lambda_t^{-n'})^{(m_1)}_{|t=0} 
}
\left( \max_{1 \leq m_2 \leq N} \left \|  (R_t^{n-2k_n})^{(m_2)}_{|t = 0}  \right \|_{\mathcal L (\mathcal B, \mathcal B)}+ 
 \max_{1 \leq m_2 \leq N} \left \| (R_t^{n'-2k_n})^{(m_2)}_{|t = 0}  \right \|_{\mathcal L (\mathcal B, \mathcal B)} \right)\\
& \times \left \| 
\max_{1 \leq m_3 \leq N} \left(\bar P_t ^{2k_n}\left(U_{k_n}(\cdot,
-t
)\right)\right)^{(m_3)}_{|t=0}
\right \|_{\mathcal B} 
 \left \| 
\max_{1 \leq m_4 \leq N} V_{k_n}(\cdot,t)_{|t=0}^{(m_4)} \right \|_{\mathcal B'} \, .
\end{align*}
Now observe that the $\max$ over $m_2$ is bounded by $O(\vartheta^{n/2})$ by \eqref{majoexpob} and the other terms
cannot grow faster than a polynomial in $n$. 
In particular, we use \eqref{approxiC} to bound the max over $m_3$ and \eqref{approxiD} to bound the max over $m_4$.
We conclude that \eqref{TwoProj} is exponentially small. 

  Therefore,  for each $\bar L\in\mathbb{N}$ we have
\begin{eqnarray*}
\sup_{\bar n\ge 0}\left|{\mathcal A}_{N,n}(u,v)-{\mathcal A}_{N,n+\bar n}(u,v)\right|
&\le&  \sum_{p\ge 0} \sup_{\bar n=0,...,2^pn}\left|{\mathcal A}_{N,2^pn}(u,v)-{\mathcal A}_{N,2^pn+\bar n}(u,v)\right|\\
\le \left( \sum_{p\ge 0}(2^pn)^{- \bar L}
\Vert u\Vert_{\mathcal W, +}\, \Vert v\Vert_{\mathcal W, -} \right)
&=&O\left(
\Vert u\Vert_{\mathcal W, +}\, \Vert v\Vert_{\mathcal W, -}  \; n^{- \bar L}\right)\, .
\end{eqnarray*}
Hence ${\mathcal A}_N(u,v)$ is well defined and satisfies
$$\left|{\mathcal A}_{N,n}(u,v)-{\mathcal A}_N(u,v)\right|=O\left(
\Vert u\Vert_{\mathcal W, +}\, \Vert v\Vert_{\mathcal W, -} 
\; n^{- \bar L}\right)\, .
\qedhere $$
\end{proof}


\section{Mixing expansion for the Sinai billiard flow}
\label{sec:mixbilliard}

\subsection{Sinai billiards}
\label{sec:billiard}
In the plane $\mathbb R^2$, we consider a $\mathbb Z^2$-periodic locally finite family of scatterers $\{ O_i+\ell;\ i=1,...,I,\ \ell\in\mathbb Z^2\}$. 
We assume that the sets $ O_i+\ell$ are disjoint, open, strictly convex and their boundaries are $C^3$ smooth
with strictly positive curvature.

The dynamics of the Lorentz gas can be described as follows.
A point particle of unit speed is flying freely 
in the interior of $\boldsymbol{\tilde{\mathcal Q}} = \mathbb R^2 \setminus  
\cup_{\ell,i}\left(  O_i+\ell\right)
$
and undergoes elastic collisions on $\partial \boldsymbol{\tilde {\mathcal Q}}$ 
(that is, the angle of reflection equals the angle of incidence). 
Throughout this paper we assume the so-called finite horizon condition, i.e. that the free flight is bounded.
The same dynamics on the compact domain 
is called Sinai billiard. 
The position of the particle is a point $q \in \boldsymbol{\tilde {\mathcal Q}}$ and its velocity is a vector
$v \in \mathcal S^1$ (as the speed is identically $1$). Since collisions happen instantaneously, the pre-collisional
and post-collisional data are identified. By convention, we use the post-collisional data,
i.e. whenever $q \in \partial \boldsymbol{\tilde {\mathcal Q}}$, we assume that $v$ satisfies 
$\vec n_q. v\ge 0$, where $.$ stands for the scalar product and $\vec n_q$
is the unit vector normal to $\partial \boldsymbol{\tilde {\mathcal Q}}$ directed inward 
$\boldsymbol{\tilde {\mathcal Q}}$.
 The phase space, that is, the set of all possible 
positions and velocities, will be denoted by 
$\boldsymbol{\tilde \Omega}  = \boldsymbol{\tilde {\mathcal Q}} \times \mathcal S^1$.

The billiard flow is denoted by $\boldsymbol{\tilde \Phi}_t: \boldsymbol{\tilde \Omega} 
\rightarrow \boldsymbol{\tilde \Omega}$,
where $t \in \mathbb R$. Let $\boldsymbol{\tilde \mu}_0$ be the Lebesgue measure on $\boldsymbol{\tilde \Omega}$
normalized so that $\boldsymbol{\tilde \mu}_0( (\boldsymbol{\tilde {\mathcal Q}} \cap [0,1]^2)\times S^1) =1$.

The Sinai billiard is defined analogously on a compact domain. That is, we consider disjoint strictly convex open
subsets $
\bar O_i
\subset \mathbb T^2$ 
(corresponding to the canonical projection of $O_i$),
$i=1,...,I$, whose boundaries are $\mathcal C^3$ smooth with strictly
positive curvature. Then we put $\boldsymbol{\mathcal Q} = \mathbb T^2 \setminus \cup_{i} O_i$. 
We define the billiard dynamics $(\boldsymbol{\Omega}, \boldsymbol{\Phi}_t, \boldsymbol{\mu}_0)$
exactly as $(\boldsymbol{\tilde \Omega}, \boldsymbol{\tilde \Phi}_t, \boldsymbol{\tilde \mu}_0)$
except that we use the billiard table $\boldsymbol{\mathcal Q} $ instead of $\boldsymbol{\tilde{\mathcal Q}}$ 
and $\boldsymbol{\mu}_0$ is a probability measure.
 
Next, we represent the flow $\boldsymbol{ \Phi}_t$ as a suspension over a map. This map is 
called the billiard ball map: the Poincar\'e section of $\boldsymbol{ \Phi}_t$ corresponding to the collisions.
That is, we define
$$
\boldsymbol{M} = \{ (q,v) \in \boldsymbol{\Omega}: q \in \partial \boldsymbol{\mathcal Q}\}
=  \{ (q,v) \in \boldsymbol{\Omega}: q \in \partial \boldsymbol{\mathcal Q}, \vec n_q. v\ge 0 \}.
$$
$\boldsymbol{T}: \boldsymbol{M} \rightarrow \boldsymbol{M}$ is defined by
$\boldsymbol{T}(x) = \boldsymbol{\Phi}_{\boldsymbol{\tau}}(x)$, where 
$\boldsymbol{\tau} = \boldsymbol{\tau}(x)$ is the smallest positive number such that
$\boldsymbol{\Phi}_{\boldsymbol{\tau}}(x) \in \boldsymbol{M}$. The projection of $\bm \mu_0$ to the Poincar\'e section
is denoted by $\bm \nu$. In fact, $\bm \nu$ has the density 
$c \vec n_q.v dq d v$, where 
$c=2|\partial \boldsymbol{\mathcal Q}|$
is a normalizing constant
such that $\bm \nu$ is a probability measure. 
Clearly, we can write
$$
\bm \Omega = \{ (x,t), x \in \bm M, \ 
t
 \in [0, \bm \tau(x))\}.
$$

With this notation, we have
$\bm \mu_0 = \frac{1}{\bm \nu(\bm \tau)}\bm \nu \otimes \mathfrak l$, where
$\mathfrak l$ is the Lebesgue measure on $[0,+\infty)$. 
Note that the measure $\bm \mu_0$ is a probability measure unlike $ \mu$ defined in Section \ref{sec:notations}.

Finally, we define the 
measure
preserving dynamical system $(\bm{\tilde M}, \bm{\tilde T}, \bm{\tilde \nu})$ analogously to 
the Lorentz gas. 
For every $\ell\in\mathbb Z^2$, we define the $\ell$-cell $\mathcal C_\ell$  as the set of the points with last reflection off 
${\boldsymbol{\widetilde{\mathcal Q}}}$
took place in the set $\bigcup_{i=1}^I(O_i+\ell)$.
Identifying $\mathbb T^2$ with the unit square $
[0,1)^2
\subset \mathbb R^2$, we see that 
$(\bm{\tilde M}, \bm{\tilde T}, \bm{\tilde \nu})$ is the $\mathbb Z^2$-extension of
$(\bm{ M}, \bm{ T}, \bm{ \nu})$ by $\bm \kappa : \bm M \rightarrow \mathbb Z^2$, where
$\bm \kappa(x) = \ell$ if $\tilde {\bm T} (x) \in \mathcal C_\ell$.

The observable $(\bm \kappa, \bm \tau): \bm M \rightarrow \mathbb Z^2 \times \mathbb R$ satisfies the central limit theorem (see e.g. \cite{CM06}).
That is, there exists a $3\times 3$ positive definite matrix 
$\Sigma_{\bm \kappa, \bm \tau}$ so that for any $A \subset \mathbb R^3$ 
whose boundary has zero Lebesgue measure 
$$
\bm{\nu}\left(x \in \bm M: \frac{(\bm \kappa_n, \bm \tau_n - n \nu(\tau))}{\sqrt n} \in A \right) = \int_A \Psi_{\Sigma_{\bm \kappa, \bm \tau}}\, ,
$$
and $\Psi$ is the Gaussian density defined by \eqref{defpsi}.
Consequently, the central limit theorem holds for the observable 
$\bm \kappa$ with a covariance matrix $\Sigma_{\bm \kappa}$, which is obtained from $\Sigma_{\bm \kappa,\bm  \tau}$
by deleting the last row and the last column.

Denote
$$\Vert \mathfrak h\Vert_{\mathcal H^\eta_{E}}=
\sup_{y\in E} |\mathfrak h(y)| +
\sup_{y,z\in E,\ y\ne z}\frac{|\mathfrak h(y)-\mathfrak h(z)|}{d(y,z)^\eta}. $$
We will say that a function $\mathfrak{h}: {\boldsymbol{\tilde \Omega} } \rightarrow\mathbb R$
is {\em smooth in the flow direction} if
\begin{equation}\label{HYPO1}
\forall N\ge 0,\quad \sum_\ell\left\Vert\frac{\partial^N}{\partial s^N}\left(\mathfrak h\circ\boldsymbol{\widetilde\Phi}_s\right)_{|s=0} \right\Vert_{\mathcal H^\eta_{\mathcal C_\ell}}<\infty \, .
\end{equation}
Note that in order for \eqref{HYPO1} to hold, it is sufficient that
 $\mathfrak h$ is $C^\infty$ in the position $q\in\boldsymbol{\widetilde{Q}}$ and satisfies
$$
\forall N\ge 0,\quad \sum_\ell\left\Vert\frac{\partial^N}{\partial q^N}\mathfrak h \right\Vert_{\mathcal H^\eta_{\mathcal C_\ell}}<\infty\, ,
$$
\begin{equation}\label{identification}
\forall (q,\vec v)\in \partial{\boldsymbol{\widetilde Q}}\times S^1,\quad\frac{\partial^N}{\partial q^N}\mathfrak h (q,\vec v)=\frac{\partial^N}{\partial q^N}\mathfrak h\left(q,\vec v-2(\vec n_q.\vec v)\vec n_q\right)\, .
\end{equation}

We say that  $\mathfrak{h}: {\boldsymbol{\tilde \Omega} } \rightarrow\mathbb R$
is {\em $\eta$-H\"older continuous} if it is $\eta$-H\"older continuous on $\tilde Q\times S^1$ and satisfies \eqref{identification} with $N=0$.

Now we are ready to formulate the main result of this section.

\begin{theorem}\label{THEOBILL}
Let $\mathfrak f,\mathfrak g: {\boldsymbol{\tilde \Omega} } \rightarrow\mathbb R$
be 
two $\eta$-H\"older continuous functions
with at least one of them smooth in the flow direction.
Assume moreover that 
there exists an integer $K_0\ge 1$ such that 
\begin{equation}\label{HYPO2}
\sum_{\ell} (1+|\ell|)^{2K_0}  \left(\left\Vert\mathfrak f \right\Vert_{\mathcal H^\eta_{\mathcal C_\ell}}+\left\Vert\mathfrak g \right\Vert_{\mathcal H^\eta_{\mathcal C_\ell}}\right)<\infty\, .
\end{equation}
Then there are real numbers
$\mathfrak C_0(\mathfrak f, \mathfrak g), \mathfrak C_1(\mathfrak f, \mathfrak g),...,\mathfrak C_{K_0}(\mathfrak f,\mathfrak g)$ so that
we have 
\begin{equation}
\label{exp2_bold}
\int_{\boldsymbol{\tilde \Omega}} \mathfrak f \, \mathfrak g\circ {\boldsymbol\tilde  \Phi}_t d{\boldsymbol{\tilde \mu}_0} =
\sum_{k=0}^{K_0}\mathfrak C_k(\mathfrak f,\mathfrak g)t^{-1-k}+o(t^{-1-K_0})\, ,
\end{equation}
as $t\rightarrow +\infty$.
Furthermore, 
$\mathfrak C_0(\mathfrak f,\mathfrak g) = \mathfrak c_0\int_{\boldsymbol{\tilde \Omega}} \mathfrak f d{\boldsymbol{\tilde \mu}_0} 
\int_{\boldsymbol{\tilde \Omega}} \mathfrak g d{\boldsymbol{\tilde \mu}_0} $ 
with 
\begin{equation}
\label{deffc}
\mathfrak c_0 = \frac{
\boldsymbol{\nu}(\boldsymbol{\tau})
}{2 \pi \sqrt{\det \Sigma_{\bm \kappa}}}
\end{equation}
and 
the coefficients $\mathfrak C_k$, as functionals over pairs of admissible functions, are bilinear.
\end{theorem}
We note that the bilinear forms $\mathfrak C_k$ are linearly independent. Namely in 
Appendix
\ref{sec:coboundary} we give examples of parts $f_k, g_k$ such that
$\mathfrak C_k(f_k, g_k)\neq 0$ 
while $\mathfrak C_j(f_k, g_k)\neq 0$ for all $j<k.$

In the remaining part of Section \ref{sec:mixbilliard}, 
we derive
Theorem \ref{THEOBILL}
from Theorem \ref{THMGENE2}. However, we will not be applying Theorem \ref{THMGENE2} directly
to $(\bm M, \bm \nu, \bm T)$, but instead we apply it to the Young tower extension of the Sinai
billiard. Thus we first briefly review the Young tower construction in Section \ref{sec:towers}. 
Then we prove condition \eqref{DMbound1} in Section \ref{sec:DMbound}
along the lines of \cite{Dima98}. Finally we complete the proof 
of Theorem \ref{THEOBILL} in Section \ref{sec:proofthmbill}.
\eqref{deffc}  is established in Section \ref{Sec:const}.

\subsection{Young towers}
\label{sec:towers}

Let $\mathcal R \subset  \bm M$ be the hyperbolic product set constructed in 
\cite[Section 8]{Y98}.
Furthermore, let $(\Delta, F)$ be the corresponding Young tower ("Markov extension"). 
There is a natural bijection $\iota$ between $\Delta_0$, the base of the tower and $\mathcal R $. We will denote points
of $\mathcal R$ by $x=(\gamma^u,\gamma^s)$, which is to be interpreted as 
$\gamma^u \cap \gamma^s$, where $\gamma^u = \gamma^u(x)$ and $ \gamma^s= \gamma^s(x)$ are an unstable and a stable
manifold containing $x$. Points of $\Delta_0$ will be denoted by 
$\hat x=(\hat \gamma^u, \hat \gamma^s)$. Note that $\iota$ can be extended to $\pi$, a mapping from $\Delta$ to $\bm M$ (this map is in general not one-to-one).

We recall the most important ingredients of the construction of \cite{Y98}.
The base of the tower has the product structure
$ X =  \Delta_0 =   \Gamma^u \times   \Gamma^s.$ The sets of the form
$A \times   \Gamma^s$, $A \subset \Gamma^u$ are called
u-sets if $\iota (A \subset \Gamma^u)$ is compact.
Similarly, sets of the form $  \Gamma^u \times B$, $B \subset   \Gamma^s$ are called 
s-sets if $\iota (B \subset \Gamma^u)$ is compact.
Also, sets 
of the form $\Gamma^u \times \{ \hat \gamma^s\}$ are called stable manifolds and sets of the form 
$\{ \hat \gamma^u\} \times  \Gamma^s$ are unstable manifolds
as they are images of (un)stable manifolds (or rather, the intersections of 
(un)stable manifolds and $\mathcal R$)
by the map $\iota^{-1}$. 
$  \Delta_0 $ has a partition $\Delta_0 = \cup_{k \in \mathbb Z_+} \Delta_{0,k}$, where
 $  \Delta_{0,k}=  \Gamma^u \times   \Gamma^s_{k}$ are s-sets.
The return time to the base on the set $\Delta_{0,k}$ is identically $r_k$, that is $  \Delta = \cup_{k\in \mathbb Z_+} \cup_{l=0}^{r_k -1}   \Delta_{l,k}$,
where $  \Delta_{l,k} = \{ (\hat x,l): \hat x \in   \Delta_{0,k}\}$. There is an $F$-invariant measure $\nu$ on $\Delta$ so that $\pi_* \nu = \mu$ and
$  F$ is an isomorphism between $  \Delta_{l,k}$ and $   \Delta_{l+1,k}$ 
and $  F(\hat x,l) = (\hat x, l+1)$. Also
$  F$ is an isomorphism between $  \Delta_{r_k-1,k}$ and $  F(  \Delta_{r_k-1,k})$, the latter being a u-set of $  \Delta_0$.
Furthermore, if $\hat x_1,\hat x_2 \in \Delta_{0,k}$ 
belong to the same (un)stable manifold, so do $  F^{r_k}(\hat x_1,0)$
and $  F^{r_k}(\hat x_2,0)$. We write $  {\mathcal F} =   F^{r_k - l}$ on $  \Delta_{l,k}$ and
$r(\hat \gamma^u, \hat \gamma^s) = r(\hat \gamma^s) = r_k$ for $(\hat \gamma^u, \hat \gamma^s) \in   \Delta_{0,k}$.
Define $\Xi$  on $  \Delta$  by 
\begin{equation}
\label{defboldgamma}
\Xi((\hat \gamma^u,\hat \gamma^s),l) = ((\hat {\bm \gamma}^u, \hat \gamma^s),l)
\text{ with a fixed }\hat {\bm \gamma}^u.
\end{equation}
Let $\bar \Delta = \Xi (  \Delta)$ and 
$\bar \nu = \Xi_*   \nu$. There is a well defined $\bar F: \bar \Delta \rightarrow \bar \Delta$
such that $\Xi \circ   F = \bar F \circ \Xi$. The dynamical system $(\bar \Delta, \bar \nu, \bar F)$,
is an expanding tower, in the sense that it satisfies assumptions 
(E1)--(E5) below.

Let $(\bar \Delta, \bar \nu, \bar F)$ be a probability preserving dynamical system with a partition $(\bar \Delta_{l,k})_{ k \in I,
  l =0,..., r_k-1}$
into positive measure subsets, where $I$ is either finite or countable and 
$r_k = r(\bar \Delta_{0,k})$ is a positive
integer. We call it {\em an expanding tower} if

\begin{enumerate}
\item[(E1)] for every $i \in I$ and $0 \leq j < r_i -1$, $F$ is a measure preserving isomorphism
between $\bar \Delta_{
j,i}$ and $\bar \Delta_{
j+1,i}.$
\item[(E2)] for every $i \in I$, $\bar F$ is an isomorphism between $\bar \Delta_{
r_i -1,i}$
and $$\bar X:= \bar \Delta_0 := \cup_{i \in I} \bar \Delta_{
0,i}.$$
\item[(E3)] Let $r(x) = r(\bar \Delta_{0,k})$ if $x \in \bar \Delta_{0,k}$ and 
$\bar {\mathcal F}:\bar X \rightarrow \bar X$ be the first return map to the base, i.e. $\bar {\mathcal F}(x) = \bar F^{r(x)} (x)$.
Let $s(x,y)$, the separation time of $x,y \in X$, be defined as the smallest integer $n$
such that $\bar {\mathcal F}^n x \in \bar \Delta_{0,i}$, $\bar {\mathcal F}^n y \in \bar  \Delta_{0,j}$ with $i \neq j$. 
As $\bar {\mathcal F}: \bar \Delta_{0,i} \rightarrow \bar X$ is an isomorphism, it has an inverse. Denote by
 $\alpha$ the logarithm of the
Jacobian of this inverse (w.r.t. the measure $\bar \nu$). Then there are constants
$\vartheta_0 < 1$
and $C>0$ such that for every $x,y \in \bar  \Delta_{0,i} $, 
$| \alpha(x) -  \alpha(y) | \leq C 
\vartheta_0
^{s(x,y)}.$
\item[(E4)] Extend $s$ to $\bar \Delta$ by setting $s(x,y) = 0$ if $x,y$ do not belong to the same
$\bar \Delta_{
j,i}$ and $s(x,y) = s(\bar F^{-j}x,\bar F^{-j}y) + 1$ if $x,y \in \bar \Delta_{
j,i}$.
$(\bar \Delta, \bar  \nu,\bar  F)$ is exact (hence ergodic and mixing) with respect to the
metric 
\begin{equation}
\label{DefSymbMetr}
d_{\vartheta}(x,y) := \vartheta^{s(x,y)}.
\end{equation}
\end{enumerate}
Furthermore, in case of Sinai billiards, we have
\begin{enumerate}
\item[(E5)] $\bar \nu (x: r(x) > n) \leq C \rho ^n$ with some $\rho <1$.
\end{enumerate}

\subsection{Condition \eqref{DMbound1} for Sinai billiards}
\label{sec:DMbound}

Given a function $f: \bm M \rightarrow \mathbb C$, we define $\hat f: \Delta \rightarrow \mathbb C$
by $\hat f = f \circ \pi$. 
Now given a function $\hat f: \Delta \rightarrow \mathbb C$ (which may or may not be a
lift-up of a function $f: \bm M \rightarrow \mathbb C$),
we write $X=\Delta_0$ and define 
\begin{gather*}
\hat f_X : X \rightarrow \mathbb C, \quad \hat f_X (\hat x) = \sum_{j=0}^{r(\hat x) -1} \hat f(F^j(\hat x)),\\
\bar f : \bar \Delta \rightarrow \mathbb C, \quad \bar f (\hat \gamma^s,l) =  
\hat f(\hat{\bm \gamma}^u,\hat \gamma^s,l)\, ,\\
\bar f_{\bar X} : \bar X \rightarrow \mathbb C, \quad \bar  f_{\bar X}(\hat \gamma^s) 
=  \sum_{j=0}^{r(\hat \gamma^s) -1} \hat f(F^j(\hat{\bm \gamma}^u,\hat \gamma^s))\, .
\end{gather*}

Fix $\varkappa<1$ and
consider the space of dynamically Lipschitz functions on $\bar X$ (w.r.t. the metric 
$d_{\varkappa}$): 
$$
C_{\varkappa}(\bar X, \mathbb C) = \{f: \bar X \rightarrow \mathbb C \text{ bounded and }  L(f) < \infty \},
$$
where 
$$L(f) = \inf \{ C: \forall x, y \in \bar X: | f(x) -  f(y)| \leq C \varkappa^{s(x,y)}  \}.$$
This space is equipped with the norm

$$
\|  f\|_{\varkappa} = L(f) + \|  f \|_{\infty} .
$$

Let $Q$ be the Perron-Frobenius-Ruelle operator associated with $\bar {\mathcal F}$, i.e.
$$
(Qh)(x) = \sum_{y: \bar {\mathcal F} y = x} e^{\alpha (y)} h(y)
$$
where 
$e^\alpha$
is the Jacobian defined in (E3).
We have for $h$ with $\|  h\|_{\varkappa}  <\infty$
\begin{equation}
\label{PSpec}
Qh = \bar \nu (h) + R h,
\end{equation}
where $\|  Rh \|_{\varkappa} \leq \rho \|  h \|_{\varkappa}$ with some $\rho <1$.

Now we introduce the (signed) temporal distance function $D$ on $\mathcal R$ by defining
\begin{eqnarray}
\label{deftemp}
D(x,y) = \sum_{\ell= - \infty}^{\infty} 
[  \bm \tau ( \bm T^\ell (\gamma^u(x),\gamma^s(x))) -  
\bm \tau ( \bm T^\ell (\gamma^u(x),\gamma^s(y))) + \\
 \bm \tau ( \bm T^\ell (\gamma^u(y),\gamma^s(y))) -  \bm \tau ( \bm T^\ell (\gamma^u(y),\gamma^s(x))) ], \nonumber
\end{eqnarray}
where  $\bm \tau$ is defined in Section \ref{sec:billiard}. 
Note that there is a lift-up $\hat{\bm \tau}: \Delta \rightarrow \mathbb R_+$ defined by $ \hat{\bm \tau}(\hat x) = \bm \tau (\pi (\hat x))$ and corresponding functions
$ \hat{\bm \tau}_X, \bar {\bm \tau}, \bar {\bm \tau}_{\bar X}$.

We also define the operators 
\begin{equation}
\label{defQxi}
Q_{\xi} h = Q(e^{i \xi \bar {\bm \tau}_{\bar X}} h)\, .
\end{equation}

For real valued functions defined on $\bar X$, we will consider the norms
$$
\| .\|_{\infty}, \quad \|  .\|_{\varkappa}, \quad \| .\|_{(\xi)} := \max \{ \| . \|_{\infty}, C_0 L(.) / \xi \}\, ,
$$
where $\xi \gg 1$ and $C_0$ is a constant to be specified later.

Now, let us consider points $x_m = (
\gamma^u(x_m),\gamma^s(x_m)) , y_m = (
\gamma^u(y_m),\gamma^s(y_m)) \in \mathcal R$
which satisfy that $\mathcal F^k (\iota^{-1} (x_m)) \in \Delta_{0, 1}$, $\mathcal F^k (\iota^{-1} (y_m)) \in \Delta_{0, a_k}$ for 
$k \geq 0$, where
$$
a_k = \begin{cases}
2 \text{ if }k = m^2  \text{ or } k = m^2 +m \\
1 \text{ otherwise. }
\end{cases}
$$
Let 
$$x_m' :=  \bm T^{
r_1(m^2+1)} (x_m )= \iota( \mathcal F^{
m^2+1} (\iota^{-1} (x_m))) \text{ and }
y_m' :=  \bm T^{
r_1 m^2  + r_2} (y_m )= \iota( \mathcal F^{
m^2+1} (\iota^{-1} (y_m))).$$
Let $\mathcal Q_m$ be the solid rectangle with corners $x_m'$, $[x_m',y_m']$, $y_m'$, $[y_m',x_m']$, 
i.e. the unique topological rectangle inside the convex hull of $\mathcal R$ which is 
bounded by two stable and unstable manifolds, such that two of its corners are $x_m'$ and $y_m'$.
We claim that there are two constants $0< c_2<c_1 <1 $ so that $c_2^m < \mu (\mathcal Q_m) < c_1^m$ for sufficiently large $m$.
To prove this claim, let $\mathcal Q_{0,i}$ denote the smallest topological rectangle containing $\iota (\Delta_{0,i})$
for $i=1,2$. Note that $\bm T^{r_1}$ is a $\mathcal C^2$ self map of $\mathcal Q_{0,i}$.
By construction, $\bm T^{jr_1} \mathcal Q_m$ is a subset of $\mathcal Q_{0,1}$ for $j=0,1,...,m-2$.
Now consider a foliation of $\mathcal Q_m$ by unstable curves. Each such curve is expanded by a factor $\Lambda >1$
by the map $\bm T^{r_1}$ and so the upper bound follows.
To prove the lower bound, observe that $\bm T^{(m-1)r_1} \mathcal Q_m$ intersects both $\mathcal Q_{0,1}$
and $\mathcal Q_{0,2}$ and so, as we can assume that the distance between $\mathcal Q_{0,1}$
and $\mathcal Q_{0,2}$ is positive, the length of the image of each unstable curve in our foliation under the map
$\bm T^{(m-1)r_1}$
is uniformly bounded from below. Furthermore, the expansion of $\bm T^{r_1}$ on $\mathcal Q_{0,1}$ is bounded from
above and so the lower bound follows as well.
Next, Lemma 5.1 of \cite{KM81} states that $\mu (\mathcal Q_m) = |D(x_m,y_m)|$ 
(see also \cite[\S 6.11]{CM06}). 
Note that $D(x_m,y_m)$ has another representation:
it is the unique small number $\sigma$ so that $\Phi^{\sigma} Y_1 = Y_5$,
where $\Phi$ is the billiard flow, $Y_1, ..., Y_5$ are points whose last collisions were at $x_m'$, $[x_m',y_m']$, $y_m'$, $[y_m',x_m']$, $x_m'$, respectively
and the pairs $(Y_1, Y_2), (Y_3, Y_4)$ are on the same stable manifold of $\Phi$ while the pairs $(Y_2, Y_3), (Y_4, Y_5)$ are on the same unstable manifold of $\Phi$ (see Lemma 6.40 in \cite{CM06}).
We summarize the results of this construction in

\begin{lemma}
\label{lemma1}
There exist some $a_0 >0$, and $c \in \mathbb R_+$ such that for any $\xi >3$ 
there are 
$x=x(\xi), y=y(\xi) \in \mathcal R$ satisfying
\begin{gather}
\iota^{-1}(\bm T^{r_1k} (x)) \in \Delta_{0,1}  \text{ for all } k = -( \ln \xi)^{3/2}, ..., -1 \, ,\label{eq:negbddr1}\\
\iota^{-1}( \bm T^{- r_2} (y)) \in \Delta_{0,2} \text{ and } \iota^{-1}( \bm T^{(k+1)r_1 - r_2} (y)) \in \Delta_{0,1 }\text{ for all } k = -( \ln \xi)^{3/2}, ..., -2\, , \label{eq:negbddr2}\\
\mathcal F^k (\iota^{-1} (x)), \mathcal F^k (\iota^{-1} (y)) \in \Delta_{0,1}  \cup \Delta_{0,2} \text{ for all } k \geq 0 \, ,\label{eq:posbddr}
\end{gather}
and
\begin{equation}
\label{eq:dio}
|e^{i \xi D(x,y)} -1| > c \xi^{-a_0}\, .
\end{equation}

\end{lemma}

\begin{proof}
It is sufficient to prove the lemma for $\xi$ large. 
Indeed, if we can prove the lemma for $\xi > \xi_0$, then we can extend it to any $\xi>3$ by choosing $c$ small
enough
unless there is some $\xi' \in [3, \xi_0]$ so that $\xi' D(x,y) = 0  \;(\bmod\; 2 \pi)$ for all $x,y$. Note that this cannot happen since this would imply $l \xi' D(x,y) = 0  \;(\bmod\; 2 \pi)$
where we can choose $l \in \mathbb Z_+$ so that $l \xi' >\xi_0$.

Now given $\xi$, we choose $m$ so that $c_1^m < \xi^{-1} \leq c_1^{m-1}$. Recall that for this $m$, we have points $x_m', y_m'$ so that $c_2^m < |D(x_m',y_m')| < c_1^m$. 
We conclude
$$
c_2 \xi^{1- \frac{\ln c_2}{\ln c_1}} \leq \xi |D(x_m',y_m')| \leq 1\, .
$$
Clearly, \eqref{eq:negbddr1}, \eqref{eq:negbddr2} and \eqref{eq:posbddr} hold for $\xi > \xi_0$ as $m^2 \gg ( \ln \xi)^{3/2}$. 
\end{proof}

Recall the definition of $Q_{\xi}$ from \eqref{defQxi}. We have

\begin{lemma}
\label{lemma2}
There are constants $a_1, C_1, C_2$ so that for every $\xi >3$,
\begin{equation}
\label{eqlemmatwo}
\left\| Q_{\xi}^{C_1 \ln \xi }\right\|_{(\xi)} < 1 - \frac{C_2}{\xi^{a_1}}\, .
\end{equation}
\end{lemma}

\begin{proof}
Let $h$ satisfy $\| h \|_{(\xi)} = 1$. \\

First recall that by \cite{BHM05}, there exists a constant $C_{0,1}$ such that
\begin{equation}
\label{Bruinest}
L(Q^n_\xi h) \leq C_{0,1} [\xi \| h \|_{\infty} + \theta^n L(h)]\, ,
\end{equation}
(see also Proposition 3.7 in \cite{M05}). Thus choosing our $C_0 = C_0(C_{0,1})$
small enough
in the definition of the norm $\| . \|_{(\xi)}$ and $C_{1,1}$
sufficiently big, we obtain
$$L\left(Q_{\xi}^{C_{1,1} \ln \xi } h\right) \leq \frac{\xi}{2 C_0}.$$
In order to prove the lemma, it remains to verify \eqref{eqlemmatwo} for the infinity norm. 

This proof is divided into three parts:
\smallskip

{\it Step 1.} We show that $\| Q_{\xi}^{C_{1,2} \ln \xi } h\|_{L^1} < 1 - \frac{C_{2,1}}{\xi^{a_{1,1}}}$
assuming the following hypothesis.

{\bf (H)}: {\it there is some 
$$u \in \bar X_{\leq 2}
:= \{ \bar x \in \bar X: \bar {\mathcal F}^n(\bar x) \in \bar \Delta_{0,1} \cup \bar \Delta_{0,2}  \text{ for all } n \in \mathbb  N \}
$$ so that 
\begin{equation*}
\label{eqsmallh}
|h(u)| < 1- \frac{C_{2,2}}{\xi^{a_{1,2}}}.
\end{equation*}}

Let $U$ denote the 
$C_{2,2}C_0\xi^{-a_{1,2}-1}/2$ neighborhood of $u$ (w.r.t the metric $d_{\varkappa}$) in $\bar X$. Since $
L(h) \le \xi/C_0
$, we have 
$|h(u')| < 1- \frac{C_{2,2}}{2\xi^{a_{1,2}}}$ for any $u' \in U$.
By the bounded distortion property and by the fact that $u \in \bar X_{\leq 2}$, we have
$\frac{C_{2,3}}{\xi^{a_{1,3}}} \leq \bar \nu (U)$. 
Observing  that
\begin{equation}
\label{maxest} 
|Q_\xi^n h| \leq Q^n|h|
\end{equation} 
holds pointwise (by definition of the operators and by induction on $n$), and using $\|h\|_{\infty} \leq 1$, we derive that for any $\ell$
\begin{gather*}
\int |Q_{\xi}^\ell h | d \bar \nu \leq \int Q^\ell |h| d \bar \nu = \int |h| d \bar \nu = \int_{ U} |h| d \bar \nu + \int_{\bar X \setminus U} |h| d \bar \nu\\
\leq \left(1- \frac{C_{2,2}}{2\xi^{a_{1,2}}} \right)\bar \nu (U) + 1- \bar \nu (U) \leq 1 - \frac{C_{2,4}}{\xi^{a_{1,4}}}\, ,
\end{gather*}
with $C_{2,4} = C_{2,2}C_{2,3}/2$ and $a_{1,4} = a_{1,2} + a_{1,3}$.
\smallskip

{\it Step 2.} Under hypothesis {\bf (H)}, we show that 
$\DS \left\| Q_{\xi}^{C_{1,3} \ln \xi } h\right\|_{\infty} < 1 - \frac{C_{2,5}}{\xi^{a_{1,5}}}$.

For any $u \in \bar X$, we have
\begin{gather*}
 \left| Q_{\xi}^{C_{1,3} \ln \xi } h \right|(u) 
 = \left| Q_{\xi}^{(C_{1,3} - C_{1,2}) \ln \xi } (Q_{\xi}^{C_{1,2} \ln \xi }  h) \right|(u)\, \\
 \leq^{\eqref{maxest}} 
 \left(Q^{(C_{1,3} - C_{1,2}) \ln \xi } \left|Q_{\xi}^{C_{1,2} \ln \xi }  h\right|\right) (u) \leq
 \bar \nu \left(\left|Q_{\xi}^{C_{1,2} \ln \xi }  h\right|\right) + C \xi \theta^{(C_{1,3} - C_{1,2}) \ln \xi}\, ,
\end{gather*}
where the last inequality follows from \eqref{PSpec},
\eqref{Bruinest} and
\eqref{maxest}.
By Step 1 and by choosing $C_{1,3} - C_{1,2}$ sufficiently large, we see that Step 2 is completed. 
\smallskip

{\it Step 3.} We show that 
$\DS \left\| Q_{\xi}^{C_{1,4} \ln \xi } h\right\|_{\infty} < 
1 - \frac{C_{2,5}}{\xi^{a_{1,5}}}$ with $C_{1,4} = 2 C_{1,3}$
without assuming {\bf (H)}.

In order to complete Step 3, it suffices to show that there exists some $v \in \bar X_{\leq 2}$ that either satisfies {\bf (H)} or satisfies the following:
\begin{equation}
\label{eqsmallh2}
|Q_{\xi}^{n} h(v)| < 1- \frac{C_{2,2}}{\xi^{a_{1,2}}} \text{ with } n = C_{1,3} \ln \xi .
\end{equation}
Indeed, if there is a $v$ satisfying {\bf (H)}, then noting that $\| Q_{\xi}\|_{\infty} \leq 1$,
the proof in Step 2 applies.
On the other hand, if there is a $v$ satisfying
\eqref{eqsmallh2}, then since $\| Q_{\xi}\|_{(\xi)} \leq 1$, we have
$
\| Q_{\xi}^{n} h \|_{(\xi)} \leq 1
$
and so we can apply the results of Step 2 for the function $h$ replaced by $Q_{\xi}^{n} h$.

For a function $f: \bar X \rightarrow \mathbb R$ and $n \in \mathbb N$, we write $f_n(x) = \sum_{j=0}^{n-1} f(\bar {\mathcal F}^j x)$.

Recall that for our $\xi$, Lemma \ref{lemma1} gives us $x, y \in \mathcal R$ (in fact, with the previous notation
$x= x_m', y=y_m'$ with $m \approx (\ln (1/c_1))^{-1} \ln \xi$). Let us write 
$(\hat \gamma^u(x), \hat \gamma^s(x)) = \iota^{-1} (x)$, 
$(\hat \gamma^u(y), \hat \gamma^s(y)) = \iota^{-1} (y)$, 
$v = \bar {\mathcal F}^{n/2}(\hat \gamma^s(x))$, $w = \bar {\mathcal F}^{n/2}(\hat \gamma^s(y))$.
We will show that in case no point satisfies {\bf (H)}, then either $v$ or $w$ satisfies \eqref{eqsmallh2}. To this end, assume by contradiction that none of them satisfies \eqref{eqsmallh2}. 

Writing $h(\bar x) = r(\bar x) e^{i \phi (\bar x)}$, we have
\begin{gather*}
(Q_{\xi}^n h )(v) = \sum_{u \in \bar X: \bar {\mathcal F}^n u = v} e^{\alpha_n (u) + i \xi (\bar {\bm \tau}_{\bar X})_n (u)} r(u) e^{i \phi (u)}\\
= e^{\alpha_n (v'_{-n}) + i \xi (\bar {\bm \tau}_{\bar X})_n (v'_{-n})} r(v'_{-n}) e^{i \phi (v'_{-n})} + e^{\alpha_n (v''_{-n}) + i \xi (\bar {\bm \tau}_{\bar X})_n (v''_{-n})} r(v''_{-n}) e^{i \phi (v''_{-n})}  + ... 
\end{gather*}
where 
$$v'_{-n} = \Xi (\iota^{-1}( \bm T^{-r_1 n/2} ( \gamma^u(x),\gamma^s(x)))),
v''_{-n} = \Xi (\iota^{-1}( \bm T^{-r_1 (n/2 -1) - r_2} ( \gamma^u(y),\gamma^s(x))))
$$ 
and $...$ corresponds to all other preimages.

Thus $(Q_{\xi}^n h )(v)$ is expressed as a weighted sum of the unit vectors 
$e^{i [\xi (\bar {\bm \tau}_{\bar X})_n (u)+ \phi (u)]} \in \mathbb C$, with weights 
$e^{\alpha_n (u)} r(u)$. Noting that 
$\DS \sum_{u \in \bar X: \bar {\mathcal F}^n u = v} e^{\alpha_n (u)} = 1$ and $|r| \leq 1$, we observe that $v$ can only violate 
\eqref{eqsmallh2} if all the unit vectors, whose weights are at least 
$C_{2,6}/\xi^{a_{1,6}}$ are nearly collinear, i.e. their angle do not differ by more than
$C_{2,6}/\xi^{a_{1,6}}$ with $a_{1,6} = a_{1,2}$. 

If $r(v_{-n}') <1/2 $ or $r(v_{-n}'') <1/2 $, then one of these points
satisfies $\bf (H)$ and so the proof is completed. If 
$r(v_{-n}') \geq 1/2$ and $r(v_{-n}'') \geq 1/2$
then we also claim that $e^{\alpha_n (v_{-n}')} > 2C_{2,6} / \xi^{a_{1,6}}$
and $e^{\alpha_n (v_{-n}'')} > 2C_{2,6} / \xi^{a_{1,6}}$. Indeed, this holds
since $v_{-n}', v_{-n}'' \in \bar X_{\leq 2}$ and since 
$\alpha$ is a H\"older function and so it 
is bounded from below by a positive number 
on the compact set $\bar X_{\leq 2}$ (and so $e^{\alpha}$  on the set $\bar X_{\leq 2}$ is 
bounded from below by a number which is bigger
than one).

Thus we have derived that 
\begin{equation*}
| [\xi (\bar {\bm \tau}_{\bar X})_n (v'_{-n}) - \xi (\bar {\bm \tau}_{\bar X})_n (v''_{-n})] - [ \phi (v'_{-n}) - \phi (v''_{-n})] | \leq C_{2,6}/\xi^{a_{1,6}}
\end{equation*}
Repeating the above argument for $w$, 
and writing 
$$w'_{-n} = \Xi (\iota^{-1}( \bm T^{-r_1 (n/2 -1) - r_2} (\gamma^u(y),\gamma^s(y)))),
w''_{-n} = \Xi (\iota^{-1}( \bm T^{-r_1 n/2 } ( \gamma^u(x),\gamma^s(y)))),
$$ 
we find
\begin{equation*}
| [\xi (\bar {\bm \tau}_{\bar X})_n (w'_{-n}) - \xi (\bar {\bm \tau}_{\bar X})_n (w''_{-n})] - [ \phi (w'_{-n}) - \phi (w''_{-n})] | \leq C_{2,6}/\xi^{a_{1,6}}.
\end{equation*}
By construction, $s(v'_{-n}, w''_{-n}) \geq n/2$ and thus $| \phi (v'_{-n}) - \phi (w''_{-n}) | \leq C_{2,6}/\xi^{a_{1,6}}$ assuming that $C_{1,2}$
is sufficiently large. Similarly, we can assume $| \phi (v''_{-n}) - \phi (w'_{-n}) | \leq C_{2,6}/\xi^{a_{1,6}}$ and thus with 
$C_{2,7} = 4 C_{2,6}$ and
$a_{1,7} = a_{1,6}+1$,
\begin{equation}
\label{collin1}
|A| \leq C_{2,7}/\xi^{a_{1,7}}
\text{ where } A =
 (\bar {\bm \tau}_{\bar X})_n (v'_{-n}) - (\bar {\bm \tau}_{\bar X})_n (v''_{-n}) + (\bar {\bm \tau}_{\bar X})_n (w'_{-n}) - (\bar {\bm \tau}_{\bar X})_n (w''_{-n}).
  \end{equation}
Recall \eqref{defboldgamma} and \eqref{deftemp}. Using the notations
$z = (\gamma^u(z), \gamma^s(z)) \in \mathcal R$, 
$\hat z = \iota^{-1}(z) = (\hat \gamma^u(z), \hat \gamma^s(z))$ and
\begin{gather*}
H(z) = \sum_{\ell=0}^{\infty} [{\bm \tau} ( \bm T^\ell (
\gamma^u(z), \gamma^s(z))) - {\bm \tau} ( \bm T^\ell 
(\bm \gamma^u, \gamma^s(z)))] \, ,
\end{gather*}
observe that we have
\begin{equation}
\label{cohomH}
 \hat{\bm \tau}_X 
 (\hat \gamma^u(z), \hat \gamma^s(z)) - 
 \bar {\bm \tau}_{\bar X} 
 (\hat \gamma^s(z)) = 
 H(\gamma^u(z), \gamma^s(z)) - 
 H( \bm T^{r(\hat \gamma^s(z))} (
\hat \gamma^u(z), \hat \gamma^s(z)))\, .
\end{equation}

To simplify notation, we write
\begin{equation}
\label{LocProd}
[z_1,z_2] = (\gamma^u(z_1),\gamma^s(z_2))
\end{equation}
and
$$ d_{\ell, f} (z_1, z_2) = 
f (  \bm T ^\ell ([z_1,z_1]) - f (  \bm T ^\ell 
([z_1,z_2])) 
-f (  \bm T ^\ell 
( [z_2,z_1])) + f (  \bm T ^\ell 
([z_2,z_2])).
$$

Recall the dynamical H\"older continuity of $\bm \tau$: there is some $C$ and 
$\vartheta <1$ so that if $z_1,z_2 \in \bm M$
are such that $\bm T^\ell (z_1)$ and $\bm T^\ell (z_1)$ stay on the same local unstable manifold for all $\ell \leq L$,
then $|\tau(z_1) - \tau(z_2)| < C \vartheta^L$. Likewise, if $\bm T^\ell (z_1)$ and $\bm T^\ell (z_1)$ 
stay on the same local stable manifold for all $\ell \geq - L$, then $|\tau(z_1) - \tau(z_2)| < C \vartheta^L$.

We have
$$
D(x_m',y_m') = \sum_{\ell= - \infty}^{\infty} d_{\ell, {\bm \tau}} (x_m',y_m')  = S_1 + S_2 + S_3,$$
where
$$
S_1 =  \sum_{\ell= - \infty}^{- r_1 n /2 -1} \bm \tau (  \bm T ^\ell (x_m'))
-  \bm \tau (  \bm T ^\ell ([x_m',y_m'])) 
$$
$$ 
\sum_{\ell= - \infty}^{- r_1 (n /2 -1) -r_2 -1}  - \bm \tau (  \bm T ^\ell ([y_m',x_m']))
+ \bm \tau (  \bm T ^\ell (y_m')),
$$

$$
S_2 =  \sum_{\ell= - r_1 n /2 }^{ r_1 (n /2 -1)} \bm \tau (  \bm T ^\ell (x_m'))
-   \sum_{\ell= - r_1 n /2 }^{ r_1 (n /2 -1)} \bm \tau (  \bm T ^\ell ([x_m',y_m'])) 
$$
$$ -
\sum_{\ell= -  r_1 (n /2 -1) - r_2 }^{r_1 (n /2 -1)} \bm \tau (  \bm T ^\ell ([y_m',x_m']))
+ \sum_{\ell= -  r_1 (n /2 -1) - r_2 }^{r_1 (n /2 -1)} \bm \tau (  \bm T ^\ell (y_m')),
$$
and 
$$
S_3 = \sum_{\ell = r_1 (n /2 -1) + 1}^{\infty} 
\bm \tau (  \bm T ^\ell (x_m')) -  \bm \tau (  \bm T ^\ell ([x_m',y_m'])) - \bm \tau (  \bm T ^\ell ([y_m',x_m']))
+  \bm \tau (  \bm T ^\ell (y_m')).
$$ 
In other words, we rearrange terms in the infinite sum according to the first return to the base
in the tower representation. 
Observe that
in view of \eqref{cohomH},
$$
S_2 = 
\sum_{k=0}^{n-1} \hat{\bm \tau}_X (\bar{\mathcal F}^k (v_{-n}' ))
- \hat{\bm \tau}_X (\bar{\mathcal F}^k (v_{-n}'' ))
- \hat{\bm \tau}_X (\bar{\mathcal F}^k (w_{-n}' ))
+ \hat{\bm \tau}_X (\bar{\mathcal F}^k (w_{-n}'') ).
$$
Next, using \eqref{collin1}, \eqref{cohomH} and performing a telescopic sum,
we find
$$
S_2 = A + d_{0,H}(  \bm T ^{-r_1 n /2}(x_m'),  \bm T ^{-r_1 (n/2 -1) - r_2}(y_m')) - d_{0,H}(  \bm T ^{r_1 n /2}(x_m'),  \bm T ^{r_1 (n/2 -1)+r_2}(y_m')).
$$
By the dynamical H\"older property of ${\bm \tau}$, $S_1 + (S_2 - A) + S_3$
can be made smaller than 
$ C_{2,7}/\xi^{a_{1,7}}$ assuming that $C_{1,2}$ is large enough.
Indeed, e.g. both series whose sum defines $S_1$ are absolutely convergent and 
are smaller than $C \frac{1}{1-\vartheta} \vartheta^{n/2}$ (the absolute convergence justifies why we can write 
$S_1$ as a sum of these two series).
Estimating $S_3$ is even simpler: we can assume $n/2 > m$ and so all of the points
$$\bm T ^\ell(x_m'), 
\bm T ^\ell([x_m',y_m']), \bm T ^\ell(y_m'), \bm T ^\ell([y_m',x_m']) 
$$
lie on the same local stable manifold for $\ell > n/2$. Assuming $n/4 > m$ as well, the 
dynamical H\"older continuity of $\bm \tau$ implies $|S_3| \leq  C \frac{1}{1-\vartheta} \vartheta^{n/4}$.
The argument is similar for $(S_2 - A)$.
Thus we derived that $D(x_m',y_m') \leq 2 C_{2,7}/\xi^{a_{1,7}}$ which is a contradiction with the choice of $x_m'$
and $y_m'$ assuming, as we can, that $a_{1,1}$ is chosen sufficiently
big so that $a_{1,7} > a_0$.
\end{proof}

Let the operator $Q_{\theta, \xi}$ be defined by 
$Q_{\theta, \xi} h = Q(e^{i \theta \cdot \bar {\bm \kappa}_{\bar X} + i \xi \bar {\bm \tau}_{\bar X}} h )$,
where $\bm \kappa: \bm M \rightarrow \mathbb Z^2$ is defined in Section \ref{sec:billiard}.
Since $\bm \kappa$ is constant on local stable manifolds, 
the proof of Lemma \ref{lemma2} can be adapted to imply the following generalization (see also Lemma 3.14 in \cite{M05} for
a similar argument):

\begin{equation}
\label{eqlemmatwo'}
\sup_{\theta \in [-\pi, \pi]^d} \left\| Q_{\theta, \xi}^{C_1 \ln \xi }\right\|_{(\xi)} < 1 - \frac{C_2}{\xi^{a_1}}.
\end{equation}

Now we revisit the tower $(\bar \Delta, \bar F)$. 
Recall that a separation time $s$ was defined in (E4). Let
\begin{equation}
\| f\|_{\mathbb B} = \| f \|_{\infty} + \sup \{ C: \forall x, y \in \bar \Delta: | f(x) -  f(y)| \leq C \varkappa^{s(x,y)} \}\, .\label{normLip}
\end{equation}
Let us denote by $\bar{\bm P}$ the Perron-Frobenius operator associated with 
$ \bar{F}$ and let 
$\bar{\bm P}_{\theta,\xi}$ be defined by 
$\bar{\bm P}_{\theta,\xi}(f):=\bar{\bm P}\left( e^{i \theta \cdot \bar {\bm \kappa} + i \xi \bar {\bm \tau}}f\right).$
We conclude this section by
\begin{lemma}
\label{DMboundlemma}
There are constants $C_3, \alpha_2$ and $\delta$ so that 
\begin{equation}
\label{DMboundbilliard}
\sup_{\theta\in{[-\pi,\pi]^d}}\Vert \bar{\bm P}^n_{\theta,\xi}\Vert_{\mathcal L(\mathbb B,L^1)} \leq C_3 |\xi|^{\alpha_2} e^{-n \delta |\xi|^{-\alpha_2}}\, .
\end{equation}
\end{lemma}

\begin{proof}This lemma is proved by operator renewal theory. The proof is very similar to Section 4 in \cite{M05}, based on our Lemma \ref{lemma2}
(but is easier as we only consider purely imaginary $i \xi$). We do not repeat the proof here.
\end{proof}


\subsection{Proof of Theorem \ref{THEOBILL}}
\label{sec:proofthmbill}

Let $\mathcal S_0=\partial \bm M=\{(q,v)\in \bm M : \vec n_q.v=0\}$
be the singularity set, i.e. the collection of points in the phase space corresponding to grazing collisions.

The transformation $\bm T$ defines a $C^1$ diffeomorphism from
$\bm M\setminus (\mathcal S_0\cup \bm T^{-1}\mathcal S_0)$
to $\bm T \setminus (\mathcal S_0\cup \bm T\mathcal S_0)$.

Moreover there exist $C_0>0$ and $\theta_0\in(0,1)$ such that the diameter of
every connected component of
$\bm M \setminus\bigcup_{j=-n}^{n}{\bm T}^{-j}\mathcal S_0$ is less than
$C_0\theta_0^n$.
We consider now $\hat s$ is a suitable separation time on $\Delta$. The main difference between $s$ and $\hat s$ is that 
counts the steps straight up in the tower, i.e.
$\hat s((x,l), (y,l)) = \hat s((x,0),(y,0)) - l$. The exact definition of $\hat s$ is not important for us and 
can be found in \cite{Y98}.

Recall that, by construction of \cite{Y98}, for every 
$x,y\in\Delta$
in the same unstable manifold,
$\pi(x)$ and $\pi(y)$
lie in the same 
connected component of $\bm M \setminus\bigcup_{j=-\infty}^{{ \hat s}(x,y)}{\bm T}^{-j}\mathcal S_0$,
with ${ \hat s}(x,y):={ \hat s}(\Xi(x),\Xi(y))$.

We will prove that the assumptions of Theorem \ref{THMGENE2}
are satisfied with:
\begin{itemize}
\item $\Sigma = \Sigma_{\bm \kappa, \bm \tau}$
\item $K = 2 K_0$
\item $d=2$, 
\item $(M,\nu,T)=(\Delta,\nu,F)$, 
$\tau:= \hat{\bm \tau} = \bm \tau \circ \pi$,
$\kappa:= \hat{\bm \kappa} = \bm \kappa \circ \pi$,
\item
$(\bar \Delta,\bar\nu,\bar T)=(\bar\Delta,\bar\nu,\bar F)$,
$ \mathfrak p = \Xi$ and  $\bar P = \bar{\bm P}$

\item $\mathcal V$ the space of functions $f: \Delta\rightarrow\mathbb C$ such that the following quantity is finite
$$\Vert f\Vert_{\mathcal V}=\Vert f\Vert_\infty+\sup_{\gamma^u;\ x, y\in\gamma^u}
\frac{|f(x)-f(y)|}{\varkappa^{{ \hat s}(x,y)}}+
\sup_{n\ge 0,\ \gamma^s;\ x, y\in\gamma^s
}
\frac{|f(F^n(x))-f(F^n(y))|}{\varkappa^{n}}\, , $$
where $\varkappa$ is a fixed real number satisfying
\begin{equation}
\label{ChooseKappa}
\max\left({\theta_0^{1/4}},\theta_0^\eta,
 \vartheta \right) <\varkappa<1,
\end{equation}
where $\vartheta$ is defined in \eqref{DefSymbMetr}.
\item The space $\mathcal B$ is the Young space of complex-valued functions $f:\bar\Delta\rightarrow\mathbb C$ such that $\Vert f\Vert_{\mathcal B}<\infty$ with $\Vert\cdot\Vert_{\mathcal B}$ defined 
by 
\begin{equation}\| f \|_{\mathcal B} =
 \sup_l \| f|_{\bar \Delta_l}\|_{\infty} e^{-l \varepsilon_0 } +  
\sup_l  
\textrm{ess}  \sup_{x,y \in \bar \Delta_l } \frac{|f(x) - f(y) |}{\varkappa ^{{ \hat s} (x,y)}} e^{- l \varepsilon_0}\, .
\label{normYoung}
\end{equation}
 with $\varkappa$ as in
\eqref{ChooseKappa}
and a suitable $\varepsilon_0$.
\item 
The space $\mathbb B$ is the space of complex-valued bounded Lipschitz functions $f:\bar\Delta\rightarrow\mathbb C$ such that $\Vert f\Vert_{\mathbb B}<\infty$ with $\Vert\cdot\Vert_{\mathbb B}$ defined in \eqref{normLip} for the same choice of $\varkappa$.
\end{itemize}
In view of (E5),
\begin{equation}
\label{defq}
\mathcal B\hookrightarrow L^{
q_0}(\bar\nu) \text{ for some }
q_0
\in(1,+\infty)\, 
\end{equation}
provided that $\eps_0$ is small enough.

Observe that, with these notations $(\boldsymbol{\tilde \Omega}, \boldsymbol{\tilde \Phi}_t, \boldsymbol{\tilde \mu}_0)$ can be represented by the suspension semiflow $(\tilde\Phi_t)_{t\ge 0}$ (with roof function $\tau$) over the $\mathbb Z^2$-extension of $(M,\nu,T)$ by $\tau$.

We define
{
$$
\| f\|_{\mathcal B_0} = \| f \|_{\infty} + \inf \{ C: \forall x, y \in \bar \Delta: | f(x) -  f(y)| \leq C \varkappa^{\hat s(x,y)} \}
.$$
}
Observe that $\mathcal B_0\subset\mathcal B\cap\mathbb B$ and that the multiplication by an element of $\mathcal B_0$ defines a continuous linear operator on $\mathcal B$ and on $\mathbb B$.

Since $\kappa$ is constant on stable manifolds, there
exists a $\bar\nu$-centered $\mathbb Z^2$-valued bounded
function $\bar\kappa\in\mathbb B$ such that 
$\bar\kappa\circ\mathfrak p=\kappa$ (therefore $m_0=0$).

Moreover, since $\bm \tau$ is 1/2-H\"older on every connected component of 
$\bm M\setminus (\mathcal S_0\cup T_0^{-1}(\mathcal S_0))$ 
and since $\sqrt{\theta_0}\le\varkappa$, we have 
$\tau \in \mathcal V$. 

Now, on $\Delta$, we define $\chi:= \sum_{k \geq 0} \left(\tau \circ F^k-\tau\circ F^k\circ\Xi\right)$.
By construction,
\begin{equation}
\label{eq:tauhom}
\tau=\bar\tau\circ\mathfrak p+\chi-\chi\circ F, \text{ where }
\bar \tau\circ\Xi(\hat x^u,l) = \bar{\bm \tau}(\hat x^u,l) = 
\hat {\bm \tau}(\hat x^u, \hat{\bm x}^s,l)\, .
\end{equation}
Next, we claim that $\chi\in\mathcal V$ and $\bar\tau\in \mathcal B_0$.

Indeed, first,
$$\Vert\chi\Vert_\infty\le \sum_{k\ge 0} \Vert \tau\circ F^k-\tau\circ F^k\circ\Xi\Vert_\infty \le  \sum_{k\ge 0}\Vert \tau\Vert_{\mathcal V}\varkappa^{ k}<\infty\, .$$
Second,
if $x,y\in \Delta$ are on the same stable manifold, 
then $\Xi(F^n(x))=\Xi(F^n(y))$ and so, since $\tau$ is 1/2-H\"older, for every nonnegative integer $n$,
$$|\chi(F^n(x))-\chi(F^n(y))|\le\sum_{k\ge 0}\left|\tau(F^{k+n}(x))-\tau(F^{k+n}(y))\right|\le C_\tau\sum_{k\ge 0} \left(C_0\theta_0^{k+n}\right)^{\frac 12} = O(\varkappa^n). $$
Third,
if $x,y\in\Delta$ are {on the same unstable manifold}, then
$$|\tau(F^j(x))-\tau(F^j(y))|+|\tau( F^j(\Xi(x)))-\tau( F^j({\Xi(y)}))|\le 
        2 C_\tau (C_0\theta_0^{{ \hat s}(x,y)-j})^{\frac 12} $$
and
$$|\tau(F^j(x))-\tau( F^j(\Xi(x)))|+
|\tau(F^j({y}))-\tau( F^j(\Xi({y})))|\le 2 C_\tau (C_0\theta_0^{j})^{\frac 12}  . $$
So,
since $\theta_0^{\frac 14}\le\varkappa$
$$|\chi(x)-\chi(y)|\le O\left(\sum_{0\le k\le { \hat s}(x,y)/2}\varkappa^{2({ \hat s}(x,y)-
k)}+\sum_{k>  \hat s(x,y)/2}\varkappa^{2
k}\right)=O\left(\varkappa^{{ \hat s}(x,y)}\right)\, .$$
This shows that $\chi \in \mathcal V$. Then clearly
$\chi\circ F\in\mathcal V$ holds as well. Since $\tau \in \mathcal V$, \eqref{eq:tauhom}
implies $\bar\tau\circ\mathfrak p\in\mathcal V$ which in turn gives $\bar\tau\in\mathcal B_0$.

Observe that $\left\Vert e^{i\xi.\chi}\right\Vert_{\mathcal V}=O(1+|\xi|)$
and that $(\bar\tau_{m_0})^{
k}e^{-i\xi\bar \tau_{m_0}}\in\mathcal B$ for every 
$k$
and $m_0=1$.

The fact that $( \bar P_{\theta,\xi}: \bar f\mapsto \bar P(e^{i\theta\cdot\bar\kappa}e^{i\xi.\bar\tau}\bar f))_{(\theta,\xi)\in[-\pi,\pi]^d\times\mathbb R}$ satisfies \eqref{boundedOp},
\eqref{DLlambdab}, \eqref{decomp2b}, \eqref{majoexpob}, 
with $J= 3$ follows from \cite{Y98,SV04} (see also \cite{Soazmixing}). 
{ Condition \eqref{DMbound1} is proved by Lemma \ref{DMboundlemma}.}

For any $f\in\mathcal V$ and any nonnegative integer $n$, we define
${\boldsymbol{\Pi}}_nf:\bar\Delta\rightarrow\mathbb C$ by
$$\forall x\in\Delta,\quad ({\boldsymbol{\Pi}}_n f)\circ \Xi(x)
:=\mathbb E_{\nu} [f\circ F^n|{ \hat s}(\cdot,x)\ge 2n]\, .$$
Note that $\boldsymbol{\Pi}_n$ is linear and continuous from $\mathcal V$ to $\mathcal B_0$ with norm in $O\left(2\varkappa^{-2n}\right)$.
By definition of $\mathcal V$, 
if $s(x,y)\ge 2n$, then by considering $z$ in the stable {manifold} containing
$x$ and in the unstable {manifold} containing $y$, 
$F^n(z)$ is in the same unstable {manifold} as $F^n(y)$ with 
${ \hat s}(F^n(y),F^n(z))\ge n$ and so
$$|f(F^n(x))-f(F^n(y))|\le |f(F^n(x))-f(F^n(z))|+|f(F^n(z))-f(F^n(y))| \le 
\Vert f\Vert_{\mathcal V}\varkappa^n\, .
$$
Therefore we have proved that
$$\forall f\in \mathcal V,\quad\Vert f\circ  F^{n}- \boldsymbol{\Pi}_n(f)\circ\Xi\Vert_{\infty} \le 
{C_0} \Vert f\Vert_{\mathcal V}\,  \varkappa^n\, ,$$
and so \eqref{approxiA} holds for any $\vartheta\ge\varkappa$.

Recall that
$$\bar P_{\theta,\xi}^{2n}h(x)=
\sum_{z\in \bar F^{-2n}(\{x\})}e^{\alpha_{2n}(z)+i\theta.\bar\kappa_{2n}(z)
+i\xi.\bar\tau_{2n}(z)}h(z)\, ,$$
with $$\alpha_{l}:=\sum_{k=0}^{l-1}\alpha\circ\bar F^k, \quad 
\bar \kappa_{l}:=\sum_{k=0}^{l-1}\bar\kappa\circ\bar F^k, \quad\textrm{and}\quad
\bar\tau_{l}:=\sum_{k=0}^{l-1}\bar\tau\circ\bar F^k.$$
By construction of {$(\bar \Delta,\bar\nu,\bar F)$}, for every $x,y\in\bar \Delta$
{with $\hat s(x,y) \geq 1$,}
there exists a bijection $W_{2n}:\bar F^{-2n}(\{x\})\rightarrow
  \bar F^{-2n}(\{y\})$ such that ${ \hat s}(z,W_{2n}(z))\ge 2n$
and so $\boldsymbol{\Pi}_nf(z)=\boldsymbol{\Pi}_nf(W_{2n}(z))$.
Moreover, since $\alpha,\bar\kappa,\bar\tau\in\mathcal B_0$, 
 for $g \in\{ \alpha,\bar\kappa,\bar\tau\}$ and for any $x,y,z$ as above, we have
$$
|g(\bar F^k (z)) - g(\bar F^k (W_n(z))) | \leq \Vert g \Vert_{\mathcal B_0} \varkappa^{\hat s (x,y) + 2n -k}\, .
$$ 
Hence
$$
|g_n(\bar F^k (z)) - g_n(\bar F^k (W_n(z))) | \leq \Vert g \Vert_{\mathcal B_0} (1-\varkappa)^{-1}\varkappa^{\hat s (x,y) + n
-k}\, .
$$ 
We conclude that 
there exists $C_0>0$ such that,
for every 
$\theta \in [-\pi, \pi]^d$, $\xi \in \mathbb R$ and for every non-negative integer $j$,
\begin{gather*}
\left\Vert \frac{\partial^j}{\partial (\theta,\xi)^j}(\bar  P_{\theta,\xi}^{2n}(e^{-i\theta.\bar \kappa_{n}-i\xi.
\bar\tau
_n}\boldsymbol{\Pi}_n f))\right\Vert_{\mathcal B_0} \leq
\left\Vert \frac{\partial^j}{\partial (\theta,\xi)^j}
{\bar P^{2n}}
(e^{i(\theta.\bar \kappa_{n}+\xi.\bar\tau
_n) \circ \bar F^n}\boldsymbol{\Pi}_n f)\right\Vert_{\infty} 
+ \\
\sup_{\substack{x, y \in \bar \Delta, \\ \hat  s(x,y) \geq 1}} \varkappa^{- \hat s(x,y)}
\left| \frac{\partial^j}{\partial (\theta,\xi)^j}
\sum_{z\in \bar F^{-2n}(x)}
\left(e^{ 
\alpha_{2n}(z) + (i\theta\bar \kappa_{n}+i\xi \tau_n) \circ \bar F^n(z)}
-e^{
\alpha_{2n}(W_n(z)) + (i\theta\bar \kappa_{n}+i\xi \tau_n) \circ \bar F^n(W_n(z))}
\right)
\boldsymbol{\Pi}_nf(z)
\right|
\\
\le C_0 n^j
(1 + |\xi|)
\Vert f\Vert_{\infty}
\end{gather*}
and
\begin{eqnarray*}
\left\Vert   \frac{\partial^j}{\partial (\theta,\xi)^j} (\boldsymbol{\Pi}_n(f)e^{i\theta\cdot\bar\kappa_{n-m_0}+i\xi.\bar\tau_n})\right\Vert_{\mathcal B'} &\le&\left\Vert   \frac{\partial^j}{\partial (\theta,\xi)^j} (\boldsymbol{\Pi}_n(f)e^{i\theta\cdot\bar\kappa_{n-m_0}+i\xi.\bar\tau_n})\right\Vert_{L^p(\bar\mu)}\\
  &\le&\left\Vert  \left( \frac{\partial^j}{\partial (\theta,\xi)^j} (\boldsymbol{\Pi}_n(f)e^{i\theta\cdot\bar\kappa_{n-m_0}+i\xi.\bar\tau_n})\right)\right\Vert_{\infty}\\ 
&\le& C_0n^j\Vert f\Vert_{\infty}\, ,
\end{eqnarray*}
where we used that $\bar\kappa$ and $\bar \tau$ are uniformly bounded
and $p$ is such that $\frac 1{
q_0}+\frac 1p=1$ with 
$q_0$
defined in \eqref{defq}.
Therefore we have proved 
\eqref{approxiB}, \eqref{approxiC} and \eqref{approxiD}.
We define $f$ and $g$ as follows:
$f(x,\ell,s)=\mathfrak f(q
+\ell
+s\vec v,\vec v)$ and similarly
$g(x,\ell,s)=\mathfrak g(q
+\ell
+s\vec v,\vec v)$  if $\pi(x)=(q,\vec v)$.
Note that $(q+\ell+s\vec v,\vec v)=\boldsymbol{\widetilde\Phi}_s(q+\ell,\vec v)$
for $s\in[0,\tau(q,\vec v))$.
Let $(\mathfrak h,h)=(\mathfrak f,f)$ or $(\mathfrak g,g)$. 
We define
$$ h_\ell(x,s):=\chi_0(s)\mathfrak h\left(\boldsymbol{\widetilde\Phi}_s(q+\ell,\vec v)\right)(1-\chi_0(s-\tau(x)))\, ,$$
with $\chi_0:\mathbb R\rightarrow[0,1]$ a fixed increasing $C^\infty$ function such that $\chi_0(u)=0$ if $u\le -\frac{\min\tau}{10}$ and 
$\chi_0(u)=1$ if $u\ge 0$.\\
Note that $h_\ell(x,\cdot)$ have support in $\left[-\frac{\min\tau}{10},\tau(x)\right]$, 
coincide with $h(x,\ell,\cdot)$ in $[0,\tau(x)-\frac{\min\tau}{10}]$,  and satisfy \eqref{decomp_h}.
Let $u\in\mathbb R$ be fixed. Then
$\DS \Vert h_\ell(\cdot,u)\Vert_\infty\le\sup_{|\ell'-\ell|\le\max\tau}\left\Vert\mathfrak h\mathbf{1}_{\mathcal C_{\ell'}}\right\Vert_{\infty}$.
Furthermore, 
since $\tau\in\mathcal V$,  $\theta_0^\eta<\varkappa$, and 
$\mathfrak h\circ\boldsymbol{\widetilde\Phi}_s$ is uniformly $\eta$-H\"older continuous for 
$s\in[-\frac{\min\tau}{10},\max\tau]$, we obtain that
there exists a uniform constant $\widetilde C>0$ such that 
\begin{equation}\label{MajohV}
\Vert h_\ell(\cdot,u)\Vert_{\mathcal V}\le
       \widetilde C \sup_{|\ell'-\ell|\le\max\tau}\left\Vert\mathfrak h\right\Vert_{\mathcal H_{\mathcal C_{\ell'}^\eta}} \, .
\end{equation}
Thus, \eqref{CCC00bisa}
and \eqref{CCC00a}
follow directly from \eqref{HYPO2}. Recall that
\begin{equation}
\label{Der-Int}
\frac{\partial^k}{\partial\xi^k}\left(e^{-i\xi.\chi}\hat h_\ell(x,\xi)\right)=\sum_{m=0}^k\frac{k!}{m!\, (k-m)!}(-i\chi)^me^{-i\xi\chi}\int_{(-\frac{\min\tau}{10},\tau(x))}(is)^{k-m}e^{i\xi s}h_\ell(x,s)\, ds\, .
\end{equation}

Next, to prove \eqref{CCC0a} it suffices to show that
\begin{equation}\label{CCC0av2}
\sum_{\ell\in\mathbb Z^d}\left(\left\Vert  \frac{\partial^k}{\partial \xi^k}
    \left(e^{-i\xi.\chi}\hat f_\ell(\cdot,\xi)\right)\right\Vert_\mathcal V 
    +\left\Vert  \frac{\partial^k}{\partial \xi^k}
    \left(e^{-i\xi.\chi}\hat g_\ell(\cdot,\xi)\right)\right\Vert_\mathcal V \right)<C (1 + |\xi|)\, .
\end{equation}
Observe that $\Vert e^{-i\xi\chi}\Vert_{\mathcal V}=O(1+|\xi|)$ and the integral in \eqref{Der-Int}
is uniformly bounded by $2\max\tau\Vert h_\ell\Vert_{\infty}$. Furthermore, for $x,y\in \gamma^u$ such that $\hat s(x,y)\ge n$ (resp. for $x,y\in F^n(\gamma^s)$) and such that $\tau(x)\le\tau(y)$, we have
\begin{align*}
&\ \ \left\vert\int_{(-\frac{\min\tau}{10},\tau(x))}...\, h_\ell(x,s)\, ds-
\int_{(-\frac{\min\tau}{10},\tau(y))}...\, h_\ell(y,s)\, ds\right\vert\\
&\le \int_{(-\frac{\min\tau}{10},\tau(x))}\left\vert ...\right\vert\, \left\vert h_\ell(x,s)-h_\ell(y,s)\right\vert\, ds+\int_{\tau(x)}^{\tau(y)}\left\vert ...\right\vert\, \left\vert h_\ell(y,s)\right\vert\, ds\\
&\le \int_{(-\frac{\min\tau}{10},\tau(x))}C\, \left\Vert h_\ell(\cdot,s)\right\Vert_{\mathcal V}\varkappa^n \, ds+\left\Vert \tau\right\Vert_{\mathcal V}\varkappa ^nC\, \left\Vert h_\ell(\cdot,s)\right\Vert_\infty\, ds\, .
\end{align*}
Now \eqref{CCC0av2} follows from \eqref{MajohV} and \eqref{HYPO2}.

Assume next that $\mathfrak h$ satisfies \eqref{HYPO1}, then the functions $h_\ell(x,\cdot)$ are $C^\infty$ and there exists a uniform constant $\widetilde C_0>0$ such that 
$$\forall N\in\mathbb N,\quad \left\Vert\frac{\partial^N}{\partial s^N}
          h_\ell(\cdot,s)\right\Vert_{\mathcal V} \le \widetilde C_0\sup_{m=0,...,N}
        \sup_{|\ell'-\ell|\le\max\tau}\left\Vert\frac{\partial^m}{\partial s^m}\left(\mathfrak h\circ\boldsymbol{\widetilde\Phi}_s\right)_{|s=0} \right\Vert_{\mathcal H^\eta_{\mathcal C_{\ell'}}}\, .$$
Moreover, since $h_\ell$ is $C^\infty$ with compact support, by classical integration by parts, we have
$$\forall N\in\mathbb N,\quad \hat h_\ell(x,\xi)= 
(-i)
^N\xi^{-N}\int_{\mathbb R}e^{i\xi\, s}\frac{\partial^N}{\partial s^N}
          h_\ell(\cdot,s)\, ds$$
Therefore, since $\chi\in\mathcal V$, we have proved that, if $\mathfrak h$ satisfies \eqref{HYPO1}, we have
\begin{equation}
\forall\gamma>0,\quad \sum_{\ell}\Vert e^{-i\xi\,\chi}\hat h_\ell(\cdot,-\xi)\Vert_{\mathcal V} =O(|\xi|^{-\gamma})\, ,
\end{equation}
which, combined with \eqref{CCC0av2} implies \eqref{CCCa}.


\subsection{Identifying $\mathfrak C_0$}
\label{Sec:const}

Recall the notations $\Sigma_{\bm \kappa, \bm \tau}, \Sigma_{\bm \kappa}$ from Section \ref{sec:billiard} 
and that here $d=2$.

Let us set $\sigma:=\sqrt{\det \Sigma_{\bm \kappa, \bm \tau}/\det\Sigma_{\bm \kappa}}$.
Observe that $\Psi_{\Sigma_{\bm \kappa, \bm \tau}}(
0,0
,u)=\frac{e^{-\frac{u^2}{2\, \sigma^2}}}{(2\pi)^{\frac{
3}2}\sqrt{\det\Sigma_{\bm \kappa, \bm \tau}}}$.

Now the leading term of $C_{t}(f,g)$ 
can be obtained by taking $m=j=k=r=q=0$ in \eqref{Formula2}:
\begin{eqnarray}
\lim_{t \rightarrow \infty} t C_t(f,g) &=&
{\boldsymbol{\nu}(\boldsymbol{\tau}) \tilde C_0(f,g)}
\\
 &=&
(\boldsymbol{\nu}(\boldsymbol{\tau}))^{
\frac 12}
\int_{\mathbb R}\psi\left(0,0,s\sqrt{\boldsymbol{\nu}(\boldsymbol{\tau})}\right)
\, ds\sum_{\ell,\ell'\in\mathbb Z^2}\int_{\mathbb R^2}
{
\mathcal B}_0(f_\ell(\cdot,u),g_{\ell'}(\cdot,v))\, dudv\nonumber\\
&=&\frac {
\sigma}{2\pi
\sqrt{\det \Sigma_{\bm \kappa, \bm \tau}}}\tilde\mu(f)\tilde\mu(g)
=\frac {
1
}{2\pi
\sqrt{\det \Sigma_{\bm \kappa}}}\tilde\mu(f)\tilde\mu(g)
\nonumber
\end{eqnarray}
where we used 
${\mathcal B}_0(u,v)= {\nu}(u) {\nu}(v)$ (see \eqref{B0-Mixing}).

Recalling that the left hand side of \eqref{exp2_bold} is an integral with respect to $\boldsymbol{\tilde \mu}_0$ as opposed to $C_t(f,g)$ which is an integral with respect to
$\tilde \mu$ and using $\tilde \mu = \bm \nu (\bm \tau) \boldsymbol{\tilde \mu}_0$, 
we obtain \eqref{deffc}.

\section{Geodesic flows}
\label{sec:geo}

Let $Q$ be a compact 
Riemannian manifold with strictly negative curvature and $\tilde Q$ be a cover of $Q$
with automorphism group $\mathbb Z^d.$ Then $\tilde Q$ can be identified with $Q \times \mathbb Z^{d}.$ 

The unit tangent bundle of $\tilde Q$ is denoted by $\tilde \Omega$ and unit tangent bundle of $ Q$ is denoted by 
$\Omega$.

The phase space of the geodesic flow $\tilde \Phi$ on $\tilde Q$ is $\tilde \Omega$ and likewise, the 
phase space of the geodesic flow $\Phi$ on $ Q$ is $\Omega$. Thus
$\tilde\Omega$ is a $\mathbb{Z}^d$ cover of $\Omega$ and we denote by
Let $\fp$ the covering map.
Geodesic flows are
Anosov flows 
and can be represented as a suspension flows over a Poincar\'e section $M$
such that $T: M \rightarrow M$, the first return map to $M$ is Markov
(see \cite{B73} and \cite{BR75}). 
Thus $M$ is a union of rectangles
$M=\cup_{k=1}^K \Delta_k$ where $\Delta_k$ have product structure
$\Delta_k=[\Delta_k^u\times \Delta_k^s]$ where $\Delta_k^u$ are $u$-sets
and $\Delta_k^s$ are $s$-sets and $[\cdot, \cdot]$ is defined by \eqref{LocProd}.

Let $\tau$ be the first return to $M$. Choose a copy $\tilde{M}\subset \tilde\Omega$ such that
$\fp(\tilde{M})=M$ and $\fp: \tilde{M}\to M$ is one-to-one.
As for billiards, we define $\mathcal C_\ell$ as the set of points in that
$\tilde \Omega$ such that the last visit
to the Poincar\'e section was in $\tilde{M} \times \{ \ell \}$ for $\ell \in \mathbb Z^d$. We denote by 
$\tilde \mu$ the Liouville
measure.

Now we have the following analogue of Theorem \ref{THEOBILL}

\begin{theorem}\label{THEOGEO}
Let $\mathfrak f,\mathfrak g: {\tilde \Omega}  \rightarrow\mathbb R$
be  two $\eta$-H\"older continuous functions 
with at least one of them being smooth in the flow direction. 
Assume moreover that 
there exists an integer $K_0\ge 1$ such that \eqref{HYPO2} holds.
Then there are real numbers
$\mathfrak C_0(\mathfrak f, \mathfrak g), \mathfrak C_1(\mathfrak f, \mathfrak g),...,\mathfrak C_{K_0}(\mathfrak f,\mathfrak g)$ so that
we have 
\begin{equation}
\label{expdgeo_bold}
\int_{{\tilde \Omega}} \mathfrak f \, \mathfrak g\circ {\tilde  \Phi}_t d{{\tilde \mu}_0} =
\sum_{k=0}^{K_0}\mathfrak C_k(\mathfrak f,\mathfrak g)t^{-\frac{d}{2}-k}
+o\left(t^{-\frac{d}{2}-K_0}\right)\, ,
\end{equation}
as $t\rightarrow +\infty$.
Furthermore, 
$\mathfrak C_0(\mathfrak f,\mathfrak g) = \mathfrak c_0\int_{{\tilde \Omega}} \mathfrak f d{{\tilde \mu}_0} 
\int_{{\tilde \Omega}} \mathfrak g d{{\tilde \mu}_0} $ 
and 
the coefficients $\mathfrak C_k$, as functionals over pairs of admissible functions, are bilinear.
\end{theorem}

\begin{proof}
The proof of Theorem \ref{THEOGEO} is a simplified version of that of Theorem \ref{THEOBILL}.
Namely, we still apply the abstract Theorem \ref{THMGENE2}  
to an appropriate symbolic system.
This system is now a subshift of finite type that 
is constructed using a Markov partition $\{\Delta_k\}.$
By mixing and by the Perron-Frobenius theorem, there exists $r$ so that for any $i,j=1,...,K$,
$T^r(\Delta_i)$ and $\Delta_j$ have a non empty intersection.
We define the spaces  $\mathcal{V}, \mathcal{B}$, and $\mathbb{B}$ 
the same way as in Section \ref{sec:mixbilliard} with 
$$\Delta_0=M\quad\textrm{and}\quad \bar\Delta_0=\bigcup_{k=1}^K \Delta_k^u. $$
and with constant height $r$. Consequently,
the norms $\| . \|_{\mathbb B}$ and $\| . \|_{\mathcal B}$ are equivalent.
The assumptions of Theorem \ref{THMGENE2}   
are verified similarly to Section \ref{sec:mixbilliard} with 
additional simplifications coming from the boundedness of the return time and the equivalence
of $\mathcal{B}$ and $\mathbb{B}.$

The only point in the proof of Theorem \ref{THEOBILL} where we used 
the special properties of billiards is in the proof of Lemma 
\ref{lemma1}, where we referred to Lemma 6.40 in \cite{CM06} (which is specific to billiards). It remains to revisit this
part of the argument (again, in a simplified version as the alphabet is finite and we do not need to verify conditions
\eqref{eq:negbddr1} - \eqref{eq:posbddr}).

Geodesic flows preserve the natural contact form $\alpha$ on the unit tangent bundle 
(corresponding to the symplectic structure on the tangent bundle).
According to the results of \cite{L04} (Lemma B.6), there is some $\eps >0$ so that
for any $z \in Q$ and for any sufficiently small
unstable vector $v \in E^u(z)$ and stable vector $w \in E^s(z)$ with the notation $x = \exp_z(v)$, 
$y = \exp_z(w)$,
the temporal distance function $D(x,y)$
(defined as in \eqref{deftemp}) satisfies
$$
D(x,y) = d \alpha(v,w) + O(\| v\|^{\eps}\| w\|^2 + \| v\|^{2}\| w\|^{\eps} )\, .
$$
Since the contact form is non-degenerate, there is a constant $R_0$ such that 
for any $z$ and any $v \in E^u(z)$, we can find some $w \in T_zQ$
such that 
$\frac{ \|v \| \|w\|}{R_0} \leq d\alpha(v,w) \leq  R_0 \|v \| \|w\|$. 
Let us decompose $w$ into center unstable and 
stable components $w = w^{cu} + w^{s}$. By Lemma B.2 in \cite{L04}, $d\alpha(v, w^{cu}) = 0$ 
and so we can
assume $w=w^s \in E^s(z)$. We conclude that for fixed $z$, there are constants $\delta_0, R_0$,
so that for any $\delta < \delta_0$ there exist vectors
$v \in E^u(z), w \in E^s(z)$ such that $\|v \| = \|w\| = \delta$
and
$$
D(x,y) \in \left[\frac{\delta^2}{2R_0}, 2R_0 \delta^2\right].
$$
Now we can complete the proof of the analogue of Lemma \ref{lemma1} as before by choosing $\delta$ in a way that
for given $\xi$, $\delta^2 \approx \xi^{-1}$.
\end{proof}

\appendix
\section{Some facts about Taylor expansions.}
\label{AppTaylor}
\begin{lemma}\label{POLY}
Let $a$ be given by \eqref{Psi-FT} and
a $C^{K+3}$-smooth function $\tilde\lambda:[-b,b]^{d+1}\rightarrow\mathbb C$ (for some $b>0$)  satisfying \eqref{LambdaDer} for some $J\le K+3$.
Denote 
$\zeta_s=\frac{\tilde\lambda_s}{a_s},$ $M= \lfloor (K+1)/(J-2)\rfloor.$ 
Then  there are 
$A_{j,k}\in\mathcal S_j$ 
(where $j=0,...\lfloor J(K+1)/(J-2)\rfloor$, $k=1,\dots, M$),
 $K_0 \in \mathbb{N}$ (depending on $K$ and $J$)
 and a function $\eta:\mathbb R^{d+1}\rightarrow[0,+\infty)$ 
continuous at $\mathbf{0}$, satisfying $\eta(\mathbf{0})=0$ 
such that after, possibly, decreasing the value of $b$,
for every $n$ large enough, every $s\in[-b\sqrt{n},b\sqrt{n}]^{d+1}$
and every $j=J,...,K+3$,
we have
\begin{equation}\label{TOTO}
\sum_{k=1}^M {n \choose k} \sum_{j_1,...,j_k\ge J\ :\ j_1+...+j_k=j}
\frac 1{j_1!...j_k!}\left(\zeta^{(j_1)}_0\otimes ...\otimes \zeta^{(j_k)}_0\right)= \sum_{k=1}^M n^kA_{j,k}
\end{equation}
and 
\begin{equation}\label{poly}
\left|\zeta_{s/\sqrt{n}}^n
-1-\sum_{k=1}^{M}\sum_{j=kJ}^{K+1+2k}
  n^kA_{j,k}*\left(\frac s{\sqrt{n}}\right)^{\otimes j}\right|\le 
 \frac{1}{a_{s/\sqrt{2}}}
n^{-\frac{K+1}2}(1+
{|s|^{K_0}}
)\eta(s/\sqrt{n})\, .
\end{equation}
\end{lemma}
Recalling that the first $J-1$ derivatives of $\zeta$ vanish at zero, we see that in case 
$\tilde\lambda$ is $C^j$ (namely, if $j\leq K+3$), the LHS of
\eqref{TOTO} is simply equal to $\frac 1{j!} (\zeta^n)_0^{(j)}$.
\begin{proof}
Decreasing if necessary the value of $b$, we may assume that
$|\tilde\lambda_{u}|\le
a_{u/\sqrt{2.5}}\le 
 a_{u/\sqrt{2}}$
and $|\tilde\lambda_u-a_u|\le C |u|^J$
for every $u\in\mathbb R^{d+1}$ with $|u|<b$ (the existence of $b$ with these properties
follows from
our assumptions on $J$ and $\tilde \lambda$).
Applying Taylor's theorem to the function $x \mapsto x^n$ near $1$
we conclude that for every $s\in\mathbb R^{d+1}$ with $|s|<b\sqrt{n}$,
\begin{gather}
\left|\zeta_{s/\sqrt{n}}^n
- \sum_{k=0}^M{n \choose k}
    \left(\zeta\left(\frac{s}{\sqrt{n}}\right)-1\right)^k\right| \nonumber \\
\le { {n \choose M+1}
\left|\zeta\left(\frac{s}{\sqrt{n}}\right)-1\right|^{M+1}
(\max(1,|\zeta\left(\frac{s}{\sqrt{n}}\right)|))^{n-M-1} }\, . \label{lempolyeq1}
\end{gather}
Recall that 
$|\tilde \lambda_{s/\sqrt{n}}| \leq a_{s/\sqrt{1.5\, n}}$.
This together with the fact that 
$a_{s/\sqrt{1.5\, n}}/a_{s/\sqrt{n}}=(a_{s/\sqrt{3n}})^{-1}$
implies that
the RHS of \eqref{lempolyeq1} is bounded by
$$
 n^{M+1} \left|\zeta(s/\sqrt{n})-1\right|^{M+1}  (a_{s/\sqrt{
3n}})^{-(n-M-1)} = 
 n^{M+1} \left|\tilde\lambda(s/\sqrt{n})-a(s/\sqrt{n})\right|^{M+1} (a_{s/\sqrt{
3n}})^{-n-M-1}\, .$$
Next, we use the identity $(a_{s/\sqrt{
3n}})^{n}=a_{s/\sqrt{
3}}$ and the inequality $|\tilde\lambda_u-a_u|\le C |u|^J$
to conclude that the last displayed expression is bounded by
$$
C_M n^{M+1} (a_{s/\sqrt{2}})^{-1}   \left((s/\sqrt{n})^{J(M+1)}\right)\, ,
$$
for every $s$, for every $n$ large enough since $(a_{s/\sqrt{3n}})^{-n-M-1}=
   \left(a_{s\sqrt{(1+\frac {M+1}n)/3}}\right)^{-1}\le (a_{s/\sqrt{2}})^{-1}$
for every $n$ large enough.
Now observe that by definition $(2-J)(M+1)<-K-1$ and so $(2-J)(M+1)\le -K-2$. Thus the last display, and hence \eqref{lempolyeq1} is bounded by
\begin{equation}
\label{lemmapolyerror1}
 C_M(a_{s/\sqrt{2}})^{-1} n^{-\frac{K+2}2}s^{J(M+1)}.
 \end{equation}
Clearly, \eqref{lemmapolyerror1} can be included in the RHS of \eqref{poly}.
Thus it remains to compute the
sum in the
LHS of \eqref{lempolyeq1}. 

To do so, we fix some $k = 1,...,M$. Let $L = K+1+2k-J(k-1)$. Using the elementary estimate
$|a^k-b^k|\le k\max(|a|,|b|)^{k-1}|a-b|$, we find
\begin{eqnarray}
&\ &{n \choose k}\left|\left(\zeta(s/\sqrt{n})-1\right)^k-
       \left(\sum_{j=J}^{ L}\frac 1{j!}
      \zeta^{(j)}_0*(s/\sqrt{n})^{\otimes j}\right)^k\right|  \label{lemma:polyeq2.1}\\
&\le&
{
n^k k\max\left( |\zeta(s/\sqrt{n})-1|,\left|\sum_{j=J}^{ L}\frac 1{j!}
      \zeta^{(j)}_0*(s/\sqrt{n})^{\otimes j}\right|
\right)^{k-1}}\label{lemma:polyeq2.2} \\
&\ &
\left|\zeta(s/\sqrt{n})-1-
       \sum_{j=J}^{ L}\frac 1{j!}
      \zeta^{(j)}_0*(s/\sqrt{n})^{\otimes j}\right|\, . \label{lemma:polyeq2.3}
\end{eqnarray}
Next by our choice of $L$
$$
L=K+1+(2-J)k+J\le K+1+(2-J)+J=K+3.
$$
Recalling that $\tilde\lambda/a$ is $C^{K+3}$ smooth and its first $J-1$ derivatives at zero vanish, Taylor's theorem implies that \eqref{lemma:polyeq2.3}
is bounded by $(s/\sqrt{n})^{ L}   \eta_0(s/\sqrt{n})$, where $\eta_0(0) = 0$ and $\eta$ is continuous at $0$.
On the other hand, \eqref{lemma:polyeq2.2} is bounded by $n^k k \left(s/\sqrt{n}\right)^{J(k-1)}$.
We conclude that
\eqref{lemma:polyeq2.1} is bounded by 
\begin{equation}
\label{lemmapolyerror2}
n^{-\frac {K+1}2} s^{ K+1+2k}   \eta_1(s/\sqrt{n})\, ,
\end{equation}
where $\eta_1 = k \eta_0$. Since $a_{s/\sqrt 2}$ is bounded from above, \eqref{lemmapolyerror2} can be included in the RHS of \eqref{poly}.
So we have approximated $\zeta_{s/\sqrt{n}}^n$ by
\begin{eqnarray*}
&\ &1+\sum_{k=1}^{M}{n \choose k}\left(\sum_{j=J}^{ L}\frac 1{j!}
      \zeta^{(j)}_0*(s/\sqrt{n})^{\otimes j}\right)^k\\
&=&1+\sum_{k=1}^{M}{n \choose k}\sum_{j_1,...,j_k=J}^{L}
\frac 1{j_1!...j_k!}\left(\zeta^{(j_1)}_0\otimes ...\otimes \zeta^{(j_k)}_0\right)*(s/\sqrt{n})^{\otimes(j_1+...+j_k)}\\
&=&1+\sum_{k=1}^{M}{n \choose k}\sum_{j=kJ}^{K+1 + 2k}\sum_{j_1,...,j_k\ge J\ :\ j_1+...+j_k=j}
\frac 1{j_1!...j_k!}\left(\zeta^{(j_1)}_0\otimes ...\otimes \zeta^{(j_k)}_0\right)\\
&\ &\ \ \ *(s/\sqrt{n})^{\otimes j}+O\left(n^{-\frac{K+2}2}s^{K+1
+
2k+1}\right)
\end{eqnarray*}
uniformly on $s\in[-b\sqrt{n},b\sqrt{n}]^{d+1}$. Note that the last step above uses the observation that 
if $j_1,...,j_k \geq J$ and $j_1+ \dots+ j_k \leq K+1 + 2k$, then necessarily $j_l \leq L$ for all $l$. Again, the last error term
can be included in the right hand side of \eqref{poly} as $a_{s/\sqrt 2}$ is bounded from above.

Finally, observe that
$${n \choose k}\sum_{j_1,...,j_k\ge J\ :\ j_1+...+j_k=j}
\frac 1{j_1!...j_k!}\left(\zeta^{(j_1)}_0\otimes ...\otimes \zeta^{(j_k)}_0\right)$$
is a polynomial of degree $k$ in $n$ with values in $\mathcal S_j$. This ensures the existence of $A_{j,k}$.
 \end{proof}

\begin{lemma}
 If $H :\mathbb R \rightarrow \mathbb R$ is in the Schwartz space (i.e. $x^a H^{(b)}(x)$ is bounded for any 
positive integers $a$ and $b$), then 
for any $L \in \mathbb N$ there is some constant $c_{H,L}$ such that 
\begin{equation}
\label{Euler}
\forall t\in\mathbb R,\\ \forall \eta>0,\quad \left| \sum_{k \in \mathbb Z} \eta H (t+k\eta) - \int_{-\infty}^{\infty} H(x) dx \right| < c_{H,L} \eta^{L}.
\end{equation}
\end{lemma}

\begin{proof}
We can assume without loss of generality that $t \in [0,1)$.
Given $L, t$ and $\eta$, we choose $A_{L}$ and $B_{L}$ 
so that the above sum for $k \notin [A_{L}/\eta,B_{L} /\eta]$
and the above integral as well as the first $L$ derivatives of $H$ for 
$x \notin  (A_{L},B_{L})$ are less than $\eta^{L}$. Such 
$A_{L}$ and $B_{L}$ exist since $H$ is in the Schwartz space.
Now Euler's summation formula (e.g. Theorem 4 in \cite{A99} with the notation
$f(x) = \eta H (t + x \eta - A_L)$, $m = L$) implies that
\begin{align*}
 \sum_{k  = - A_L/\eta}^{B_L/\eta} \eta H (t+k\eta) - \int_{A_L}^{B_L} H(x) dx  
&=  \frac{1}{(2L + 1)!} \int_{A_L}^{B_L} \mathcal{P}_{2L +1}(x/\eta) H^{(2L+1)}(x) dx \eta^{2L +1}  \\
& + \sum_{r=1}^L \frac{\mathcal{B}_{2r}}{(2r)!} \left[ H^{(2r-1)} (B_L) - H^{(2r-1)} (A_L) \right] \eta^{2r} \\
&+ \frac12 \eta [H(B_L) - H(A_L)],
\end{align*}
where $\mathcal{P}_k(x)$ are the periodic Bernoulli polynomials and $\mathcal{B}_k$ 
are Bernoulli numbers. 
Now \eqref{Euler} follows from the choice of $A_L, B_L$.
\end{proof}
Observe that \eqref{Euler}
and the fact that $H$ is in the Schwartz space imply 
\begin{equation}\label{approxisumint}
\forall K>0,\quad \forall \varepsilon>0,\quad
\sum_{n= t/\nu(\tau)-t^{\frac 12+\eps}}^{t/\nu(\tau)+t^{\frac 12+\eps}}  H \left( \frac {t-n\nu(\tau)}{\sqrt{t}} \right) = 
\frac{\sqrt t}{{\nu(\tau)}} \int_{\mathbb R} H(x) dx + O (t^{-K})
\end{equation}
(clearly, the constant in "$O$" depends on $K$ and $\varepsilon$).

\begin{lemma}\label{sumint}
For every $\gamma{\in\mathbb R}$ and $Q \in \mathbb Z_+$,
$$ \sum_{n= t_- }^{t_+} 
n^{\gamma} \Psi^{(\alpha)} \left( 0, \frac{t-n\nu(\tau)}{\sqrt n}\right) 
$$
\begin{equation}
\label{eq:sumint0}
=\left(\frac{t}{\nu(\tau)}\right)^\gamma
\sum_{q=0}^{Q}\frac 1{q!}\frac{t^{- \frac {q-1}2}}{\nu(\tau)} \int_{\mathbb R}{\partial^{q}_2}h_{\alpha, \gamma}\left(s,  1\right)(-s)^q\, ds+O\left(t^{\gamma-\frac {Q}2}\right)
\end{equation}
where $h_{\alpha, \gamma}$ is defined by \eqref{defh}
$\partial^q_2$ 
denotes the derivative of order $q$ with respect to the second variable.
\end{lemma} 

\begin{proof}
For ease of notation, we prove the lemma coordinate-wise, i.e. we replace  $\Psi^{(\alpha)}(s)$ by $\frac{\partial^{\alpha}}{\partial s_{j_1} \dots \partial s_{j_\alpha}} \Psi (s)$.

Observe that due to the rapid decay of $\Psi^{(m+j+r)}(0,\cdot)$, we can replace
$\sum_{n=t_-}^{t_+}$ by
$\sum_{n=t/\nu(\tau)-t^{\frac 12+\eps}}^{t/\nu(\tau)+t^{\frac 12+\eps}}$, for any $\eps>0$
(here, we can choose e.g., $\eps = 1/4$).

Next, observe that by the definition \eqref{defh},
$$
 \left(\frac n{t/\nu(\tau)}\right)^\gamma \Psi^{(\alpha)}\left(0,
\frac{t-n\nu(\tau)}{\sqrt n}\right) = h_{\alpha, \gamma}
\left(\frac{t-n\nu(\tau)}{\sqrt t},\frac {n\nu(\tau)}{t}\right).
$$
Thus it remains to estimate the sum
\begin{equation}
\label{eqsumint1}
\sum_{n= t/\nu(\tau)-t^{\frac 12+\eps} }^{t/\nu(\tau)+t^{\frac 12+\eps}} 
h_{\alpha, \gamma} \left(\frac{t-n\nu(\tau)}{\sqrt t},\frac {n\nu(\tau)}{t}\right).
\end{equation}
Using Taylor expansion, we can rewrite \eqref{eqsumint1} as
\begin{equation}
\label{eqsumint2}
\left[ \sum_{n= t/\nu(\tau)-t^{\frac 12+\eps} }^{t/\nu(\tau)+t^{\frac 12+\eps}}  \sum_{q=0}^{Q}\frac 1{q!}{\partial^{q}_2}h_{\alpha, \gamma}\left(\frac{t-n\nu(\tau)}{\sqrt {t}},  1\right)\left(-\frac{t-n\nu(\tau )}t\right)^q 
\right] +O\left(t^{-\frac {Q}2}\right)\, .
\end{equation}
Indeed, we control the error term using the estimate
\begin{eqnarray*}
\sum_{n=t/\nu(\tau)-t^{\frac 12+\varepsilon}}^{t/\nu(\tau)+t^{\frac 12+\varepsilon}}\sup_{|y-1|<1/2}
   \left|{\partial^{Q+1}_2}h_{\alpha, \gamma}\left(\frac{t-n\nu(\tau)}{\sqrt {t}},  y\right)\right|\, \left|\frac{t-n\nu(\tau)}t\right|^{Q+1}
&=&O\left(t^{-\frac Q2}\right)\, ,
\end{eqnarray*}
which can be derived similarly to \eqref{approxisumint}.
Performing summation over $n$ in \eqref{eqsumint2}, using \eqref{approxisumint},
we obtain that \eqref{eqsumint1} (and thus the left hand side of \eqref{eq:sumint0}) equals to
$$
\sum_{q=0}^{Q}\frac 1{q!}\frac{t^{- \frac {q-1}2}}{\nu(\tau)} \int_{\mathbb R}{\partial^{q}_2}
h_{\alpha, \gamma} \left(s,  1\right)(-s)^q\, ds+O\left(t^{-\frac {Q}2}\right)\, .
$$
This completes the proof of the lemma.
\end{proof}

\begin{lemma}\label{vanishint}
Let $b,q$ be non-negative integers.
The function $s\mapsto\partial_2^qh_{b,\gamma}(s,1)(-s)^q$ is even if $b + q$ is even (and is odd 
if $b + q$ is odd).
\end{lemma}
\begin{proof}
The lemma follows since if $P(x)$ is a polynomial with odd (even, resp.) 
leading term, then
$
\frac{d}{dx} (P(x)e^{cx^2}) = Q(x)e^{cx^2}
$
where $Q(x)$ is a polynomial  with even (odd, resp.) leading term.
\end{proof}

\section{Correlation functions of coboundaries}
\label{sec:coboundary}

\begin{lemma}
 Let $\mathbf{G}^t:\mathbf{M}\to \mathbf{M}$ be a flow preserving a measure $\boldsymbol{\mu}$
(finite or infinite). Let $f, f', g: \mathbf{M}\to\mathbf{M}$ be bounded
integrable
observables such that
$f'(x)=\frac{d}{dt}|_{t=0} f(\mathbf{G}^t x).$
Denote 
$$ C_t=\int_{\mathbf{M}} f \left(g\circ \mathbf{G}^t \right) d\mathbf{\mu}, \quad
C_t'=\int_{\mathbf{M}} f' \left( g\circ \mathbf{G}^t \right) d\mathbf{\mu}. $$

Assume that 
there exist real numbers $\alpha>0,$
$c_0,...,c_{ K-1}, c'_0,...,c'_{ K}$ satisfying:
\begin{equation}
\label{coboundary1}
C_t=t^{-\alpha}\left(\sum_{k=0}^{ K-1}c_k t^{-k}+o\left(t^{-(K-1)}\right)\right)\, \quad
\mbox{and}
\quad C_t'=t^{-\alpha}\left(\sum_{k=0}^{ K} c'_k t^{-k}+o\left(t^{
-K}\right)\right)\, .
\end{equation}
Then $c'_0=0$ and $c'_k=-c_{k-1}(\alpha+k-1)$ for every $k=1,...,K-1$.

In particular if $K=1$ and $c_0\ne 0$, then $c'_0=0$ and 
\begin{equation} \label{coboundary2}
C_t(f',g)\sim -c_0\alpha t^{-\alpha-1}
\end{equation}
\end{lemma}

We note that the fact that the rate of mixing for coboundaries is faster than for general observables
is used, for example, in \cite{FFK16, FU12}.

\begin{proof}
By  integration by parts
\begin{eqnarray*}
C_t'&=&\int_{\mathbf{M}} f'\left(g\circ \mathbf{G}^t\right)\, d\boldsymbol{\mu}=
-\int_{\mathbf{M}}f\left(g'\circ \mathbf{G}^t\right)\, d\boldsymbol{\mu} \\
&=&-\int_{\mathbf{M}}f.\frac\partial{\partial t}\left(g\circ \mathbf{G}^t\right)\, d\boldsymbol{\mu}
=-\frac\partial{\partial t}\int_{\mathbf{M}}f\left(g\circ\mathbf{G}^t\right) \, d\boldsymbol{\mu}=-\frac\partial{\partial t}C_t.
\end{eqnarray*}
Since $\DS \lim_{t\to+\infty} C_t=0$ 
\begin{eqnarray*}
C_t&=&\int_t^{+\infty}C_s' \, ds
=\int_{t}^{+\infty} \sum_{k=0}^{
K}c'_k s^{-\alpha-k}+o(s^{-\alpha
-K
})\, ds\, .
\end{eqnarray*}
It follows that $c'_k=0$ if $\alpha+k
 \le 1
$ and 
$$C_t= \sum_{k=0}^{
K
}\frac{c'_k}{-\alpha-k+1} t^{-\alpha+1-k}+
o\left(t^{-\alpha-K+1}\right)\, .$$
The lemma follows by comparing the above expansion with the first equation in \eqref{coboundary1}.
\end{proof}

\section*{Acknowledgements}
This research started while PN was affiliated with the University of Maryland.
FP thanks the University of Maryland, where this work was started, for its hospitality.
The research of DD was partially sponsored by NSF DMS 1665046, and the research of PN was partially sponsored by NSF DMS 1800811.
 PN and FP thank the hospitality of CIRM, Luminy and 
Centro di Ricerca Matematica Ennio De Giorgi, Pisa, where part of this work was done. FP thanks the IUF for its important support.


\begin{thebibliography}{999}

\bibitem{AD01} Aaronson, J., Denker, M.,
Local limit theorems for partial sums of stationary sequences generated by Gibbs-Markov maps,
{\it Stochastics \& Dynamics} {\bf 1}  (2001) 193--237.

\bibitem{AN17} Aaronson, J., Nakada, H.,
On multiple recurrence and other properties of 'nice' infinite measure-preserving
transformations,
{\it Ergodic Theory Dynam. Systems} {\bf 37}  (2017) 1345--1368.

\bibitem{A99} Apostol, T. M., An Elementary View of Euler's Summation Formula
The American Mathematical Monthly 106 (1999) 409--418

\bibitem{B73}
Bowen, R.,
Symbolic dynamics for hyperbolic flows
{\it Amer. J. Math.} {\bf 95} (1973) 429--460.

\bibitem{BR75} Bowen, R., Ruelle, D.,
The Ergodic Theory of Axiom A Flows. {\it Invent. Math} {\bf 29} (1975) 181--202.

\bibitem{BHM05}
Bruin, H., Holland, M., Melbourne, I.,
Subexponential decay of correlations for compact group extensions of nonuniformly expanding systems. 
{\it Ergod. Th. \& Dynam. Sys.} {\bf 25} (2005)
1719-1738.

\bibitem{CM06} Chernov N., Markarian R., {\it Chaotic billiards,} 
Math. Surveys \& Monographs {\bf 127}  AMS, Providence, RI, 2006. xii+316 pp.

\bibitem{Dima98} Dolgopyat, D.,
On decay of correlations in Anosov flows. 
{\em Ann. of Math.} {\bf 147} (1998) 357--390. 


\bibitem{DN1} Dolgopyat D., N\'andori P.,
On mixing and the local central limit theorem for hyperbolic flows,
	arXiv:1710.08568, to appear in {\it Erg. Th., Dyn. Sys.}

\bibitem{DN2} Dolgopyat D., N\'andori P.,
Infinite measure renewal theorem and related results,
	{\it Bulletin LMS} {\bf 51} (2019) 145-167.

\bibitem{FFK16} Fayad B., Forni G., Kanigowski A.,
Lebesgue spectrum for area preserving flows on the two torus,
arXiv:1609.03757.

\bibitem{FL18}
Fernando, K., Liverani, C.,
Edgeworth expansions for weakly dependent random variables,
https://arxiv.org/abs/1803.07667

\bibitem{FU12}
Forni G., Ulcigrai C.,
Time-changes of horocycle flows,
{\it J. Mod. Dyn.} {\bf 6} (2012) 251--273. 

\bibitem{G11} Gou\"ezel, S., 
Correlation asymptotics from large deviations in dynamical systems with infinite measure. 
{\it Colloq. Math.} {\bf 125} (2011) 193--212.


\bibitem{G89}
Guivarc'h, Y.,
Propri\'et\'es ergodiques, en mesure infinie, 
de certains syst\`emes dynamiques fibr\'es,
{\it Ergod. Th. Dynam. Sys.} {\bf 9} (1989) 433--453.

\bibitem{GH88} Guivarc'h, Y., Hardy, J.,
Theoremes limites pour une classe de chaines de Markov et applications aux diffeomorphismes d'Anosov,
{\it AHIP} {\bf 24} (1988) 73--98.

\bibitem{I08} Iwata, Y.,
A generalized local limit theorem for mixing semi-flows,
{\it Hokkaido Math. J.} {\bf 37} (2008) 215--240.

\bibitem{KM81} 
Kubo, I., Murata, H., Perturbed billiard systems II, Bernoulli properties, 
{\it Nagoya Math. J.} {\bf 81} (1981), 1-25.




\bibitem{L04} Liverani C., On contact Anosov flows,
{\it Annals of Math.} {\bf 159} (2004) 1265--1312

\bibitem{LT16} Liverani C., Terhesiu D., Mixing for some non-uniformly hyperbolic systems,
{\it Ann. Henri Poincare} {\bf 17} (2016) 179--226. 

\bibitem{Mar}
Margulis, G., On some applications of ergodic theory to the study of manifolds on negative curvature,
{\it Func. Anal. Appl.} {\bf 3} (1969) 89--90.

\bibitem{M05} Melbourne, I., Rapid decay of correlations for nonuniformly hyperbolic flow, 
{\it Transactions of the AMS}, {\bf 359} (2007) 2421--2441.

\bibitem{MT13} Melbourne I., Terhesiu D., First and higher order uniform dual ergodic theorems for dynamical systems with infinite measure, 
{\it Israel J. Math.} {\bf 194} (2013) 793--830.

\bibitem{MT17} Melbourne I., Terhesiu D.,
Operator renewal theory for continuous time dynamical systems with finite and infinite measure,
{\it Monatsh. Math.} {\bf 182} (2017) 377--431. 

\bibitem{MT18} Melbourne I., Terhesiu D.,
Renewal theorems and mixing for non Markov flows with infinite measure, {\it preprint.}

\bibitem{OP17} Oh H., Pan W.,
		Local mixing and invariant measures for horospherical subgroups on abelian covers. 
		arXiv:1701.02772.

\bibitem{PS} Pollicott M,. Sharp R.,
Asymptotic expansions for closed orbits in homology classes,
{\it Geom. Dedicata} {\bf 87} (2001) 123--160. 

\bibitem{Soazmixing}
P\`ene, F., Mixing and decorrelation in infinite measure: the case
of the periodic Sinai billiard, {\it Annales de l'Institut Henri Poincar\'e, P\&S}
{\bf 55} (2019) 378--411.

\bibitem{SV04} Sz\'{a}sz, D., Varj\'{u}, T.,
 Local limit
theorem for the Lorentz process and its recurrence in the plane,
{\it Ergodic Theory Dynam. Systems} {\bf 24} (2004) 254--278.



\bibitem{T19}  Terhesiu D., Krickeberg mixing for Z extensions of Gibbs Markov semiflows,
preprint
https://arxiv.org/abs/1901.08648

\bibitem{Y98} Young, L-S., Statistical properties of systems with some 
hyperbolicity including certain billiards,
{\it Ann. Math.} {\bf 147} (1998) 585--650.


\end{thebibliography}
\end{document}